\documentclass{article}[12pt]
\usepackage{amsmath}
\usepackage{amssymb}
\usepackage{extarrows}
\usepackage{hyperref}
\newtheorem{theorem}{Theorem}

\newtheorem{lemma}[theorem]{Lemma}

\newtheorem{proposition}[theorem]{Proposition}
\newtheorem{remark}[theorem]{Remark}

\newenvironment{proof}[1][Proof]{\noindent\textbf{#1.} }{\ \rule{0.5em}{0.5em}}
\usepackage{amssymb}
\usepackage{graphicx}
\usepackage[dvipsnames,usenames]{color}

\input{epsf.tex}
\usepackage[affil-it]{authblk}

\author[a]{Ioannis Dimitriou \footnote{ idimit@uoi.gr}\footnote{Corresponding author.}}
\author[b]{Ivo J.B.F. Adan\footnote{i.adan@tue.nl}}
\affil[a]{\small Department of Mathematics, 
	University of Ioannina, 
	45110, Ioannina, Greece.}
\affil[b]{\small Department of Industrial Engineering and Innovation Sciences, Eindhoven University of Technology, P.O. Box 513, 5600 MB Eindhoven, the Netherlands}

\begin{document}
\title{On vector-valued functional equations with multiple recursive terms}

\maketitle
\begin{abstract}
 In this work, we study vector-valued functional equations with multiple recursive terms that arise naturally when we are dealing with vector-valued multiplicative Lindley-type recursions. We provide a detailed framework for the solution of such equations. Our theoretical results are applied in a wide range
of semi-Markovian queueing, and vector-valued autoregressive processes.
    \end{abstract}
    \vspace{2mm}
	

\section{Introduction}
The primary aim of this work is to solve systems of functional equations of the form
\begin{equation*}
    \begin{array}{rl}
        \begin{pmatrix}
            z_{1}(s)\\
            z_{2}(s)\\
            \vdots\\
            z_{N}(s)
        \end{pmatrix}= &\begin{pmatrix}
            h_{1,1}(s)&h_{2,1}(s)&\ldots&h_{N,1}(s)\\
             h_{1,2}(s)&h_{2,2}(s)&\ldots&h_{N,2}(s)\\
             \ldots&\ldots&\ldots&\ldots\\
              h_{1,N}(s)&h_{2,N}(s)&\ldots&h_{N,N}(s)\\
        \end{pmatrix}\times\begin{pmatrix}
            z_{1}(\alpha_{1}(s))\\
            z_{2}(\alpha_{2}(s))\\
            \vdots\\
            z_{N}(\alpha_{N}(s))
        \end{pmatrix}+\begin{pmatrix}
            v_{1}(s)\\
            v_{2}(s)\\
            \vdots\\
            v_{N}(s)
        \end{pmatrix},
    \end{array}
\end{equation*}
for $Re(s)\geq 0$, where $\alpha_{i}(s)$, $i=1,\ldots,N$ are commutative contraction mappings. Equivalently, in matrix notation
\begin{equation}
         \tilde{Z}(s)=H(s)\sum_{i=1}^{N}\tilde{P}^{(i)}\tilde{Z}(\alpha_{i}(s))+\tilde{V}(s),\,Re(s)\geq 0,\label{opq}
    \end{equation}
    where, $\tilde{Z}(s):=(z_{1}(s),\ldots,z_{N}(s))^{T}$, and
    \begin{displaymath}
        H(s):=\begin{pmatrix}
            h_{1,1}(s)&h_{2,1}(s)&\ldots&h_{N,1}(s)\\
             h_{1,2}(s)&h_{2,2}(s)&\ldots&h_{N,2}(s)\\
             \ldots&\ldots&\ldots&\ldots\\
              h_{1,N}(s)&h_{2,N}(s)&\ldots&h_{N,N}(s)\\
        \end{pmatrix},\,\tilde{V}(s):=\begin{pmatrix}
            v_{1}(s)\\
            v_{2}(s)\\
            \vdots\\
            v_{N}(s)
        \end{pmatrix},
    \end{displaymath}
    are known matrix/vector functions, $\tilde{P}^{(i)}:=(\tilde{P}^{(i)})_{p,q}$, $i,p,q\in E=\{1,2,\ldots,N\}$ is an $N\times N$ matrix, with the element $\tilde{P}^{(i)}_{i,i}=1$, and all the other elements $\tilde{P}^{(i)}_{p,q}=0$, $p,q\neq i$. Note that $\sum_{i=1}^{N}\tilde{P}^{(i)}=I$, that is the identity matrix. Moreover, $\tilde{P}^{(i)}$ is a projection matrix onto the $i$th coordinate, $i\in E$. Such type of functional equations arise naturally when we are dealing with the stationary behaviour of vector-valued multiplicative Lindley-type recursions of a certain type with numerous applications in applied probability, including semi-Markovian queueing models, and vector-valued reflected autoregressive processes that are discussed in detail in the following sections. A queueing example of such a model is a Markov-modulated queue where arrivals and services are dependent on a background Markov chain, and the focus is on the workload just before an arrival that adds work in the system but makes obsolete a fixed portion of the work that is already in the system. This fixed portion also depends on the state of the background process and is related to the function $\alpha_{i}(s)$, $i\in E$. 

    To help the reader with the notation, unless otherwise specified, vectors/matrices are represented by uppercase letters, possibly augmented with diacritical marks such as a tilde or a hat, whereas their elements are denoted by lowercase letters.

The corresponding transient behaviour of such processes is described by similar, although a bit more complicated, functional equations of the form
    \begin{equation}
        \tilde{Z}(r,s,\eta)=G(r,s,\eta)\sum_{i=1}^{N}\tilde{P}^{(i)}\tilde{Z}(r,\alpha_{i}(s),\eta)+K(r,s,\eta),\label{basic0}
    \end{equation}
    for $Re(s)\geq 0$, $Re(\eta)\geq 0$, $|r|<1$, and $\tilde{Z}(r,s,\eta):=(z_{1}(r,s,\eta),z_{2}(r,s,\eta),\ldots,z_{N}(r,s,\eta))^{T}$. Moreover, 
    \begin{displaymath}
        G(r,s,\eta)=r\begin{pmatrix}
            g_{1,1}(s,\eta)&g_{2,1}(s,\eta)&\ldots&g_{N,1}(s,\eta)\\
             g_{1,2}(s,\eta)&g_{2,2}(s,\eta)&\ldots&g_{N,2}(s,\eta)\\
             \ldots&\ldots&\ldots&\ldots\\
              g_{1,N}(s,\eta)&g_{2,N}(s,\eta)&\ldots&g_{N,N}(s,\eta)\\
        \end{pmatrix},\,\,K(r,s,\eta)=\begin{pmatrix}
        k_{1}(r,s,\eta)\\
            k_{2}(r,s,\eta)\\
            \vdots\\
            k_{N}(r,s,\eta)
        \end{pmatrix}.
    \end{displaymath}
    We assume that $G(r,s,\eta):=rU(s,\eta)$, $K(r,s,\eta)$ are known matrix/vector functions. 

A primary aim of this work is to extend the analysis in \cite{adan} to a matrix-valued framework. In particular, \eqref{opq} corresponds to a vector-valued analogue of equation (2) in \cite{adan}. Functional equations of the form \eqref{opq}, \eqref{basic0} arise naturally from the analysis of multiplicative vector-valued stochastic recursions of the form
\begin{equation}
    \tilde{Z}_{n+1}=[V_{n}(Y_{n})\tilde{Z}_{n}+\tilde{S}_{n}(Y_{n})-\tilde{A}_{n}(Y_{n})]^{+},\label{vbn}
\end{equation}
where $\{Y_{n};n\in\mathbb{N}\}$ denotes an irreducible Markov chain with a finite state space $E=\{1,2,\ldots,N\}$, and vectors $\tilde{Z}_{n},\tilde{S}_{n}(Y_{n}),\tilde{A}_{n}(Y_{n})\in \mathbb{R}^{n}$, and where $V_{n}(Y_{n}),\tilde{S}_{n}(Y_{n}),\tilde{A}_{n}(Y_{n})$ depend on the state of $Y_{n}$, i.e., we consider Markov-modulated vector-valued stochastic recursions. 

Going one step further, our aim is also to deal with the solution of multidimensional versions of \eqref{opq}, which are quite challenging. A special case of such a class is given in \eqref{masip},
    \begin{equation}
        \tilde{Z}(s,t)=K(s,t)+R(s,t)\sum_{i=1}^{N}\tilde{P}^{(i)}\tilde{Z}(t,\alpha_{i}(t))+T(s)\sum_{i=1}^{N}\tilde{P}^{(i)}\tilde{Z}(s,\alpha_{i}(t)),\label{masip}
    \end{equation}
where $K(s,t)$, $R(s,t)$, $T(s)$, are known matrix-valued functions and $\alpha_{i}(t)$ are commutative contraction mappings in the right-half plane (e.g., $\alpha_{i}(t):=a_{i}t$, $a_{i}\in(0,1)$, $i=1,\ldots,N$). Note that this type of functional equation is different from \eqref{opq} since now we have two coupled recursions: one in $t$ and one in $s$, which creates cross dependency between variables. However, the solution machinery is closely related to the one for \eqref{opq}. Vector-valued functional equations of the form \eqref{masip} arise naturally in the modeling of Markov-modulated two-queue tandem queues with L\'evy input and consumption (see \cite{boxkelres} for the scalar case, where the contractions in the arguments on the right hand side do not depend on a background Markov chain). There, our focus is on deriving the Laplace-Stieltjes transform (LST) of the steady-state joint buffer content distribution, which is shown to satisfy a vector-valued functional equation of the form \eqref{masip}.

In view of multidimensional versions of functional equations \eqref{opq}, \eqref{basic0}, it would be of great interest to consider vector-valued reflected autoregressive processes that are described by stochastic recursions of the form $\tilde{Z}_{n+1}=[A\tilde{Z}_{n}+\tilde{S}_{n}-\tilde{A}_{n}]^{+}$, where the autoregressive parameter $A$ is now an $N\times N$ known matrix with elements $a_{i,j}\in[0,1)$, and our main concern is to provide the joint LST of the stationary vector $\tilde{Z}$. Under specific assumptions, the vector of the marginal LSTs is shown to satisfy a (more general) functional equation of the form \eqref{opq} (that can be handled in a similar way). On the other hand, in the case where our focus is on the joint LST of the stationary vector $\tilde{Z}$, the functional equation has a different form. In particular, when $N=2$ we have:
\begin{equation}
\begin{array}{rl}
    f(s_{1},s_{2})=&c_{1}(s_{1},s_{2})f(T(s_{1},s_{2}))-c_{2}(s_{1},s_{2})f(T(s_{1},\lambda_{2}))\vspace{2mm}\\
    &-c_{3}(s_{1},s_{2})f(T(\lambda_{1},s_{2}))+c_{4}(s_{1},s_{2})f(T(\lambda_{1},\lambda_{2})),
    \end{array}\label{bzaa}
\end{equation}
where, $T(s_{1},s_{2}):=A^{T}\tilde{s}=(a_{1,1}s_{1}+a_{2,1}s_{2},a_{1,2}s_{1}+a_{2,2}s_{2})$, $\tilde{s}=(s_{1},s_{2})^{T}$, and $c_{i}(s_{1},s_{2})$, $i=1,2,3,4$ are known functions. The form of \eqref{bzaa} seems to be simpler compared to \eqref{opq}, since the autoregressive parameter (i.e., the matrix $A$) does not depend on the state of a background Markov chain, although we have a two-dimensional argument that is expressed as a linear combination of $s_{1}$, $s_{2}$ that complicates the analysis. However, in case $T(s_{1},s_{2})$ is a contraction mapping, its solution method follows steps similar to the one of \eqref{opq}, since it relies on the use of multiple recursive terms.
\subsection{Related work}
In the following, we provide a brief overview of the existing analytical results in the scalar case, since to the best of our knowledge, the only vector-valued version of the stochastic recursion \eqref{vbn} that has been treated analytically refers to the case where $V_{n}(Y_{n})=a\in(0,1)$; see \cite{dimitriou2024markov}.

In \cite{box1}, the authors considered the (scalar) recursion $Z_{n+1}=[aZ_{n}+S_{n}-A_{n+1}]^{+}$ (with $[x]^{+}:=max\{x,0\}$), where $\{S_{n}-A_{n+1}\}_{n\in\mathbb{N}_{0}}$ forms a sequence of independent and identically
distributed (i.i.d.) random variables and $a\in(0,1)$. The authors provided explicit results for the case where $\{S_{n}\}_{n\in\mathbb{N}_{0}}$ being a sequence
of independent $exp(\lambda)$ distributed random variables, and $\{Y_{n}\}_{n\in\mathbb{N}_{0}}$ i.i.d., nonnegative and independent of $\{S_{n}\}_{n\in\mathbb{N}_{0}}$ with
distribution function $F_{Y}(.)$ and Laplace–Stieltjes transform (LST) $\phi_{Y}(.)$. Note that in such a case, $Z_{n}$ could be interpreted as the workload in a queueing system just before the $n$th arrival, which adds $Y_{n}$ work, and makes obsolete a fixed fraction $1-a$ of the already present work. The case where $a=1$ corresponds to the classical Lindley recursion describing the waiting time of the classical M/G/1 queue, while the case where $a=-1$ was investigated in \cite{vlasiou}. Further progress has recently been made in \cite{box2} where the scalar autoregressive process described by the recursion $Z_{n+1}=[V_{n}Z_{n}+S_{n}-A_{n+1}]^{+}$ was investigated, where $\{V_{n}\}_{n\in\mathbb{N}_{0}}$ a sequence of i.i.d. random variables. In \cite{box3}, the authors considered the case where $V_{n}Z_{n}$ replaced by $F(Z_{n})$, where $\{F(t)\}$ is a Levy subordinator (recovering also the case in \cite{box1}, where $F(t)=at$). Recently, in \cite{boxman}, the authors motivated by applications that arise in queueing and insurance risk models, they considered Lindley-type recursions where the sequences $\{S_{n}\}_{n\in\mathbb{N}_{0}}$, $\{A_{n}\}_{n\in\mathbb{N}_{0}}$ obey a semi-linear dependence. These recursions can also be treated as of autoregressive type. Furthermore, the authors in \cite{adan} developed a method to study scalar versions of functional equations of the form \eqref{opq} that arise in a wide range of queueing, autoregressive and branching processes. Finally, the author in \cite{hoo}, considered a generalized version of the model in \cite{box2}, by assuming $V_{n}$ to take values in $(-\infty,1]$. In \cite[eq. (9)]{adan3}, the authors considered a queueing system with two classes of impatient customers, leading to a specific example of \eqref{basic0} that was analyzed in detail; see also \cite{kumar2021}. Quite recently, in \cite{dimitriou2024} the authors generalized the work in \cite{box1} by considering, among others, non-trivial dependence structures among $\{S_{n}\}_{n\in\mathbb{N}_{0}}$, $\{A_{n}\}_{n\in\mathbb{N}}$. We also mention the work in \cite{boxkelres} where, among others, the authors attempted to provide analytical results for ASIP (asymmetric inclusion process) tandem queues with consumption, the analysis of which leads to a functional equation that can be solved similarly as those in \cite{box3}.  

A primary motivation for our work is related to recent results on vector-valued reflected autoregressive processes. Quite recently, in \cite{dimitriou2024markov}, the author investigated vector-valued reflected autoregressive processes, where the sequences $\{S_{n}\}_{n\in\mathbb{N}_{0}}$, $\{A_{n}\}_{n\in\mathbb{N}}$ are governed by an irreducible background Markovian process with finite state space, i.e., the author considered Markov-modulated reflected autoregressive processes, where each transition of the Markov chain generates a new interarrival time $A_{n+1}$ and its corresponding service time $S_{n}$, thus, considered the Markov-dependent version of the process analysed in \cite{box1}. The analysis in \cite{dimitriou2024markov} results in vector-valued functional equations of the form \eqref{dimm}.
    \begin{equation}
    \begin{array}{rl}
        \tilde{Z}(r,s,\eta)=&G(r,s,\eta)\tilde{Z}(r,as,\eta)+K(r,s,\eta),\vspace{2mm}\\
        \tilde{Z}(s)=&H(s)\tilde{Z}(as)+\tilde{V}(s),
        \end{array}\label{dimm}
    \end{equation}
    where $a\in(0,1)$. 

Note that the specific case of $a = 1$ in \cite{dimitriou2024markov} corresponds to the waiting time in a single server queue with Markov-dependent interarrival and service times studied in \cite{adan2}. Markov-dependent structure of the form considered in \cite{adan2} has also been used in insurance mathematics; see \cite{albbox}. The process analysed in \cite{adan2} (i.e., for $a=1$) is a special case of the class of processes studied in \cite{asmkella}, although in \cite{adan2}, all the results were given explicitly. The case where $a=-1$ was investigated in \cite{vlasioudep} (see also \cite[Chapter 5]{vlasiou}) in the context of carousel models. In \cite{dimitriou2024markov} the author focused on the case where $a\in(0,1)$. Moreover, contrary to the case considered in \cite{adan2,vlasioudep}, in which given the state of the Markov chain at times $n$, $n+1$, the distributions of $A_{n+1}$, $S_{n}$ are independent of one another for all $n$ (although their distributions depend on the state of the background Markov chain), the author further considered the case where there is a dependence among  $\{S_{n}\}_{n\in\mathbb{N}_{0}}$, $\{A_{n}\}_{n\in\mathbb{N}}$ based on Farlie-Gumbel-Morgenstern (FGM) copula. More precisely, $\{(S_{n},A_{n+1})\}_{n\in\mathbb{N}_{0}}$ form a sequence of i.i.d. random vectors with a distribution function defined by FGM copula and dependent on the state of the underlying discrete time Markov chain. The author also considered the case where $\{(S_{n},A_{n+1})\}_{n\in\mathbb{N}_{0}}$ have a bivariate matrix-exponential distribution, which is dependent on the state of the underlying discrete time Markov chain, as well as the case where there is a linear dependence among them. Special treatment was also given to the analysis of the case where the service times were dependent on the waiting time, as well as to the case where the server's speed is workload proportional, i.e., a modulated shot-noise queue. The time-dependent analysis of a Markov-modulated reflected autoregressive process was also investigated.
\subsection{Our contribution} Our primary goal is to explore a class of vector-valued functional equations \eqref{opq}, \eqref{basic0}, \eqref{masip} that arise naturally when we are dealing with Markov-modulated stochastic recursions. More precisely, vector-valued functional equations \eqref{opq}, \eqref{basic0} arise in the analysis of stochastic recursions \eqref{vbn}. Equations of the form \eqref{masip} arise when we are dealing with multi-dimensional versions of \eqref{vbn}.

In Section \ref{theory}, we provide a detailed theoretical framework in solving \eqref{opq}, \eqref{basic0} (see Theorems \ref{th1}, \ref{tranth}), and thus extending the methodology developed in \cite{adan} from the scalar to the vector/matrix framework, also taking into account the transient behavior. We also provide a theoretical framework to solve \eqref{masip}; see Theorem \ref{th-masip-full}. We emphasize that although \eqref{masip} is of more general structure compared with \eqref{opq}, its solution machinery is similar.

The theoretical framework developed in Section \ref{theory} is applied in a number of application examples: More precisely, we first consider a MAP/G/1-type reflected autoregressive process, the analysis of which results in  a functional equation of the type given in \eqref{opq}. We also go one step further and consider extensions, under which given the state of the background Markov chain, there is additional dependence structure based on the FGM copula, i.e., the service and the interarrival times are not conditionally independent, but instead they are dependent based on the FGM copula. In this work, we have mainly restricted ourselves to the case where $\alpha_{i}(s)$ are commutative contraction mappings. In most of the application examples, we focus on the case where $\alpha_{i}(s):=a_{i}s$. However, in subsection \ref{mark} we deal with a commutative mapping that is not a contraction as in \cite{adan3}, i.e., $\alpha_{i}(s)=s+\mu_{i}c$, $i=1,\ldots,N$. In subsection \ref{inar}, we consider an even more general version of \eqref{opq}, related to the generating function of stationary queue-length distribution of a Markov-modulated $M/G/1$ queue with a general impatience scheme. This model is described by an integer-valued Markov-modulated reflected autoregressive process.

We then cope with the transient analysis of stochastic recursions \eqref{vbn} that results in a vector-valued functional equation of the form \eqref{basic0}, thus, extending considerably the work in \cite{dimitriou2024markov}, where vector-valued functional equations of the form in \eqref{dimm} were investigated, as well as the seminal works in \cite{adan,box1}, where scalar functional equations of the form \eqref{basic0}, \eqref{opq} were studied, respectively. More explicitly, we focus on a Markov modulated M/G/1-type reflected autoregressive process. The analysis of such a model leads to a vector-valued functional equation of the form \eqref{basic0}. To come up with this equation, we make use of Liouville's theorem \cite{tit} and Wiener-Hopf boundary value theory \cite{cohen}. We further consider the transient analysis of a Markov-modulated fluid flow model with consumption by generalising existing work in fluid queues by incorporating the concept of consumption. That case leads also to a functional equation of the form \eqref{basic0}.  

Our next contribution relies on extending the analysis to multidimensional versions of \eqref{opq}, as given in \eqref{masip}. An interesting application example is a modulated ASIP tandem system of two queues with L\'evy input and consumption. This is the Markov-modulated version of the model in \cite{boxkelres}, where we have further assumed that the consumption rate depends on the state of the background Markov chain. Thus, we considerably generalize \cite{boxkelres}. Our analysis leads to a vector valued functional equation of the form \eqref{masip} for the LST of the steady-state joint buffer content distribution. 

Finally, we investigate more intricate vector-valued functional equations than \eqref{opq} by focusing on the stationary analysis of a vector-valued autoregressive process (VAR(1)) that is described by stochastic recursions of the form $\tilde{Z}_{n+1}=[A\tilde{Z}_{n}+\tilde{S}_{n}-\tilde{A}_{n}]^{+}$, where $\tilde{Z}_{n}=(Z_{1,n},\ldots,Z_{N,n})^{T}$, and $A$ is an $N\times N$ matrix. Under some independence assumptions, the vector of the marginal LSTs of the stationary version $\tilde{Z}$ satisfies a slightly more general functional equation of the form \eqref{opq}, although its solution method is similar. A quite challenging part is to consider the LST the stationary joint distribution of $\tilde{Z}$ that satisfies a vector-valued functional equation of the form \eqref{bzaa}. Despite the fact that the form of \eqref{bzaa} is quite complicated, and the arguments of each component of the LST are linear combinations of $s_{i}$, $i=1,\ldots,N$, we succeed to solve it in case where the argument of $f(.)$ in the right hand side of \eqref{bzaa} is an $N$-dimensional contraction mapping. Its solution method is similar to that of \eqref{opq} and involves multiple recursive terms; see Theorem \ref{thvar}.

\subsection{Organization of the paper}
The remainder of the paper is organized as follows. In Section \ref{theory}, we present the general theory to solve equations \eqref{opq}, \eqref{basic0}, \eqref{masip}. The results obtained there are applied in in a series of application examples presented in Section \ref{appl}. In subsection \ref{mdc}, we investigate the stationary behaviour of a particular autoregressive-type MAP/G/1 queue with dependencies between successive service times and between inter-arrival times and service times, as well as an extension that incorporates additional dependencies based on the FGM copula. The analysis results on a functional equation of the form \eqref{opq}, where $\alpha_{i}(s)=a_{i}s$, $a_{i}\in(0,1)$, $i=1,\ldots,N$. Subsection \ref{shot} is devoted to the stationary analysis of a modulated shot-noise queue, while in subsection \ref{mark} we focus on the stationary behaviour of a modulated Markovian queue with dependencies among service time and waiting time. An integer vector-valued reflected autoregressive process is investigated in subsection \ref{inar}. In subsection \ref{mmc}, we focus on the time-dependent analysis of a Markov-modulated reflected autoregressive process, which corresponds to a more general version of the model considered in Section \ref{mdc}, and where its analysis results in \eqref{basic0}. The transient analysis of a Markov-modulated fluid queue with consumption is treated in subsection \ref{fluid}. 

In Section \ref{multi}, we consider multidimensional versions of \eqref{opq} of the form given in \eqref{masip}, and thus solved by using Theorem \ref{th-masip-full}. The stationary behaviour of a Markov-modulated ASIP tandem queue with consumption is investigated in subsection \ref{asip}, and its analysis result in \eqref{masip}. Subsection \ref{var2} refers to stationary analysis of reflected VAR(1) processes. We consider two instances of such processes. The analysis of the first one results in a generalized version of \eqref{opq}. The second is much more general (see \eqref{bzaa} when $N=2$), since the solution approach needs now the argument of the unknown LST on the right-hand side of the functional equation to be a $N$-dimensional contraction mapping. Conclusions and some topics for further research are presented in Section \ref{conc}.
\section{The recursion: general theory}\label{theory}
In this section, we present the general theory for solving vector-valued functional equations of the form \eqref{opq} for the vector $\tilde{Z}(s)$, whose components are defined by $z_{i}(s):=\mathbb{E}(e^{-sW}1_{\{Y=i\}})$, $i=1,\ldots,N$, and $W$ is a non-negative random variable (with $\mathbb{E}(W)<\infty$) modulated by a background Markov chain with state-space $E=\{0,1,\ldots,N\}$ (it is the LST (with respect to $W$) of the joint distribution of $(W,Y)$ evaluated at $Y=i$). Let $Y$ denote the state of this Markov chain in stationarity, with distribution $\tilde{\pi}:=(\pi_{1},\ldots,\pi_{N})^{T}$. Each component $z_{i}(s)$ represents the contribution to the transform arising from realizations where the background Markov chain is in state 
$i$. The functional equation describes a Markov-modulated stochastic fixed-point relation, where the terms $z_{i}(\alpha_{i}(s))$ correspond to evaluating the transform at contracted arguments, while the matrix
$H(s)$ captures the interaction between states of the background Markov chain, and 
the vector $\tilde{V}(s)$ represents an external input, such that $\tilde{V}(s)=\tilde{0}$ (i.e., the $N\times 1$ vector of zeros).

Assume that the elements of $H(s)$, $\tilde{V}(s)$ are analytic functions, i.e there is a $\zeta>0$ such that $h_{i,j}(s)$, $v_{j}(s)$ can be continued analytically in $Re(s)>-\zeta$. Moreover, $\alpha_{i}(s)$, $i=1,\ldots,N$ are commutative contraction mappings on the closed positive half plane. More precisely, let $\mathbb{C}_{+}:=\{s\in \mathbb{C}:Re(s)\geq 0\}$ and define $\alpha_{i}:\mathbb{C}_{+}\to \mathbb{C}_{+}$, $i=1,\ldots,N$ with
\begin{displaymath}
    |\alpha_{i}(s)-\alpha_{i}(u)|\leq \kappa |s-u|,\,s,u\in\mathbb{C}_{+},
\end{displaymath}
with $\kappa\in(0,1)$ and
\begin{displaymath}
    \alpha_{i}(\alpha_{j}(s))=\alpha_{j}(\alpha_{i}(s)).
\end{displaymath}
A typical example is $\alpha_{i}(s)=a_{i}s$, $a_{i}\in(0,1)$ with $\kappa=max\{a_{1},\ldots,a_{N}\}$.

Therefore, the contractions $\alpha_{i}(s)$ have the same fixed-point $a$ with $\kappa\in(0,1)$, where for the case $\alpha_{i}(s)=a_{i}s$ is $a=0$. In that case, it is readily seen from the form (and the interpretation) of $H(s)$ that $H(0)\to P^{T}$, i.e., the transpose of the one-step transition probability matrix of the background Markov chain. In our general framework, $H(a)$ plays the
role of a transition probability matrix (among the states of the background Markov chain) at the fixed point. Note also that for $i_1+\ldots+i_N=n$, where $n$ is sufficiently large, we have that $\alpha_{i_{1},\ldots,i_{N}}(s)$ is close to $a$. 

Let us consider the general case, where the contractions $\alpha_{i}(s)$ have the same fixed-point $a$, but not necessarily equal to 0. Iterating $n$ times \eqref{opq} we obtain
\begin{equation}
    \tilde{Z}(s)=\sum_{k=0}^{n}\sum_{i_{1}+\ldots+i_{N}=k}F_{i_{1},\ldots,i_{N}}(s)\tilde{V}(\alpha_{i_{1},\ldots,i_{N}}(s))+\sum_{i_{1}+\ldots+i_{N}=n+1}F_{i_{1},\ldots,i_{N}}(s)\tilde{Z}(\alpha_{i_{1},\ldots,i_{N}}(s)),
   \label{itert}
\end{equation}
where $\alpha_{i_{1},\ldots,i_{N}}(s)=\alpha_{1}^{i_{1}}(\alpha_{2}^{i_{2}}(\ldots(\alpha_{N}^{i_{N}}(s))\ldots))$ and $\alpha_{i}^{n}(s)$ denotes the $n$th iterate of $\alpha_{i}(s)$, with $\alpha_{0,\ldots,0}(s)=s$, and the functions $F_{i_{1},\ldots,i_{N}}(s)$ are recursively defined by
\begin{equation}
    F_{i_{1},\ldots,i_{N}}(s)=\sum_{k=1}^{N}F_{i_{1},\ldots,i_{k}-1,\ldots,i_{N}}(s)H(\alpha_{i_{1},\ldots,i_{k}-1,\ldots,i_{N}}(s))\tilde{P}^{(k)},\label{fas}
\end{equation}
with $F_{0,\ldots,0}(s):=I$, and $F_{i_{1},\ldots,i_{N}}(s)=O$ (i.e., the zero matrix) if one of the indices equals $-1$.

\begin{theorem}\label{th1}
Assume that $\mathbb{E}(W)<\infty$, and that $\tilde{V}(s)$ is analytic in a neighborhood of $a$, with $\tilde{V}(a)=\tilde{0}$. Then, the vector $\tilde{Z}(s)$ that satisfies \eqref{opq} is given by
\begin{equation}
\tilde{Z}(s)
=\sum_{k=0}^{\infty}\sum_{i_{1}+\ldots+i_{N}=k}
F_{i_{1},\ldots,i_{N}}(s)\,\tilde{V}(\alpha_{i_{1},\ldots,i_{N}}(s))
+\lim_{n\to\infty}\sum_{i_{1}+\ldots+i_{N}=n+1}
F_{i_{1},\ldots,i_{N}}(s)\,\tilde{Z}(a),\label{solution}
\end{equation}
where the series converges absolutely and the limit exists, provided that
\begin{displaymath}
    \|H(a)\|_1\kappa < 1,
\end{displaymath}
where $\|.\|_{1}$ denotes the matrix $1$-norm (the maximum absolute column sum of the matrix).
\end{theorem}

\begin{proof}
For every $n$, we rewrite \eqref{itert} as follows:
\begin{equation}
    \begin{array}{rl}
       \tilde{Z}(s)
=&\sum_{k=0}^{n}\sum_{i_{1}+\ldots+i_{N}=k}
F_{i_{1},\ldots,i_{N}}(s)\,\tilde{V}(\alpha_{i_{1},\ldots,i_{N}}(s)) \vspace{2mm}\\
&+\sum_{i_{1}+\ldots+i_{N}=n}
F_{i_{1},\ldots,i_{N}}(s)\,\tilde{Z}(a)+\sum_{i_{1}+\ldots+i_{N}=n+1}
F_{i_{1},\ldots,i_{N}}(s)\big(\tilde{Z}(\alpha_{i_{1},\ldots,i_{N}}(s))-\tilde{Z}(a)\big).
    \end{array}\label{decomp} 
\end{equation}
The proof proceeds by controlling each term in \eqref{decomp}
and then taking the limit as $n\to\infty$. In particular, we have to a) show the absolute convergence of the series in the first term of the right-hand side of \eqref{decomp}, b) show the convergence of the second term of the right-hand side of \eqref{decomp}, and c) show that the last term on the right-hand side of \eqref{decomp} vanishes.

Before proceeding with the proof, let us discuss the above requirements. The convergence of the series requires the bounding of the matrices $F$ that are recursively defined in \eqref{fas}, thus we have to use the contraction mapping property, the continuity of $H$, as well as the projection property of the matrices $\tilde{P}^{(i)}$, which is crucial in order to avoid the multiplicative growth of the terms. The analyticity of $\tilde{V}(s)$ close to $a$ plays also a vital role. To prove that the last term on the right-hand side of \eqref{decomp} vanishes we have to control $\tilde{Z}(\alpha_{i_{1},\ldots,i_{N}}(s))-\tilde{Z}(a)$. So it is important to exploit the contraction mapping along with the Lipschitz continuity of $\tilde{Z}(s)$. The bounding of $F$, along with the projection property of $\tilde{P}^{(i)}$ also plays a vital role. Finally, the existence of the limit as $n\to\infty$ of the second term requires the control $H(\alpha_{i_{1},\ldots,i_{N}}(s))-H(a)$ and the bounding of $F$. Again, the projection property of $\tilde{P}^{(i)}$ simplifies the analysis. We are now ready to proceed with a step-by-step proof.

\medskip\noindent
\textbf{Step 1: Lipschitz continuity of $\tilde{Z}(s)$.}

Since $z_j(s)=\mathbb{E}(e^{-sW}1_{\{Y=j\}})$, $j\in E$, and $W\geq 0$, $\mathbb{E}(W)<\infty$ we have
\begin{displaymath}
    |z_j(s)-z_j(t)|\leq \mathbb{E}(|e^{-sW}-e^{-tW}|1_{\{Y=j\}})
\leq |s-t|\mathbb{E}(W1_{\{Y=j\}}).
\end{displaymath}
Hence $\tilde{Z}(s)$ is Lipschitz continuous, and there exists $C^{\prime}>0$ such that
\begin{displaymath}
    \|\tilde{Z}(s)-\tilde{Z}(a)\|_\infty=\max_{j\in E}|z_j(s)-z_j(a)| \leq (\max_{j\in E}\mathbb{E}(W1_{\{Y=j\}})) |s-a|=C^{\prime}|s-a|,
\end{displaymath}
with $C^{\prime}:=\max_{j\in E}\mathbb{E}(W1_{\{Y=j\}})$ ($\|.\|_{\infty}$ stands for vector $\infty$-norm). Since $\alpha_{i}(s)$ are contraction mappings, there exists $\kappa\in(0,1)$ such that for $i_1+\ldots+i_N=n$, we have $|\alpha_{i_{1},\ldots,i_{N}}(s)-a| \le \kappa^n |s-a|$. Therefore, we have
\begin{displaymath}
    \|\tilde{Z}(\alpha_{i_{1},\ldots,i_{N}}(s))-\tilde{Z}(a)\|_{\infty}
\leq C^{\prime} \kappa^n |s-a|.
\end{displaymath}

\medskip

\noindent
\textbf{Step 2: Bound on $F_{i_{1},\ldots,i_{N}}(s)$.}

Since $\tilde{P}^{(i)}=e_i e_i^T$, $i=1,\ldots,N$ (where $e_{i}$ denotes the $N\times 1$ vector with the $i$th element equal to 1, and the remaining elements equal to 0), each matrix $F_{i_{1},\ldots,i_{N}}(s)$ has at most one nonzero column (Due to $\tilde{P}^{(i)}=e_i e_i^T$, each matrix 
$F_{i_{1},\ldots,i_{N}}(s)$ has at most one nonzero column. 
Hence, when summing over all multi-indices, no combinatorial growth appears). Moreover, by setting $s=a$ in \eqref{opq} we have
\begin{displaymath}
    \tilde{Z}(a)=H(a)\tilde{Z}(a),
\end{displaymath}
and thus, $1$ is an eigenvalue of $H(a)$ and $\tilde{Z}(a)$ is a corresponding right eigenvector. Moreover, the continuity of $H(s)$ at $s=a$, implies that for every $\epsilon>0$ there exists $\delta>0$ such that for $|u-a|<\delta$,
\begin{displaymath}
    \|H(u)-H(a)\|_1 < \epsilon.
\end{displaymath}
Hence,
\begin{displaymath}
    \|H(u)\|_1=\|(H(u)-H(a))+H(a)\|_1\leq \|H(u)-H(a)\|_1+\|H(a)\|_1 \leq \|H(a)\|_1 + \epsilon.
\end{displaymath}
Using \eqref{fas}, and since $\tilde{P}^{(k)}=e_k e_k^T$, each term $F_{i_{1},\ldots,i_{k}-1,\ldots,i_{N}}(s)H(\alpha_{i_{1},\ldots,i_{k}-1,\ldots,i_{N}}(s))\,\tilde{P}^{(k)}$ has nonzero entries only in column $k$. Therefore, when taking the matrix $1$-norm,
\begin{displaymath}
    \begin{array}{rl}
        \|F_{i_{1},\ldots,i_{N}}(s)\|_1
= &\max_{1\leq k\le N}
\|F_{i_{1},\ldots,i_{k}-1,\ldots,i_{N}}(s)H(\alpha_{i_{1},\ldots,i_{k}-1,\ldots,i_{N}}(s))\tilde{P}^{(k)}\|_1  \vspace{2mm}\\
         \leq & \max_{1\le k\le N}
\|F_{i_{1},\ldots,i_{k}-1,\ldots,i_{N}}(s)\|_1 \|H(\alpha_{i_{1},\ldots,i_{k}-1,\ldots,i_{N}}(s))\|_1 
    \end{array}
\end{displaymath}
For $i_{1}+\ldots+i_{N}=n$ we repeatedly apply this bound, so that: 
\begin{displaymath}
    \|F_{i_{1},\ldots,i_{k}-1,\ldots,i_{N}}(s)\|_1\leq \max_{1\leq l\leq N}\|F_{i_{1},\ldots,i_{k}-1,\ldots,i_{l}-1,\ldots,i_{N}}(s)\|_1 \|H(\alpha_{i_{1},\ldots,i_{k}-1,\ldots,i_{l}-1,\ldots,i_{N}}(s))\|_1.
\end{displaymath}
By repeatedly applying the above bound and using the fact that 
$\alpha_{i_{1},\ldots,i_{N}}(s)\to a$, together with the continuity of $H(\cdot)$, we obtain for sufficiently large $n$ that
\begin{displaymath}
    \|F_{i_{1},\ldots,i_{N}}(s)\|_1\leq \prod_{m=1}^{n}(\|H(a)\|_1+\epsilon)=(\|H(a)\|_1+\epsilon)^n.
\end{displaymath}

\medskip

\noindent
\textbf{Step 3: Vanishing of the third term in \eqref{decomp}}

For $i_{1}+\ldots+i_{N}=n+1$, and using the contraction mapping, we have already shown in Step 1 that 
\begin{displaymath}
    \|\tilde{Z}(\alpha_{i_{1},\ldots,i_{N}}(s))-\tilde{Z}(a)\|_\infty
\leq C^{\prime} \kappa^{n+1} |s-a|.
\end{displaymath}
Combining with Steps 1 and 2,
\begin{displaymath}
    \|F_{i_{1},\ldots,i_{N}}(s)[\tilde{Z}(\alpha_{i_{1},\ldots,i_{N}}(s))-\tilde{Z}(a)]\|_\infty
\leq (\|H(a)\|_1+\epsilon)^{n+1} C^{\prime} \kappa^{n+1} |s-a|.
\end{displaymath}
Let
\begin{displaymath}
    D_n(s)
=\sum_{i_{1}+\ldots+i_{N}=n+1}
F_{i_{1},\ldots,i_{N}}(s)\big(\tilde{Z}(\alpha_{i_{1},\ldots,i_{N}}(s))-\tilde{Z}(a)\big).
\end{displaymath}
Since each matrix $F_{i_{1},\ldots,i_{N}}(s)$ has at most one nonzero column, the vector
$F_{i_{1},\ldots,i_{N}}(s)(\tilde{Z}(\alpha_{i_{1},\ldots,i_{N}}(s))-\tilde{Z}(a))$ contributes only to a single coordinate (due to the projection property of $\tilde{P}^{(k)}$).
Hence, when taking the $\infty$-norm,
\begin{displaymath}
    \|D_n(s)\|_\infty
\le \max_{i_1+\ldots+i_N=n+1}
\|F_{i_{1},\ldots,i_{N}}(s)(\tilde{Z}(\alpha_{i_{1},\ldots,i_{N}}(s))-\tilde{Z}(a))\|_\infty.
\end{displaymath}
Using the previous bounds,
\begin{displaymath}
    \|D_n(s)\|_{\infty}
\leq C^{\prime} |s-a|((\|H(a)\|_1+\epsilon)\kappa)^{n+1}.
\end{displaymath}
Since $\|H(a)\|_1\kappa<1$, there exists $\epsilon>0$ such that $(\|H(a)\|_1+\epsilon)\kappa<1$. Thus, it follows that $D_n(s)\to 0$ as $n\to\infty$.

\medskip

\noindent
\textbf{Step 4: Convergence of the series.}

Our aim is to show that the series
\begin{displaymath}
    \sum_{k=0}^{\infty}\sum_{i_{1}+\ldots+i_{N}=k}
F_{i_{1},\ldots,i_{N}}(s)\,\tilde{V}(\alpha_{i_{1},\ldots,i_{N}}(s)),
\end{displaymath}
converges absolutely. Fix $k$ and consider all multi-indices such that $i_1+\ldots+i_N=k$. Then, we have
\begin{displaymath}
    \|F_{i_{1},\ldots,i_{N}}(s)\tilde{V}(\alpha_{i_{1},\ldots,i_{N}}(s))\|_\infty
\leq \|F_{i_{1},\ldots,i_{N}}(s)\|_1 \|\tilde{V}(\alpha_{i_{1},\ldots,i_{N}}(s))\|_\infty.
\end{displaymath}
Having in mind a) from Step 2 that $\|F_{i_{1},\ldots,i_{N}}(s)\|_1 \le (\|H(a)\|_1+\epsilon)^k$, b) the fact that since $\tilde{V}(a)=0$ and $\tilde{V}$ is analytic at $a$, there exists $M>0$ such that $\|\tilde{V}(u)\|_\infty \le M |u-a|$, and c) the contraction assumption, $|\alpha_{i_{1},\ldots,i_{N}}(s)-a| \le \kappa^k |s-a|$, which implies that $\|\tilde{V}(a_{i_{1},\ldots,i_{N}}(s))\|_\infty \leq M \kappa^k |s-a|$, we finally have
\begin{displaymath}
    \|F_{i_{1},\ldots,i_{N}}(s)\tilde{V}(\alpha_{i_{1},\ldots,i_{N}}(s))\|_\infty
\le M (\|H(a)\|_1+\epsilon)^k \kappa^k |s-a|.
\end{displaymath}

Now, since each matrix $F_{i_{1},\ldots,i_{N}}(s)$ has at most one nonzero column, the sum over all multi-indices with $i_1+\ldots+i_N=k$ does not accumulate over all indices when taking the $\infty$-norm. Instead, the norm is controlled by the maximal column contribution (due to the projection structure $\tilde{P}^{(i)}=e_i e_i^T$), and therefore by summing over all multi-indices with $i_1+\ldots+i_N=k$, the $\infty$-norm of the sum is bounded. That is, there exists a constant $\tilde{C}>0$ such that
\begin{displaymath}
    \left\|
\sum_{i_1+\ldots+i_N=k}
F_{i_{1},\ldots,i_{N}}(s)\,\tilde{V}(\alpha_{i_{1},\ldots,i_{N}}(s))
\right\|_\infty 
\le \tilde{C} \max_{i_1+\cdots+i_N=k}
\|F_{i_{1},\ldots,i_{N}}(s)\,\tilde{V}(\alpha_{i_{1},\ldots,i_{N}}(s))\|_\infty
\le \tilde{C} M ((\|H(a)\|_1+\epsilon)\kappa)^k |s-a|.
\end{displaymath}
Since $\|H(a)\|_1\kappa<1$, we can find $\epsilon>0$ such that $(\|H(a)\|_1+\epsilon)\kappa<1$, and thus we have
\begin{displaymath}
    \sum_{k=0}^\infty
\left\|
\sum_{i_1+\ldots+i_N=k}
F_{i_{1},\ldots,i_{N}}(s)\,\tilde{V}(\alpha_{i_{1},\ldots,i_{N}}(s))
\right\|_\infty
< \infty.
\end{displaymath}
Therefore, the series converges absolutely.
\medskip

\noindent
\textbf{Step 5: Existence of the limit term.}

Let
\begin{displaymath}
    S_n := \sum_{i_1+\ldots+i_N=n+1}
F_{i_1,\ldots,i_N}(s)\tilde{Z}(a).
\end{displaymath}
We show that $(S_n)$ is a Cauchy sequence in the vector $\infty$-norm, and since $S_{n}\in \mathbb{C}^{N}$, equipped with vector $\infty$-norm is a finite-dimensional normed space, it implies that converges.

Using the recursive structure of $F_{i_{1},\ldots,i_{N}}(s)$ given in \eqref{fas} (and the fact that $\tilde{P}^{(k)}=e_{k}e_{k}^{T}$), we may write
\begin{displaymath}
\begin{array}{rl}
    S_{n+1}=  & \sum_{i_{1}+\ldots+i_{N}=n+2}\sum_{k=1}^{N}F_{i_{1},\ldots,i_{k}-1,\ldots,i_{N}}(s)H(\alpha_{i_{1},\ldots,i_{k}-1,\ldots,i_{N}}(s))\tilde{P}^{(k)}\tilde{Z}(a)\vspace{2mm} \\
     =& \sum_{i_{1}+\ldots+i_{N}=n+2}\sum_{k=1}^{N}F_{i_{1},\ldots,i_{k}-1,\ldots,i_{N}}(s)H(\alpha_{i_{1},\ldots,i_{k}-1,\ldots,i_{N}}(s))z_{k}(a)e_{k}.
\end{array}
\end{displaymath}
Now fix $k$, and define
\begin{displaymath}
    j_{l}=\left\{\begin{array}{ll}
         i_{l},&l\neq k,  \\
         i_{k}-1,&l=k. 
    \end{array}\right.
\end{displaymath}
Then, $j_{1}+\ldots+j_{N}=n+1$, and $F_{i_{1},\ldots,i_{k}-1,\ldots,i_{N}}(s)=F_{j_{1},\ldots,j_{k},\ldots,j_{N}}(s)$. Thus,
\begin{displaymath}
\begin{array}{rl}
    S_{n+1}= &\sum_{i_{1}+\ldots+i_{N}=n+2}\sum_{k=1}^{N}F_{i_{1},\ldots,i_{k}-1,\ldots,i_{N}}(s)H(\alpha_{i_{1},\ldots,i_{k}-1,\ldots,i_{N}}(s))z_{k}(a)e_{k} \vspace{2mm}\\
    =&\sum_{j_{1}+\ldots+j_{N}=n+1}\sum_{k=1}^{N}F_{j_{1},\ldots,j_{k},\ldots,j_{N}}(s)H(\alpha_{j_{1},\ldots,j_{k},\ldots,j_{N}}(s))z_{k}(a)e_{k} \vspace{2mm}\\
    =&\sum_{j_{1}+\ldots+j_{N}=n+1}F_{j_{1},\ldots,j_{k},\ldots,j_{N}}(s)H(\alpha_{j_{1},\ldots,j_{k},\ldots,j_{N}}(s))\sum_{k=1}^{N}z_{k}(a)e_{k}\vspace{2mm}\\
    =&\sum_{j_{1}+\ldots+j_{N}=n+1}F_{j_{1},\ldots,j_{k},\ldots,j_{N}}(s)H(\alpha_{j_{1},\ldots,j_{k},\ldots,j_{N}}(s))\tilde{Z}(a),
\end{array}
\end{displaymath}
since $\sum_{k=1}^{N}z_{k}(a)e_{k}=\tilde{Z}(a)$. Having in mind that $H(a)\tilde{Z}(a)=\tilde{Z}(a)$, it follows that
\begin{displaymath}
    S_{n+1}-S_n
= \sum_{i_1+\ldots+i_N=n+1}
F_{i_1,\ldots,i_N}(s)\,
\big(H(\alpha_{i_1,\ldots,i_N}(s))-H(a)\big)\tilde{Z}(a).
\end{displaymath}
By taking norms we obtain
\begin{displaymath}
    \|S_{n+1}-S_n\|_\infty
\le \max_{i_1+\ldots+i_N=n+1}
\|F_{i_1,\ldots,i_N}(s)\|_1
\|H(\alpha_{i_1,\ldots,i_N}(s))-H(a)\|_1
\|\tilde{Z}(a)\|_\infty.
\end{displaymath}

Since $H(s)$ is analytic, it is locally Lipschitz continuous in a neighborhood of $a$. Hence, there exists a constant $L>0$ such that for $u$ sufficiently close to $a$, $\|H(u)-H(a)\|_1 \le L |u-a|$. Applying this with $u=\alpha_{i_{1},\ldots,i_{N}}(s)$ and using the contraction property, we obtain
\begin{displaymath}
\|H(\alpha_{i_{1},\ldots,i_{N}}(s))-H(a)\|_1
\le L |\alpha_{i_{1},\ldots,i_{N}}(s)-a|
\le L |s-a| \kappa^{n+1}.
\end{displaymath}
Moreover, from Step 2 and for $i_1+\ldots +i_{N}=n+1$, we have that
\begin{displaymath}
\|F_{i_1,\ldots,i_N}(s)\|_1 \leq (\|H(a)\|_1+\epsilon)^{n+1}.
\end{displaymath}
Combine the above, we have
\begin{displaymath}
    \|S_{n+1}-S_n\|_\infty
\le L|s-a|((\|H(a)\|_1+\epsilon)\kappa)^{n+1}\|\tilde{Z}(a)\|_\infty.
\end{displaymath}
Since $(\|H(a)\|_1+\epsilon)\kappa<1$, we obtain
\begin{displaymath}
    \sum_{n=0}^\infty \|S_{n+1}-S_n\|_\infty < \infty.
\end{displaymath}
Therefore, $(S_n)$ is a Cauchy sequence, and since $(\mathbb{C}^{N}, \|\cdot\|_\infty)$ is a complete space (since it is a finite-dimensional normed space), the limit of $S_{n}$ as $n\to\infty$ exists.
\end{proof}

In the following sections, we present several models, motivated by queueing related applications, where Theorem \ref{th1} applies. 
\subsection{The transient case}
We now focus on the investigation of the transient behavior, i.e., the solution approach for \eqref{basic0}. Clearly, $Re(\eta)\geq 0$, and $|r|<1$. In particular, assuming $|r|<1$ simplifies the analysis, since incorporate geometric convergence. As before, assume that $\alpha_{i}(s)$ are commutative contraction mappings with a common fixed-point $a$. Now, we are interested in deriving the vector $\tilde{Z}(r,s,\eta)$, whose components are $z_j(r,s,\eta)
= \sum_{n=1}^{\infty} r^n \mathbb{E}(e^{-sW_n-\eta T_n}1_{\{Y_n=j\}})$, $j\in E$. 

Iterating \eqref{basic0} $n$ times yields
\begin{equation}
\begin{array}{rl}
    \tilde{Z}(r,s,\eta)=&\sum_{k=0}^{n}\sum_{i_{1}+\ldots+i_{N}=k}F_{i_{1},\ldots,i_{M}}(r,s,\eta)K(r,\alpha_{i_{1},\ldots,i_{N}}(s),\eta)\vspace{2mm}\\&+\sum_{i_{1}+\ldots+i_{N}=n+1}F_{i_{1},\ldots,i_{N}}(r,s,\eta)\tilde{Z}(r,\alpha_{i_{1},\ldots,i_{N}}(s),\eta),
    \end{array}\label{iternm}
\end{equation}
where the functions $F_{i_{1},\ldots,i_{N}}(r,s,\eta)$ are recursively defined by
\begin{equation}
    F_{i_{1},\ldots,i_{N}}(r,s,\eta)=r\sum_{k=1}^{N}F_{i_{1},\ldots,i_{k}-1,\ldots,i_{N}}(r,s,\eta)U(\alpha_{i_{1},\ldots,i_{k}-1,\ldots,i_{N}}(s),\eta)\tilde{P}^{(k)},\label{tranf}
\end{equation}
with $F_{0,\ldots,0}(r,s,\eta):=I$, and $F_{i_{1},\ldots,i_{N}}(r,s,\eta)=O$ (that is, the zero matrix) if one of the indices is equal to $-1$. 

\begin{theorem}\label{tranth}
Assume that $\mathbb{E}(W_n)<\infty$, $|r|<1$, $Re(s)\ge 0$, $Re(\eta)\ge 0$, $\tilde{Z}(r,s,\eta)$ satisfies \eqref{basic0} and
\begin{itemize}
\item[(i)] $U(s,\eta)$ is analytic in a neighborhood of $s=a$, and $\|U(a,\eta)\|_1 \leq M_U(\eta),$ for some $M_{U}(\eta)>0$.
\item[(ii)] $K(r,s,\eta)$ is analytic in a neighborhood of $s=a$.
\item[(iii)] $|r|M_U(\eta) < 1$.
\end{itemize}
Then
\begin{equation}
\tilde{Z}(r,s,\eta)
= \sum_{k=0}^{\infty}\sum_{i_1+\ldots+i_N=k}
F_{i_1,\ldots,i_N}(r,s,\eta)\,K(r,\alpha_{i_1,\ldots,i_N}(s),\eta),\label{transientsol}
\end{equation}
and the series converges absolutely.
\end{theorem}

\begin{proof}
The proof has a similar structure with the one given in Theorem \ref{th1} and is composed to the following steps: Iterating \eqref{basic0} $n$ times yields
\begin{equation}
\begin{array}{rl}
\tilde{Z}(r,s,\eta)
=&\sum_{k=0}^{n}\sum_{i_{1}+\ldots+i_{N}=k}
F_{i_{1},\ldots,i_{N}}(r,s,\eta)K(r,\alpha_{i_{1},\ldots,i_{N}}(s),\eta)\vspace{2mm}\\
&+\sum_{i_{1}+\ldots+i_{N}=n}
F_{i_{1},\ldots,i_{N}}(r,s,\eta)\tilde{Z}(r,a,\eta)\vspace{2mm}\\
&+\sum_{i_{1}+\ldots+i_{N}=n+1}
F_{i_{1},\ldots,i_{N}}(r,s,\eta)
\big(\tilde{Z}(r,\alpha_{i_{1},\ldots,i_{N}}(s),\eta)-\tilde{Z}(r,a,\eta)\big).
\end{array}
\label{iter_final}
\end{equation}

\medskip

\noindent
\textbf{Step 1: Lipschitz continuity.}

Since $\mathbb{E}(W_n)<\infty$, we have
\begin{displaymath}
    |z_j(r,s,\eta)-z_j(r,t,\eta)|
\leq \sum_{n=1}^{\infty}|r|^n\mathbb{E}(|e^{-sW_{n}}-e^{-tW_{n}}|e^{-Re(\eta)T_{n}}1_{\{Y_{n}=j\}}) \leq |s-t|\sum_{n=1}^{\infty}|r|^n \mathbb{E}(W_n 1_{\{Y_n=j\}}).
\end{displaymath}
since $|e^{-sW_{n}}-e^{-tW_{n}}|\leq |s-t|W_{n}$. Thus, there exists $C_{1}:=\max_{1\leq j\le N} \sum_{n=1}^\infty |r|^n \mathbb{E}(W_n 1_{\{Y_n=j\}}) >0$ such that
\begin{displaymath}
    \|\tilde{Z}(r,s,\eta)-\tilde{Z}(r,a,\eta)\|_\infty
\leq C_{1} |s-a|.
\end{displaymath}

\medskip

\noindent
\textbf{Step 2: Bound on $F_{i_{1},\ldots,i_{N}}(r,s,\eta)$.}

From the \eqref{tranf}, and since $\|\tilde{P}^{(k)}\|_{1}=1$, we obtain
\begin{displaymath}
\|F_{i_{1},\ldots,i_{N}}(r,s,\eta)\|_1
\le |r| \sum_{k=1}^{N}
\|F_{i_{1},\ldots,i_{k}-1,\ldots,i_{N}}(r,s,\eta)\|_1 \,
\|U(\alpha_{i_{1},\ldots,i_{k}-1,\ldots,i_{N}}(s),\eta)\|_1.
\end{displaymath}
Since $\tilde{P}^{(k)}=e_k e_k^T$, we have for any $j\in E=\{1,\ldots,N\}$,
\begin{displaymath}
    \tilde{P}^{(k)} e_j =
\begin{cases}
e_k, & k=j,\\
0, & k\neq j.
\end{cases}
\end{displaymath}
Therefore, when applying the recursion to the $j$-th column, only the term corresponding to $k=j$ contributes. More precisely,
\[
F_{i_{1},\ldots,i_{N}}(r,s,\eta)e_j
= r\,F_{i_{1},\ldots,i_{j}-1,\ldots,i_{N}}(r,s,\eta)\,
U(\alpha_{i_{1},\ldots,i_{j}-1,\ldots,i_{N}}(s),\eta)\,e_j.
\]

Thus, each column evolves independently through a single recursive branch, and no accumulation over $k$ occurs. Consequently, taking the maximum over the columns yields
\[
\|F_{i_{1},\ldots,i_{N}}(r,s,\eta)\|_1
\le |r| \max_{1\le k\le N}
\left(
\|F_{i_{1},\ldots,i_{k}-1,\ldots,i_{N}}(r,s,\eta)\|_1 \,
\|U(\alpha_{i_{1},\ldots,i_{k}-1,\ldots,i_{N}}(s),\eta)\|_1
\right).
\]
Let $i_{1}+\ldots+i_{N}=n$ with $n\to\infty$, and since $U(s,\eta)$ is continuous at $s=a$, by using similar arguments as in Theorem \ref{th1} we have that $\|U(s,\eta)\|_{1}\leq\|U(a,\eta)\|_{1}+\epsilon:=M_{U}(\eta)+\epsilon$ for $\epsilon>0$. Then, by applying the previous inequality recursively $n$ times, we obtain
\begin{displaymath}
    \|F_{i_{1},\ldots,i_{N}}(r,s,\eta)\|_1
\le |r|^n \prod_{m=1}^{n} (M_{U}(\eta)+\epsilon)=|r|^{n}(M_{U}(\eta)+\epsilon)^{n}.
\end{displaymath}
\medskip

\noindent
\textbf{Step 3: Vanishing of the third term in \eqref{iter_final}.}

Define
\begin{displaymath}
    D_n(r,s,\eta)
=\sum_{i_{1}+\ldots+i_{N}=n+1}
F_{i_{1},\ldots,i_{N}}(r,s,\eta)
\big(\tilde{Z}(r,\alpha_{i_{1},\ldots,i_{N}}(s),\eta)-\tilde{Z}(r,a,\eta)\big).
\end{displaymath}
Then
\begin{displaymath}
    \|D_n(r,s,\eta)\|_\infty
\le \max_{i_{1}+\ldots+i_{N}=n+1}
\|F_{i_{1},\ldots,i_{N}}(r,s,\eta)\|_1
\|\tilde{Z}(r,\alpha_{i_{1},\ldots,i_{N}}(s),\eta)-\tilde{Z}(r,a,\eta)\|_\infty.
\end{displaymath}
Moreover, due to the contraction argument and the Lipschitz continuity of $\tilde{Z}(r,s,\eta)$ we have
 \begin{displaymath}
      \|\tilde{Z}(r,\alpha_{i_{1},\ldots,i_{N}}(s),\eta)-\tilde{Z}(r,a,\eta)\|_\infty
\leq C_{1} \kappa^{n+1}|s-a|.
 \end{displaymath}
 Therefore, using Step 2
\begin{displaymath}
    \|D_n(r,s,\eta)\|_\infty
\le C_{1} (|r|(M_U(\eta)+\epsilon)\kappa)^{n+1} |s-a|.
\end{displaymath}
Since $|r|M_U(\eta) < 1$ and $\kappa<1$, there exists an $\epsilon>0$ such that $|r|(M_U(\eta)+\epsilon)\kappa < 1$. Thus, $D_n(r,s,\eta)\to 0$ as $n\to\infty$.

\medskip

\noindent
\textbf{Step 4: Vanishing of the limit term in \eqref{iter_final}.}

Let
\begin{displaymath}
    S_n := \sum_{i_{1}+\ldots+i_{N}=n+1}
F_{i_{1},\ldots,i_{N}}(r,s,\eta)\tilde{Z}(r,a,\eta).
\end{displaymath}
Then
\begin{displaymath}
    \|S_n\|_\infty
\le 
\left\|
\sum_{i_{1}+\ldots+i_{N}=n+1} F_{i_{1},\ldots,i_{N}}(r,s,\eta)
\right\|_1
\|\tilde{Z}(r,a,\eta)\|_\infty.
\end{displaymath}

As in the stationary case, due to the projection structure $\tilde{P}^{(k)}=e_k e_k^T$, each matrix 
$F_{i_{1},\ldots,i_{N}}(r,s,\eta)$ has at most one nonzero column, and each column evolves independently through a unique recursive path. Thus, when summing over all multi-indices with $i_1+\ldots+i_N=n+1$, no combinatorial growth occurs. Thus, there exists a constant $C_{2}>0$ such that
\begin{displaymath}
\left\|
\sum_{i_{1}+\ldots+i_{N}=n+1} F_{i_{1},\ldots,i_{N}}(r,s,\eta)
\right\|_1
\le C_{2} \max_{i_{1}+\ldots+i_{N}=n+1}
\|F_{i_{1},\ldots,i_{N}}(r,s,\eta)\|_1.
\end{displaymath}

Using the bound from Step 2, we have $\|F_{i_{1},\ldots,i_{N}}(r,s,\eta)\|_1
\le (|r|(M_U(\eta)+\epsilon))^{n+1}$, and thus
\begin{displaymath}
\|S_n\|_\infty
\le C_2 \|\tilde{Z}(r,a,\eta)\|_\infty (|r|(M_U(\eta)+\epsilon))^{n+1}.
\end{displaymath}
Since $|r|M_U(\eta)<1$, there exists $\epsilon>0$ such that 
$|r|(M_U(\eta)+\epsilon)<1$, and therefore $S_n \to 0$ as $n\to\infty$.

\medskip

\noindent
\textbf{Step 5: Convergence of the series.}

We prove that the series
\begin{displaymath}
    \sum_{k=0}^{\infty}\sum_{i_{1}+\ldots+i_{N}=k}
F_{i_{1},\ldots,i_{N}}(r,s,\eta)K(r,\alpha_{i_{1},\ldots,i_{N}}(s),\eta),
\end{displaymath}
converges absolutely in the vector $\infty$-norm. Since $K(r,s,\eta)$ is analytic at $s=a$, it is locally bounded. Thus, there exists a $M_K>0$ such that $\|K(r,s,\eta)\|_\infty \le M_K$.

Hence,
\begin{displaymath}
    \|F_{i_1,\ldots,i_N}(r,s,\eta)K(r,\alpha_{i_1,\ldots,i_N}(s),\eta)\|_\infty
\le M_K (|r|(M_U(\eta)+\epsilon))^k.
\end{displaymath}
Using the projection structure, the sum over all multi-indices with 
$i_1+\ldots+i_N=k$ does not cause combinatorial growth, and therefore
\begin{displaymath}
    \left\|
\sum_{i_1+\ldots+i_N=k}
F_{i_1,\ldots,i_N}(r,s,\eta)K(r,\alpha_{i_1,\ldots,i_N}(s),\eta)
\right\|_\infty
\le M_K (|r|(M_U(\eta)+\epsilon))^k.
\end{displaymath}
Since $|r|M_U(\eta)<1$, there exists $\epsilon>0$ such that $|r|(M_U(\eta)+\epsilon)<1$, and thus, the series converges absolutely.

Having the above in mind, by taking the limit in \eqref{iter_final} as $n\to\infty$ we obtain \eqref{transientsol}.
\end{proof}

\subsection{An extension}
In the following, we cope with \eqref{masip}, which is a slightly more general version compared with \eqref{opq}, since now we have two sums, and the focus is of $\tilde{Z}(s,t) = (z_1(s,t),\ldots,z_N(s,t))^T$, where $z_i(s,t)=\mathbb{E}\big(e^{-sW_1 - tW_2}1_{\{Y=i\}}\big)$, thus, they are bounded and 
analytic for $Re(s)>0$, $Re(t)> 0$, and continuous on $\mathbb{C}_+\times\mathbb{C}_{+}$. Again, assume that the commutative contractions $\alpha_{i}(s)$, $i=1,\ldots,N$ have a common fixed-point $a$, thus there exists  $\kappa\in(0,1)$ such that $|\alpha_i(t)-a| \leq \kappa |t-a|$, and $\alpha_i(\alpha_j(s))=\alpha_j(\alpha_i(s))$.
\begin{theorem}\label{th-masip-full}
Consider the functional equation \eqref{masip} and assume that the matrices $K(s,t)$, $R(s,t)$ and $T(s)$ are analytic in a neighborhood of $(a,a)$. Define, for $k\geq 0$, the index set
\begin{displaymath}
    \Gamma_k := \{ \gamma = ((i_1,\varepsilon_1),\ldots,(i_k,\varepsilon_k)) 
: i_m\in\{1,\ldots,N\},\; \varepsilon_m\in\{R,T\},m=1,\ldots,k\},
\end{displaymath}
with $\Gamma_0=\{\emptyset\}$. Here, $\varepsilon_m\in\{R,T\}$ is a label indicating the type of recursion 
applied at step $m$: $\varepsilon_m=R$ corresponds to the term involving 
$R(s,t)\tilde{Z}(t,\alpha_i(t))$, while $\varepsilon_m=T$ corresponds to the 
term involving $T(s)\tilde{Z}(s,\alpha_i(t))$. Note that each $\gamma\in\Gamma_{k}$ corresponds to a path of length $k$ in the 
iteration tree generated by \eqref{masip}. For $\gamma\in\Gamma_k$, define recursively the iterated arguments: $(s_0,t_0)=(s,t)$ and for $m=1,\ldots,k$,
\begin{displaymath}
    (s_m,t_m)=
\begin{cases}
(t_{m-1},\, \alpha_{i_m}(t_{m-1})), & \varepsilon_m=R,\\[2mm]
(s_{m-1},\, \alpha_{i_m}(t_{m-1})), & \varepsilon_m=T.
\end{cases}
\end{displaymath}
Moreover, define the operator product
\[
\mathcal{K}_\gamma(s,t)
:= 
\begin{cases}
\prod_{m=1}^{k}R(s_{m-1},t_{m-1})\,\tilde{P}^{(i_m)}, & \varepsilon_m=R,\\
\prod_{m=1}^{k}T(s_{m-1})\,\tilde{P}^{(i_m)}, & \varepsilon_m=T,
\end{cases}
\]
with $\mathcal{K}_\emptyset=I$. Note that each matrix $\mathcal{K}_\gamma(s,t)$ has at most one nonzero column (This follows from the fact that $\tilde{P}^{(i)}=e_i e_i^T$, so that for any matrix $A$,
$A\tilde{P}^{(i)}=A e_i e_i^T$ has nonzero entries only in column $i$, as in Theorem \ref{th1}). 
By induction on the length of $\gamma$, it follows that $\mathcal{K}_\gamma(s,t)$ 
has at most one nonzero column.

Then, for all $(s,t)$ with $Re(s)\geq 0$, $Re(t)\geq 0$, the solution of \eqref{masip} is given by
\begin{equation}
\tilde{Z}(s,t)
= \sum_{k=0}^{\infty}\sum_{\gamma\in\Gamma_k}
\mathcal{K}_\gamma(s,t)\,K(s_k,t_k)
+ \lim_{n\to\infty}\sum_{\gamma\in\Gamma_n}
\mathcal{K}_\gamma(s,t)\,\tilde{Z}(s_n,t_n),
\label{masip-solution-correct}
\end{equation}
provided that $M(a)\kappa<1$, where
\[
M(a):=\max\{\|R(a,a)\|_1,\|T(a)\|_1\}.
\]
\end{theorem}
\begin{proof}
    The proof follows similar arguments (although more complicated due to the presence of mixed recursions, i.e, recursions that involve two different transformations of the arguments, namely $(s,t)\mapsto (t,\alpha_i(t))$ and 
$(s,t)\mapsto (s,\alpha_i(t))$, which are applied at each iteration) to those given in Theorem \ref{th1}. Note that along any iteration, the second argument contracts geometrically to $a$, 
while the first argument is mapped to the second through the mixed 
recursion, and therefore also converges to $a$. 

Moreover, the matrices $R(s,t)$ and $T(s)$ are analytic and locally 
bounded, and the projection structure of $\tilde{P}^{(i)}$ prevents (as in Theorem \ref{th1}) the combinatorial 
growth of the iterates. As a result, the iterative scheme exhibits the same geometric decay and boundedness properties as in the proof of Theorem \ref{th1}. For clarity, we present the detailed proof of the theorem in Appendix \ref{appen}.
\end{proof}
\begin{remark}
    To help the reader with the notation used in Theorem \ref{th-masip-full}, denote for $n=1$,
     \begin{equation}
        \tilde{Z}(s,t)=K(s,t)+R(s,t)\sum_{i_1=1}^{N}\tilde{P}^{(i)}\tilde{Z}(t,\alpha_{i}(t))+T(s)\sum_{i_1=1}^{N}\tilde{P}^{(i)}\tilde{Z}(s,\alpha_{i}(t)),\label{masip1}
    \end{equation}
    and let $\tilde{Z}_{1}(s,t):=\tilde{Z}(s,t)$. Then, for $n=2$, simple calculations result in
    \begin{equation}
\begin{array}{rl}
\tilde{Z}_{2}(s,t)=  
& K(s,t)  +\sum_{i_1=1}^{N}R(s,t)\tilde{P}^{(i_1)}K(t,\alpha_{i_1}(t)) +\sum_{i_1=1}^{N}T(s)\tilde{P}^{(i_1)}K(s,\alpha_{i_1}(t))  \vspace{2mm}\\
& +\sum_{i_1=1}^{N}\sum_{i_2=1}^{N}
R(s,t)\tilde{P}^{(i_1)}R(t,\alpha_{i_1}(t))\tilde{P}^{(i_2)}
\tilde{Z}(\alpha_{i_1}(t),\alpha_{i_2}(\alpha_{i_1}(t)))\vspace{2mm}\\
&+\sum_{i_1=1}^{N}\sum_{i_2=1}^{N}
R(s,t)\tilde{P}^{(i_1)}T(t)\tilde{P}^{(i_2)}
\tilde{Z}(t,\alpha_{i_2}(\alpha_{i_1}(t)))\vspace{2mm}\\
& +\sum_{i_1=1}^{N}\sum_{i_2=1}^{N}
T(s)\tilde{P}^{(i_1)}R(s,\alpha_{i_1}(t))\tilde{P}^{(i_2)}
\tilde{Z}(\alpha_{i_1}(t),\alpha_{i_2}(\alpha_{i_1}(t)))\vspace{2mm}\\
&+\sum_{i_1=1}^{N}\sum_{i_2=1}^{N}
T(s)\tilde{P}^{(i_1)}T(s)\tilde{P}^{(i_2)}
\tilde{Z}(s,\alpha_{i_2}(\alpha_{i_1}(t))).
\end{array}\label{lpa}
\end{equation}

Then, \eqref{lpa} can be written in the form given in Theorem \ref{th-masip-full} as follows:
\begin{equation*}
\tilde{Z}(s,t)
= \sum_{k=0}^{1}\sum_{\gamma\in\Gamma_k}
\mathcal{K}_\gamma(s,t)\,K(s_k,t_k)
+ \sum_{\gamma\in\Gamma_2}
\mathcal{K}_\gamma(s,t)\,\tilde{Z}(s_2,t_2).
\end{equation*}

In particular, for $k=0$ we have $\Gamma_0=\{\emptyset\}$, $\mathcal{K}_\emptyset=I$, and $(s_0,t_0)=(s,t)$, so the corresponding term is $K(s,t)$. For $k=1$, $\Gamma_1=\{(i_1,\varepsilon_1): i_1=1,\ldots,N,\ \varepsilon_1\in\{R,T\}\}$, and
\begin{displaymath}
    \mathcal{K}_\gamma(s,t)=
\begin{cases}
R(s,t)\tilde{P}^{(i_1)}, & \varepsilon_1=R,\\
T(s)\tilde{P}^{(i_1)}, & \varepsilon_1=T,
\end{cases}
\end{displaymath}
with $(s_1,t_1)=(t,\alpha_{i_1}(t))$ if $\varepsilon_1=R$, and $(s_1,t_1)=(s,\alpha_{i_1}(t))$ if $\varepsilon_1=T$. For $k=2$, $\Gamma_2=\{((i_1,\varepsilon_1),(i_2,\varepsilon_2))\}$, and each term is given by
\begin{displaymath}
    \mathcal{K}_\gamma(s,t)\,\tilde{Z}(s_2,t_2),
\end{displaymath}
where $\mathcal{K}_\gamma(s,t)$ is the interleaved product of the operators $R,T$ with the projections $\tilde{P}^{(i_m)}$, and $(s_2,t_2)$ are the iterated arguments defined recursively. Note also that the four double sums in the right-hand side of \eqref{lpa}  correspond to the paths $(R,R)$, $(R,T)$, $(T,R)$ and $(T,T)$, respectively.
\end{remark}
\begin{remark}
Although the contraction mappings $\alpha_i(\cdot)$ satisfy the commutativity property
$\alpha_i(\alpha_j(s))=\alpha_j(\alpha_i(s))$ in both Theorem \ref{th1} and Theorem \ref{th-masip-full}, the structure of the iteration is somewhat different. In particular, in Theorem~\ref{th1}, the functional equation involves only compositions of the mappings $\alpha_i(\cdot)$ and multiplication by matrices depending on the contracted argument. As a result, the order in which the mappings $\alpha_i(\cdot)$ are applied does not affect the final argument, and the iteration can be indexed by multi-indices $(i_1,\ldots,i_N)$ with $i_1+\ldots+i_N=n$, which simply count how many times each mapping is used.

Contrary to that case, in the present setting the recursion involves matrix-valued operators $R(s,t)$ and $T(s)$ acting on different arguments, together with interleaved projection matrices $\tilde{P}^{(i)}$. Although the mappings $\alpha_i(\cdot)$ still commute, the operators $R$ and $T$ do not commute in general, and their action depends on the evolving pair $(s_m,t_m)$. As a result, the order of application affects both the operator product and the sequence of arguments, and the iteration cannot be reduced to counts. Thus, we must keep tracking the full sequence $\gamma\in\Gamma_k$.
\end{remark}
\begin{remark}
    The condition 
\begin{displaymath}
    \max\{\|R(a,a)\|_1,\|T(a)\|_1\}\,\kappa < 1,
\end{displaymath}
is the natural analogue of the condition $\|H(a)\|_1\kappa<1$ in 
Theorem \ref{th1}. Indeed, $\kappa$ governs the geometric contraction of the iterated arguments, while $\|R(a,a)\|_1$ and $\|T(a)\|_1$ control the growth of the multiplicative coefficients in each recursive step. Therefore, their product determines the overall decay rate of each term in the iteration. When the above condition holds, we ensure the absolute convergence of the series and the existence of the 
limit of the remainder term (as in Theorem \ref{th1}).
\end{remark}
\begin{remark}
    In subsection \ref{var2} and Theorem \ref{thvar} we cope with the solution of the functional equation \eqref{bzaa} when $A=diag(a_1,a_2)$, i.e., $T(s_1,s_2)=(a_{1}s_{1},a_{2}s_{2})$, $a_{i}\in(0,1)$, $i=1,2$, $c_{1}(0,0)=1$, $c_{2}(0,0)=c_{3}(0,0)=c_{4}(0,0)=0$. Although this is a generalized result, since it corresponds to the two-dimensional case, the proof similar to the one in Theorem \ref{th1}. For clarity, we present it in subsection \ref{var2}.
\end{remark}

\section{Application examples}\label{appl}
In the following subsections, we present various queueing related examples, the analysis of which result in vector-valued functional equations of the form \eqref{opq}, \eqref{basic0}. Note that the examples that refer to the transient analysis (subsections \ref{mmc}, \ref{fluid}) are more complicated since in order to result in the functional equation we need to apply some Wiener-Hopf arguments. However, this is not the case in the examples related to the stationary analysis; see subsections \ref{mdc}-\ref{inar}.
\subsection{The Markov dependent case}\label{mdc}
In the following, we cope with generalizing the work in \cite[Section 2]{dimitriou2024markov} (also \cite{adan2} in the autoregressive context), by further assuming that the autoregressive parameter depends also on the state of a background Markov chain. In such a case, we will assume that $\alpha_{i}(s)=a_{i}s$, $i=1,\ldots,N$.

Consider a FIFO single-server queue, and let $T_{n}$ $n$th arrival to the system with $T_{1}=0$. Define also $A_{n}=T_{n}-T_{n-1}$, $n=2,3,\ldots$, i.e., is the time between the $n$th and $(n-1)$th arrival. Let $S_{n}$ be the service time of the $n$th arrival, $n\geq 1$. We assume that the inter-arrival and service times are regulated by an irreducible discrete-time Markov chain $\{Y_{n}, n\geq 0\}$ with state space $E=\{1, 2,...,N\}$ and transition probability matrix $P:=(p_{i,j})_{i,j\in E}$. 
Let $\tilde{\pi}:=(\pi_{1},\ldots,\pi_{N})^{T}$ be the stationary distribution of $\{Y_{n};n\geq 0\}$. Let $W_{n}$ the workload in the system just before the $n$th customer arrival. Such an arrival adds $S_n$ work but makes obsolete a fixed fraction $1- a_{i}$ of
the work that is already present in the system, given that $Y_{n}=i\in E$. Recall that $z_{i}^{n}(s):=\mathbb{E}(e^{-sW_{n}}1_{\{Y_{n}=i\}})$, $Re(s)\geq 0$, $i=1,\ldots,N$, $n\geq 0$, and assuming the limit exists, define $z_{i}(s)=\lim_{n\to\infty}z_{i}^{n}(s)$, $i=1,\ldots,N$. Let also $\tilde{Z}(s)=(z_{1}(s),\ldots,z_{N}(s))^{T}$.

The sequences $\{A_n\}_{n\in\mathbb{N}}$ and $\{S_n\}_{n\in\mathbb{N}_{0}}$ are autocorrelated as well as cross-correlated. Assume that for $n\geq 0$, $x,y\geq 0$, $i,j=1,\ldots,N$:
\begin{equation}
    \begin{array}{l}
         P(A_{n+1}\leq x,S_{n}\leq y,Y_{n+1}=j|Y_{n}=i,A_{2},\ldots,A_{n},S_{1},\ldots,S_{n-1},Y_{1},\ldots,Y_{n-1}) \vspace{2mm}\\
        =  P(A_{n+1}\leq x,S_{n}\leq y,Y_{n+1}=j|Y_{n}=i)=B_{i}(y)p_{i,j}(1-e^{-\lambda_{j}x}):=p_{i,j}B_{i}(y)G_{A,j}(x),
    \end{array}\label{mp1}
\end{equation}
where $B_{i}(.)$, $G_{A,j}(.)$ denote the distribution functions of service and interarrival times, given $Y_{n}=i$, $Y_{n+1}=j$, respectively.
Note that $A_{n+1}$, $S_{n}$, $Y_{n+1}$ are independent of the past given $Y_{n}$, and $A_{n+1}$, $S_{n}$ are conditionally independent given $Y_{n}$, $Y_{n+1}$. Let also $B^{*}(s)=diag(\beta_{1}^{*}(s),\ldots,\beta_{N}^{*}(s))$, where $\beta_{i}^{*}(s):=\int_{0}^{\infty}e^{-sy}dB_{i}(y)$, $L(s):=diag(\frac{\lambda_{1}}{\lambda_{1}-s},\ldots,\frac{\lambda_{N}}{\lambda_{N}-s})$, $\Lambda=diag(\lambda_{1},\ldots,\lambda_{N})$.  

 \begin{remark}
    Moreover, an extension to the case where $A_{n+1}|Y_{n+1}=j$ is of phase-type, e.g., a mixed Erlang distribution with cdf
  \begin{displaymath}
   \begin{array}{c}
        G_{A,j}(x):=\sum_{m=1}^{M}q_{m}(1-e^{-\lambda_{j}x}\sum_{l=0}^{m-1}\frac{(\lambda_{j}x)^{l}}{l!}),\,x\geq 0,\end{array}
    \end{displaymath}
    can be handled at a cost of a more complicated expression. 
\end{remark}
    
\begin{theorem}\label{theoremm}
The transforms $Z_{j}(s)$, $j=1,\ldots,N$ satisfy the system
\begin{equation}
    z_{j}(s)=\frac{\lambda_{j}}{\lambda_{j}-s}\sum_{i=1}^{N}p_{i,j}\beta_{i}^{*}(s)z_{i}(a_{i}s)-\frac{s}{\lambda_{j}-s}v_{j},\label{op}
\end{equation}
where $v_{j}:=\sum_{i=1}^{N}p_{i,j}\beta_{i}^{*}(\lambda_{j})z_{i}(a_{i}\lambda_{j})$, $j=1,\ldots,N$. Equivalently, in matrix notation, the transform vector $\tilde{Z}(s)$ satisfies
\begin{equation}
    \tilde{Z}(s)=H(s)\sum_{i=1}^{N}\tilde{P}^{(i)}\tilde{Z}(a_{i}s)+\tilde{V}(s),
    \label{bhj}
\end{equation}
where $\tilde{P}^{(i)}:=(\tilde{P}^{(i)})_{p,q}$, $i,p,q\in E$ is an $N\times N$ matrix, with the $(i,i)$-element $\tilde{P}^{(i)}_{i,i}=1$, and all the other elements $\tilde{P}^{(i)}_{p,q}=0$, $p,q\neq i$. Note that $\sum_{i=1}^{N}\tilde{P}^{(i)}=I$. Moreover, $H(s)=L(s)P^{T}B^{*}(s)$, $\tilde{V}(s):=s(I-L(s))\tilde{v}$, $\tilde{v}:=(v_{1},\ldots,v_{N})^{T}$. 
\end{theorem}
\begin{proof}
From the recursion $W_{n+1}=[a_{i}W_{n}+S_{n}-A_{n+1}]^{+}$ (given $Y_{n}=i$) we obtain the following equation for the transforms $z_{j}^{n+1}(s)$, $j=1,\ldots,N$:
\begin{displaymath}
    \begin{array}{rl}
        z_{j}^{n+1}(s)= &\mathbb{E}(e^{-sW_{n+1}}1_{\{Y_{n+1}=j\}})=\sum_{i=1}^{N}P(Y_{n}=i)E(e^{-sW_{n+1}}1_{\{Y_{n+1}=j\}}|Y_{n}=i)  \vspace{2mm}\\
         =& \sum_{i=1}^{N}P(Y_{n}=i)\mathbb{E}(e^{-s[a_{i}W_{n}+S_{n}-A_{n+1}]^{+}}1_{\{Y_{n+1}=j\}}|Y_{n}=i)\vspace{2mm}\\
         =& \sum_{i=1}^{N}P(Y_{n}=i)p_{i,j}\mathbb{E}(e^{-s[a_{i}W_{n}+S_{n}-A_{n+1}]^{+}}|Y_{n+1}=j,Y_{n}=i)\vspace{2mm}\\
         =&\sum_{i=1}^{N}P(Y_{n}=i)p_{i,j}\left[\mathbb{E}\left(\int_{0}^{a_{i}W_{n}+S_{n}}e^{-s(a_{i}W_{n}+S_{n}-y)}\lambda_{j}e^{-\lambda_{j}y}dy|Y_{n}=i\right)+\mathbb{E}\left(\int_{a_{i}W_{n}+S_{n}}^{\infty}\lambda_{j}e^{-\lambda_{j}y}dy|Y_{n}=i\right)\right]\vspace{2mm}\\
         =&\sum_{i=1}^{N}P(Y_{n}=i)p_{i,j}\mathbb{E}\left(\lambda_{j}e^{-s(a_{i}W_{n}+S_{n})}\int_{0}^{a_{i}W_{n}+S_{n}}e^{-(\lambda_{j}-s)y}dy+\int_{a_{i}W_{n}+S_{n}}^{\infty}\lambda_{j}e^{-\lambda_{j}y}dy|Y_{n}=i\right)\vspace{2mm}\\
         =&\sum_{i=1}^{N}P(Y_{n}=i)p_{i,j}\mathbb{E}\left(\frac{\lambda_{j}}{\lambda_{j}-s}e^{-s(a_{i}W_{n}+S_{n})}(1-e^{-(\lambda_{j}-s)(a_{i}W_{n}+S_{n})})+e^{-\lambda_{j}(a_{i}W_{n}+S_{n})}|Y_{n}=i\right)\vspace{2mm}\\
         =&\sum_{i=1}^{N}P(Y_{n}=i)p_{i,j}\mathbb{E}\left(\frac{\lambda_{j}e^{-s(a_{i}W_{n}+S_{n})}-se^{-\lambda_{j}(a_{i}W_{n}+S_{n})}}{\lambda_{j}-s}|Y_{n}=i\right)\vspace{2mm}\\
         =&\sum_{i=1}^{N}p_{i,j}\left[\frac{\lambda_{j}}{\lambda_{j}-s}z_{i}^{n}(a_{i}s)\beta_{i}^{*}(s)-\frac{s}{\lambda_{j}-s}z_{i}^{n}(a_{i}\lambda_{j})\beta_{i}^{*}(\lambda_{j})\right]\vspace{2mm}\\
         =&\frac{\lambda_{j}}{\lambda_{j}-s}\sum_{i=1}^{N}p_{i,j}z_{i}^{n}(a_{i}s)\beta_{i}^{*}(s)-\frac{s}{\lambda_{j}-s}\sum_{i=1}^{N}p_{i,j}z_{i}^{n}(a_{i}\lambda_{j})\beta_{i}^{*}(\lambda_{j}).
    \end{array}
\end{displaymath}
Letting $n\to\infty$ so that $z_{j}^{n}(s)$ tends to $z_{j}(s)$ we get \eqref{op}. Writing the resulting equations in matrix form we get \eqref{bhj}. 
\end{proof}
\begin{remark}
    It is easy to realize that \eqref{bhj} refers to the matrix analogue of equation (2) in \cite{adan}. Moreover, $\tilde{V}(0)=\tilde{0}$ ($\tilde{0}$ is the $N\times 1$ column vector of zeros), so that $\tilde{1}^{T}\tilde{V}(0)=0$ ($\tilde{1}$ is the $N\times 1$ column vector of ones) and $H(0)=P^{T}$, so that by denoting by $H_{j}(0)$ the $j$th column of matrix $H(0)$, we will have $\tilde{1}H_{j}(0)=\sum_{k=1}^{N}p_{j,k}=1$, $j=1,\ldots,N$. So, Theorem \ref{th1} is a matrix analogue of \cite[Theorem 2]{adan}.
\end{remark}

Note that the form of \eqref{bhj} is exactly the same as in \eqref{opq}, where now $H(s)=L(s)P^{T}B^{*}(s)$, $\tilde{V}(s):=s(I-L(s))\tilde{v}$, $\tilde{v}:=(v_{1},\ldots,v_{N})^{T}$, and the contractions $\alpha_{i}(s)$ have a common fixed-point $a=0$. The solution of the vector-valued functional equation \eqref{bhj} is given in Theorem \ref{th1}. It remains to obtain the vector $\tilde{v}$. This is given by the following proposition.
\begin{proposition}\label{prop1}
    The vector $\tilde{v}$ is the unique solution of the following system of equations:
    \begin{equation}
        v_{j}=e_{j}P^{T}B^{*}(\lambda_{j})\sum_{i=1}^{N}\tilde{P}^{(i)}\tilde{Z}(a_{i}\lambda_{j}),\,j=1,\ldots,N,\label{j1}
    \end{equation}
    where $e_{j}$, an $1\times N$ vector with the $j$th element equal to one and all the others equal to zero, and for $i,j=1,\ldots,N$:
    \begin{displaymath}
        \begin{array}{rl}
    \tilde{Z}(a_{i}\lambda_{j})=&\sum_{k=0}^{\infty}\sum_{i_{1}+\ldots+i_{N}=k}F_{i_{1},\ldots,i_{N}}(a_{i}\lambda_{j})V(\alpha_{i_{1},\ldots,i_{N}}(a_{i}\lambda_{j}))+\lim_{n\to\infty}\sum_{i_{1}+\ldots+i_{N}=n+1}F_{i_{1},\ldots,i_{N}}(a_{i}\lambda_{j})\tilde{\pi}.
    \end{array}
    \end{displaymath}
\end{proposition}

\subsubsection{The case where $A_{n+1}$, $S_{n}$ are conditionally dependent based on the FGM copula}\label{fgm}
Contrary to the case considered above, we now assume that given $Y_{n}$, $Y_{n+1}$, the random variables $A_{n+1}$, $S_{n}$ are dependent based on the FGM copula. Under such an assumption, for $n\geq 0$, $x,y\geq 0$, $i,j=1,\ldots,N$:
\begin{equation}
    \begin{array}{l}
         P(A_{n+1}\leq x,S_{n}\leq y,Y_{n+1}=j|Y_{n}=i,A_{2},\ldots,A_{n},S_{1},\ldots,S_{n-1},Y_{1},\ldots,Y_{n-1}) \vspace{2mm}\\
        =  P(A_{n+1}\leq x,S_{n}\leq y,Y_{n+1}=j|Y_{n}=i)=p_{i,j}F_{S,A|i,j}(y,x),
    \end{array}\label{mp2}
\end{equation}
where, $F_{S,A|i,j}(y,x)$ is the bivariate distribution function of $(S_{n},A_{n+1})$ given $Y_{n}$, $Y_{n+1}$ with marginals $F_{S,i}(y):=B_{i}(y)$, $F_{A,j}(x)$ defined as $F_{S,A|i,j}(y,x)=C_{\Theta}^{FGM}(F_{S,i}(y),F_{A,j}(x))$ for $(y,x)\in \mathbb{R}^{+}\times \mathbb{R}^{+}$. The bivariate density of $(S,A)$ is given by 
\begin{displaymath}
\begin{array}{rl}
     f_{S,A|i,j}(y,x)=&c_{\Theta}^{FGM}(F_{S,i}(y),F_{A,j}(x))f_{S,i}(y)f_{A,j}(x)= f_{S,i}(y)f_{A,j}(x)+\theta_{i,j} g_{i}(y)(2\bar{F}_{A,j}(x)-1)f_{A,j}(x),
\end{array}
\end{displaymath}
where $g_{i}(y):=f_{S,i}(y)(1-2F_{S,i}(y))$ with Laplace transform $g_{i}^{*}(s)=\int_{0}^{\infty}e^{-sy}g_{i}(y)dy$, $\bar{F}_{A,j}(x):=1-F_{A,j}(x)$, $f_{S,i}(y)$, $f_{A,j}(x)$ the densities of $S_{n}|Y_{n}=i$, $A_{n+1}|Y_{n+1}=j$, and $\theta_{i,j}\in[-1,1]$. In our case,
\begin{equation}
    \begin{array}{c}
         f_{S,A|i,j}(y,x)=f_{S,i}(y)\lambda_{j}e^{-\lambda_{j}x}+\theta_{i,j} g_{i}(y)\left[2\lambda_{j}e^{-2\lambda_{j}x}-\lambda_{j}e^{-\lambda_{j}x}\right].
    \end{array}\label{biv1}   
\end{equation}
Our aim is to obtain $z_{j}(s;\Theta)=\mathbb{E}(e^{-sW_{n}}1_{\{Y_{n}=j\}};\Theta)$, where $\Theta:=(\theta_{i,j})_{i,j=1,\ldots,N}$.
\begin{theorem}\label{th2}
    The transforms $z_{j}(s;\Theta)$, $j=1,\ldots,N$, satisfy the following equation:
    \begin{equation}
        \begin{array}{rl}
            z_{j}(s;\Theta)= &\frac{\lambda_{j}}{\lambda_{j}-s}\sum_{i=1}^{N}p_{i,j}\left(\beta_{i}^{*}(s)-\frac{\theta_{i,j} s}{2\lambda_{j}-s}g_{i}(s)\right)z_{i}(a_{i}s;\Theta) \vspace{2mm} \\
             &-s\sum_{i=1}^{N}p_{i,j}\left[\frac{\theta}{2\lambda_{j}-s}g_{i}(2\lambda_{j})z_{i}(2a_{i}\lambda_{j};\Theta)+\frac{\beta_{i}^{*}(\lambda_{j})-\theta g_{i}(\lambda_{j})}{\lambda_{j}-s}z_{i}(a_{i}\lambda_{j};\Theta)\right]. 
        \end{array}\label{fgma}
    \end{equation}
    In matrix terms,
    \begin{equation}
        \tilde{Z}(s;\Theta)=U(s;\Theta)\sum_{i=1}^{N}\tilde{P}^{(i)} \tilde{Z}(a_{i}s;\Theta)+\Tilde{V}(s;\Theta),\label{fefgm}
    \end{equation}
    where now, 
    \begin{displaymath}
        \begin{array}{rl}
            U(s;\Theta):= &L_{1}(s)P^{T}B^{*}(s)+(L_{1}(s)-L_{2}(s))(P^{T}\circ\Theta)G^{*}(s)\vspace{2mm}\\
            =&H(s)+(L_{1}(s)-L_{2}(s))(P^{T}\circ\Theta)G^{*}(s),\vspace{2mm}\\
            \Tilde{V}(s;\Theta)=& (I-L_{1}(s))\tilde{v}^{(1)}+(I-L_{2}(s))\tilde{v}^{(2)},
        \end{array}
    \end{displaymath}
    with $P^{T}\circ\Theta$ denotes the $N\times N$ matrix with $(i,j)$ element equal to $\theta_{i,j}p_{i,j}$ (i.e., the operator "$\circ$" denotes the Hadamard product), $G^{*}(s):=diag(g_{1}^{*}(s),\ldots,g_{N}^{*}(s))$, $L_{k}(s)=diag(\frac{k\lambda_{1}}{k\lambda_{1}-s},\ldots,\frac{k\lambda_{N}}{k\lambda_{N}-s})$, $\tilde{v}^{(k)}:=(v_{1}^{(k)},\ldots,v_{N}^{(k)})^{T}$, $k=1,2,$ where for $j=1,\ldots,N$,
    \begin{displaymath}
   v_{j}^{(1)}:=\sum_{i=1}^{N}p_{i,j} \left(\beta_{i}^{*}(\lambda_{j})-\theta_{i,j} g_{i}^{*}(\lambda_{j}) \right)  z_{i}(a_{i}\lambda_{j};\theta),\,\,\,v_{j}^{(2)}:=\sum_{i=1}^{N}p_{i,j}\theta_{i,j} g_{i}^{*}(2\lambda_{j})z_{i}(2a_{i}\lambda_{j};\theta).
    \end{displaymath} 
    \end{theorem}
The proof of Theorem \ref{th2} is similar to the one in Theorem \ref{theoremm} and further details are omitted. The form of \eqref{fefgm} is the same as that treated in Theorem \ref{th1}, where now the matrix $H(s)$ is replaced by $U(s;\Theta)$ (the form of $U(s;\Theta)$ also indicates the difference in this model with respect to the previous one under the additional dependency based on the FGM copula), and $\tilde{V}(s)$, by $\tilde{V}(s;\Theta)$ in order to indicate the dependency on the the parameters $\theta_{i,j}$. The following lemma summarizes the main result.
    \begin{lemma}\label{thj}
    The vector $\tilde{Z}(s;\Theta)$ is given by
    \begin{equation}
\begin{array}{rl}
    \tilde{Z}(s;\Theta)=&\sum_{k=0}^{\infty}\sum_{i_{1}+\ldots+i_{N}=k}F_{i_{1},\ldots,i_{M}}(s;\Theta)V(\alpha_{i_{1},\ldots,i_{N}}(s);\Theta)+\lim_{n\to\infty}\sum_{i_{1}+\ldots+i_{N}=n+1}F_{i_{1},\ldots,i_{N}}(s;\Theta)\tilde{\pi},
    \end{array}\label{iterbtx}
\end{equation}
where the functions $F_{i_{1},\ldots,i_{N}}(s;\Theta)$ are recursively defined by
\begin{displaymath}
    F_{i_{1},\ldots,i_{N}}(s;\Theta)=\sum_{k=1}^{N}F_{i_{1},\ldots,i_{k}-1,\ldots,i_{N}}(s;\Theta)U(\alpha_{i_{1},\ldots,i_{k}-1,\ldots,i_{N}}(s);\Theta)\tilde{P}^{(k)},
\end{displaymath}
with $F_{0,\ldots,0}(s;\Theta):=I$, and $F_{i_{1},\ldots,i_{N}}(s;\Theta)=O$ (i.e., the zero matrix) if one of the indices equals $-1$.
\end{lemma}
Now it remains to obtain the vectors $\tilde{v}^{(k)}$, $k=1,2,$ i.e., we need obtain $2N$ equations for the $2N$ unknowns $v_{j}^{(k)}$, $j=1,\ldots,N$, $k=1,2$. Setting $s=a_{i}\lambda_{j}$, and $s=2a_{i}\lambda_{j}$, $i,j=1,\ldots,N$ in \eqref{iterbtx} we obtain expressions for $z_{i}(a_{i}\lambda_{j})$, $z_{i}(a_{i}\lambda_{j})$. 
\begin{proposition}
    The vectors $\tilde{v}^{(k)}$, $k=1,2,$ are given as the unique solution of the following system of equations for $j=1,\ldots,N$:
    \begin{equation}
        \begin{array}{rl}
             v_{j}^{(1)}=&e_{j}[P^{T}B^{*}(\lambda_{j})+(P^{T}\circ\Theta)G^{*}(\lambda_{j})]\sum_{i=1}^{N}\tilde{P}^{(i)}\tilde{Z}(a_{i}\lambda_{j};\Theta),  \vspace{2mm}\\
             v_{j}^{(2)}=&e_{j}(P^{T}\circ\Theta)G^{*}(2\lambda_{j})]\sum_{i=1}^{N}\tilde{P}^{(i)}\tilde{Z}(2a_{i}\lambda_{j};\Theta),
        \end{array}
    \end{equation}
    where, for $m=1,2,$
    \begin{displaymath}
        \begin{array}{rl}
          \tilde{Z}(ma_{i}\lambda_{j};\Theta)=&\sum_{k=0}^{\infty}\sum_{i_{1}+\ldots+i_{N}=k}F_{i_{1},\ldots,i_{M}}(ma_{i}\lambda_{j};\Theta)V(\alpha_{i_{1},\ldots,i_{N}}(ma_{i}\lambda_{j});\Theta)\vspace{2mm}\\+&\lim_{n\to\infty}\sum_{i_{1}+\ldots+i_{N}=n+1}F_{i_{1},\ldots,i_{N}}(ma_{i}\lambda_{j};\Theta)\tilde{\pi}.
        \end{array}
    \end{displaymath}
\end{proposition}
\begin{remark}
    Note that for $\Theta=O$, i.e., the independent copula with $\theta_{i,j}=0$, $i,j=1,\ldots,N,$ Lemma \ref{thj} is reduced to Theorem \ref{th1}.
\end{remark}
 \begin{remark}
    An extension to the case where $A_{n+1}|Y_{n+1}=j$ is of phase-type, e.g., a mixed Erlang distribution with cdf
  \begin{displaymath}
   \begin{array}{c}
        G_{A,j}(x):=\sum_{m=1}^{M}q_{m}(1-e^{-\lambda_{j}x}\sum_{l=0}^{m-1}\frac{(\lambda_{j}x)^{l}}{l!}),\,x\geq 0,\end{array}
    \end{displaymath}
    can be handled at a cost of a more complicated expressions. In such a case,
    \begin{displaymath}
        f_{S,A|i,j}(y,x)=\sum_{m=1}^{M}q_{m}e^{-\lambda_{j}x}\frac{\lambda_{j}^{m}x^{m-1}}{(m-1)!}\left(f_{S,i}(y)+\theta_{i,j}g_{i}(y)\left[2\sum_{m=1}^{M}q_{m}e^{-\lambda_{j}x}\sum_{l=0}^{m-1}\frac{\lambda_{j}^{l}}{l!}-1\right]\right).
    \end{displaymath}
\end{remark}

\subsection{A modulated $D_{N}/G/1$ shot-noise queue}\label{shot}
We now focus on the workload at arrival instants of a modulated $D_{N}/G/1$ shot-noise queue, that refers to a single server queue where the server’s speed is workload proportional, i.e., when the workload is $x$, the server’s speed equals $rx$; see \cite{shots} for a recent survey on shot-noise queueing systems, as well as \cite[Section 6]{dimitriou2024markov} that focused on the case where $N=1$. Our system operates as follows: Assume that
the interarrival times $A_1,A_2,\ldots$ are such that $\mathbb{E}(e^{-sA_{n}}1_{\{Y_{n-1}=i\}}) = e^{-st_{i}}$, $i=1,\ldots,N$. There is a single server, and
service requirements of successive customers $S_1, S_2,\ldots$ are i.i.d. random variables. We assume that just before the arrival of the $n$th customer, additional amount of work equal to
$C_n$ is added. This can be explained as random noise caused by the arrival, and may be positive or negative. Let
\begin{displaymath}
    C_{n}=\left\{\begin{array}{ll}
    C_{n}^{+},&\text{ with probability }p, \\
     -C_{n}^{-},&\text{ with probability }q=1-p. 
\end{array}\right.
\end{displaymath}
We further adopt the dynamics in \eqref{mp1}, i.e.,, for $n\geq 0$, $x,y\geq 0$, $i,j=1,\ldots,N$:
\begin{equation*}
    \begin{array}{l}
         P(C_{n+1}\leq x,S_{n}\leq y,Z_{n+1}=j|Z_{n}=i,(C_{r+1},S_{r},Z_{r}),r=0,1,\ldots,n-1) \vspace{2mm}\\
        =  P(C_{n+1}\leq x,S_{n}\leq y,Z_{n+1}=j|Z_{n}=i)=p_{i,j}F_{S,i}(y)G_{C,j}(x),
    \end{array}
\end{equation*}
where $C^{+}|j$ has a general distribution with LST $c_{j}^{*}(s)$, and $C^{-}|j\sim\exp(\nu_{j})$, $j=1,\ldots,N$. Then, if $W_{n}$ is the workload before the $n$th arrival, we are dealing with a modulated stochastic recursion of the form $W_{n+1}=[e^{-rA_{n+1}}(W_{n}+S_{n})+C_{n+1}]^{+}$. Then, for $j=1,\ldots,N,$
\begin{displaymath}
    \begin{array}{rl}
        z_{j}^{n+1}(s)= &\mathbb{E}\left(e^{-sW_{n+1}}1_{\{Y_{n+1}=j\}}\right)=\sum_{i=1}^{N}P(Y_{n}=i) p_{i,j}E\left(e^{-sW_{n+1}}|Y_{n+1}=j,Y_{n}=i\right) \vspace{2mm}\\
         =& \sum_{i=1}^{N}P(Y_{n}=i)p_{i,j} \left[ p\mathbb{E}\left(e^{-s(e^{-rA_{n+1}}(W_{n}+S_{n})+C^{+}_{n+1})}|Y_{n+1}=j,Y_{n}=i\right)\right.\vspace{2mm}\\
         &\left.+q\mathbb{E}\left(e^{-s[e^{-rA_{n+1}}(W_{n}+S_{n})-C^{-}_{n+1}]^{+}}|Y_{n+1}=j,Y_{n}=i\right) \right]\vspace{2mm} \\
         =&\sum_{i=1}^{N}p_{i,j} \left[ pc_{j}^{*}(-s)\beta_{i}^{*}(se^{-rt_{i}})z_{i}^{n}(se^{-rt_{i}})\right.\vspace{2mm}\\
         &\left.+qP(Y_{n}=i)\mathbb{E}\left(\int_{y=0}^{e^{-rA_{n+1}(W_{n}+S_{n})}}e^{-s(e^{-rA_{n+1}}(W_{n}+S_{n})-y)}\nu_{j}e^{-\nu_{j}y}dy\right.\right.\vspace{2mm}\\
         &\left.\left.+\int_{y=e^{-rA_{n+1}}(W_{n}+S_{n})}^{\infty}\nu_{j}e^{-\nu_{j}y}dy|Y_{n}=i\right) \right]\vspace{2mm}\\
         =&\sum_{i=1}^{N}p_{i,j} \left[ pc_{j}^{*}(s)\beta_{i}^{*}(se^{-rt_{i}})z_{i}^{n}(se^{-rt_{i}})\right.\vspace{2mm}\\
         &\left.+qP(Y_{n}=i)\mathbb{E}\left(\frac{\nu_{j}}{\nu_{j}-s}e^{-se^{-rA_{n+1}}(W_{n}+S_{n})}-\frac{s}{\nu_{j}-s}e^{\nu_{j}e^{-rA_{n+1}}(W_{n}+S_{n})}|Y_{n}=i\right)\right]\vspace{2mm}\\
         =&\left(pc_{j}^{*}(s)+q\frac{\nu_{j}}{\nu_{j}-s}\right)\sum_{i=1}^{N}p_{i,j}\beta_{i}^{*}(se^{-rt_{i}})z_{i}^{n}(se^{-rt_{i}})-\frac{sq}{\nu_{j}-s}\sum_{i=1}^{N}p_{i,j}\beta_{i}^{*}(\nu_{j}e^{-rt_{i}})z_{i}^{n}(\nu_{j}e^{-rt_{i}}).
    \end{array}
\end{displaymath}
Letting $n\to\infty$, so that $z_{j}^{n}(s)\to z_{j}(s)$ we have
\begin{equation}
    z_{j}(s)=\left(pc_{j}^{*}(s)+q\frac{\nu_{j}}{\nu_{j}-s}\right)\sum_{i=1}^{N}p_{i,j}\beta_{i}^{*}(se^{-rt_{i}})z_{i}(se^{-rt_{i}})-\frac{sq}{\nu_{j}-s}\sum_{i=1}^{N}p_{i,j}\beta_{i}^{*}(\nu_{j}e^{-rt_{i}})z_{i}(\nu_{j}e^{-rt_{i}}).\label{ol}
\end{equation}
In matrix notation, \eqref{ol} is written as
\begin{equation}
    \tilde{Z}(s)=\tilde{C}(s)P^{T}\sum_{i=1}^{N}\tilde{P}^{(i)}B^{*}(se^{-rt_{i}})\tilde{Z}(se^{-rt_{i}})+\tilde{Q}(s),\label{ol1}
\end{equation}
where $\tilde{P}^{(i)}$ as given in Theorem \ref{th1}, $\tilde{C}(s):=pC(s)+q\widehat{L}(s)$, $C(s):=diag(c_{1}^{*}(s),\ldots,c_{N}^{*}(s))$, $\widehat{L}(s):=diag(\frac{\nu_{1}}{\nu_{1}-s},\ldots,\frac{\nu_{N}}{\nu_{N}-s})$, $\tilde{Q}(s):=q(I-\widehat{L}(s))\tilde{r}$, $\tilde{r}=(r_{1},\ldots,r_{N})$, with $r_{j}=\sum_{i=1}^{N}p_{i,j}\beta_{i}^{*}(\nu_{j}e^{-rt_{i}})z_{i}^{n}(\nu_{j}e^{-rt_{i}})$.

Note that \eqref{ol1} is more complicated with respect to \eqref{bhj}, since the matrix $B^{*}(se^{-rt_{i}})$ is inside the summation, i.e., \eqref{ol1} is written as
\begin{equation}
    \tilde{Z}(s)=H(s)\sum_{i=1}^{N}\tilde{P}^{(i)}B^{*}(a_{i}s)\tilde{Z}(a_{i}s)+\tilde{Q}(s),
    \label{bhja}
\end{equation}
where $a_{i}=e^{-rt_{i}}$, $i=1,\ldots,N$, $H(s):=\tilde{C}(s)P^{T}$. Moreover, the $n$th iterative of $\zeta_{i}(s):=a_{i}s$, i.e., $\zeta_{i}^{(n)}(s)=\zeta_{i}(\zeta_{i}(\ldots\zeta_{i}(s)\ldots))=se^{-rnt_{i}}\to 0$ as $n\to \infty$. Moreover, \eqref{bhja} is slightly different from \eqref{bhj}, since the mappings $\zeta_{i}(s)$, $i=1,\ldots,N,$ are inside the summation. However, $H(e^{-rt_{i}m}s)\to P^{T}$ as $m\to\infty$, $i=1,\ldots,N$, $\tilde{Q}(0)=\tilde{0}$. Note that \eqref{bhja} is not exactly of the form \eqref{opq}, due to the presence of the matrix $B^{*}(a_i s)$ inside the summation. However, it can be treated similarly. Indeed, the mappings $\zeta_i(s)=a_i s$, $i=1,\ldots,N$, are commutative contractions with common fixed point $0$, since $a_i\in(0,1)$ and $\zeta_i^{(n)}(s)\to 0$ as $n\to\infty$. Moreover, $B^{*}(s)$ is analytic and satisfies $B^{*}(0)=I$, while $H(\zeta_{i}^{(n)}(s))\to P^T$ and $\tilde{Q}(0)=\tilde{0}$. Therefore, near the fixed point the recursion reduces to the canonical form of \eqref{opq}, and the additional factor $B^{*}(a_i s)$ constitutes a bounded analytic perturbation that does not affect the convergence arguments. Consequently, the iterative scheme and absolute convergence results of Theorem~\ref{th1} remain valid.

Note that it remains to obtain the values of the vector $\tilde{r}$. This task can be accomplished using steps similar to those in Proposition \ref{prop1}, so further details are omitted.
\begin{remark}
    Note that in this section we assumed that given $Y_{n}$, $Y_{n+1}$, $C_{n+1}$ and $S_{n}$ are conditionally independent. The analysis can be further applied in case we consider additional dependency, e.g., by assuming that $C_{n+1}$ and $S_{n}$ are dependent based on the FGM copula as in subsection \ref{fgm}, or have a (semi-)linear depend structure, e.g., $C_{n+1}=aS_{n}+J_{n}$, where $J_{n}$ independent random variable from $S_{n}$.
\end{remark}
\subsection{A modulated Markovian queue where service time depends on the waiting time}\label{mark}
Consider the following modulated version of a variant of a M/M/1 queue that was investigated in \cite[Section 5]{boxman}; see also \cite[subsection 6.2]{dimitriou2024markov} for a modulated version. In particular, consider a variant of a M/M/1 queue, in which if the waiting time of the $n$th arriving customer equals $W_n$, then her service time equals $[S_{n}-cW_{n}]^{+}$, where $c>0$. The dynamics in \eqref{mp1} are such that for $n\geq 0$, $x,y\geq 0$, $i,j=1,\ldots,N$:
\begin{equation*}
    \begin{array}{l}
         P(A_{n+1}\leq x,S_{n}\leq y,Y_{n+1}=j|Y_{n}=i,A_{2},\ldots,A_{n},S_{1},\ldots,S_{n-1},Y_{1},\ldots,Y_{n-1}) \vspace{2mm}\\
        =  P(A_{n+1}\leq x,S_{n}\leq y,Y_{n+1}=j|Y_{n}=i)=p_{i,j}(1-e^{-\mu_{i}y})(1-e^{-\lambda_{j}x}).
    \end{array}
\end{equation*}
Contrary to the case in \cite[subsection 6.2]{dimitriou2024markov}, we now assume that the service time also depends on the state of the background process.

Using similar arguments as above:
\begin{displaymath}
    \begin{array}{rl}
        z_{j}^{n+1}(s)= &\mathbb{E}\left(e^{-sW_{n+1}}1_{\{Y_{n+1}=j\}}\right)= \sum_{i=1}^{N}P(Y_{n}=i)p_{i,j}\mathbb{E}\left(e^{-s[W_{n}+[S_{n}-cW_{n}]^{+}-A_{n+1}]^{+}}|Y_{n+1}=j,Y_{n}=i\right)\vspace{2mm}\\
        =&\sum_{i=1}^{N}P(Y_{n}=i)p_{i,j}\left[\mathbb{E}\left(\int_{0}^{W_{n}+[S_{n}-cW_{n}]^{+}}e^{-s(W_{n}+[S_{n}-cW_{n}]^{+}-y)}\lambda_{j}e^{-\lambda_{j}y}dy|Y_{n}=i\right)\right.\vspace{2mm}\\&\left.+\mathbb{E}\left(\int_{W_{n}+[S_{n}-cW_{n}]^{+}}^{\infty}\lambda_{j}e^{-\lambda_{j}y}dy|Y_{n}=i\right)\right]\vspace{2mm}\\
         =&\sum_{i=1}^{N}P(Y_{n}=i)p_{i,j}\mathbb{E}\left(\frac{\lambda_{j}e^{-s(W_{n}+[S_{n}-cW_{n}]^{+})}-se^{-\lambda_{j}(W_{n}+[S_{n}-cW_{n}]^{+})}}{\lambda_{j}-s}|Y_{n}=i\right)\vspace{2mm}\\
=&\frac{\lambda_{j}}{\lambda_{j}-s}\sum_{i=1}^{N}p_{i,j}\left(z_{i}^{n}(s)-\frac{s}{\mu_{i}+s}z_{i}^{n}(s+\mu_{i} c)\right)-\frac{s}{\lambda_{j}-s}\sum_{i=1}^{N}p_{i,j}\left(z_{i}^{n}(\lambda_{j})-\frac{\lambda_{j}}{\mu_{i}+\lambda_{j}}z_{i}^{n}(\lambda_{j}+\mu_{i} c)\right).
    \end{array}
\end{displaymath}
As $n\to\infty$, $z^{n}_{j}(s)\to z_{j}(s)$, we have,
\begin{equation*}
    z_{j}(s)=\frac{\lambda_{j}}{\lambda_{j}-s}\sum_{i=1}^{N}p_{i,j}\left[z_{i}(s)-\frac{s}{\mu_{i}+s}z_{i}(s+\mu_{i} c)\right]-\frac{s}{\lambda_{j}-s}\sum_{i=1}^{N}p_{i,j}\left[z_{i}(\lambda_{j})-\frac{\lambda_{j}}{\mu_{i}+\lambda_{j}}z_{i}(\lambda_{j}+\mu_{i} c)\right],
\end{equation*}
or equivalently,
\begin{equation}
\lambda_{j}\sum_{i=1}^{N}p_{i,j}z_{i}(s)+(s-\lambda_{j})z_{j}(s)-s\lambda_{j}\sum_{i=1}^{N}\frac{p_{i,j}}{\mu_{i}+s}z_{i}(s+\mu_{i} c)=sv_{j},
    \label{kop}
\end{equation}
where $v_{j}=\sum_{i=1}^{N}p_{i,j}\left[z_{i}(\lambda_{j})-\frac{\lambda_{j}}{\mu_{i}+\lambda_{j}}z_{i}(\lambda_{j}+\mu_{i} c)\right]$, $j=1,\ldots,N$. 

For $s=0$, \eqref{kop} yields $z_{j}(0)=\sum_{i=1}^{N}p_{i,j}z_{i}(0)$, thus, $\tilde{Z}(0)=\tilde{\pi}$. In matrix terms, \eqref{kop} is rewritten as:
\begin{equation}
    D(s)\tilde{Z}(s)=s\tilde{v}+\Lambda P^{T}\sum_{i=1}^{N}H^{(i)}(s)\tilde{Z}(s+\mu_{i} c),
    \label{kop1}
\end{equation}
where $D(s)=sI-\Lambda(I-P^{T})$, $\tilde{v}=(v_{1},\ldots,v_{N})^{T}$, $H^{(i)}(s)=(I-M(s))\tilde{P}^{(i)}$, $M(s)=diag(\frac{\mu_{1}}{\mu_{1}+s},\ldots,\frac{\mu_{N}}{\mu_{N}+s})$. Note that $H^{(i)}(s)$ is an $N\times N$ matrix with the $(i,i)$ element equal to $\frac{1}{\mu_{i}+s}$, $i=1,2,\ldots,N$, and all other elements equal to zero. Note that $v_{j}=P(W=0;1_{\{Y_{n+1}=j\}})$.
\begin{lemma}
    The matrix $\Lambda(I-P^{T})$ has exactly $N$ eigenvalues $\gamma_{i}$, $i=1,\ldots,N$, with $\gamma_{1}=0$, and $Re(\gamma_{i})>0$, $i=2,\ldots,N$. 
\end{lemma}
\begin{proof}
    Clearly, $s:=\gamma_{1}=0$ is a root of $det(D(s)=0)$, since $P$ is a stochastic matrix. By applying Gersgorin's circle theorem \cite[Th. 1, Section 10.6]{lanc}, every eigenvalue of $\Lambda(I-P^{T})$ lies in at least one of the disks
    \begin{displaymath}
        \{s:|s-\lambda_{i}(1-p_{i,i})|\leq \sum_{k\neq i}|\lambda_{i}p_{k,i}|=\lambda_{i}\sum_{k\neq i}p_{k,i}\}.
    \end{displaymath}
    Therefore, for each $i$, the real part of $\gamma_{i}$ is positive.
\end{proof}

Let, 
\begin{displaymath}
    \zeta_{s}:=\{s:Re(s)\geq 0, Det(D(s))\neq 0\}.
\end{displaymath}
Then, for $s\in \zeta_{s}$,
\begin{equation}
    \tilde{Z}(s)=A(s)\tilde{v}+G(s)\sum_{i=1}^{N}H^{(i)}(s)\tilde{Z}(s+\mu_{i} c),\label{cz}
\end{equation}
where $A(s):=sD^{-1}(s)$, $G(s):=D^{-1}(s)\Lambda P^{T}$. Note that in this case $\alpha_{i}(s)$ are contractions, but not commutative. Iterating \eqref{cz} and having in mind that $\tilde{Z}(s)\to \tilde{0}$ as $s\to\infty$ we finally have the following lemma.
\begin{lemma}
    The vector $\tilde{Z}$ is given by
    \begin{equation}
    \tilde{Z}(s)=\sum_{k=0}^{\infty}\sum_{i_{1}+\ldots+i_{N}=k}L_{i_{1},\ldots,i_{N}}(s)A(\zeta_{i_{1},\ldots,i_{N}}(s))\tilde{v},\label{no}
\end{equation}
where the functions $L_{i_{1},\ldots,i_{N}}(s)$ are derived recursively by,
\begin{displaymath}
    L_{i_{1},\ldots,i_{N}}(s)=\sum_{k=1}^{N}L_{i_{1},\ldots,i_{k}-1,\ldots,i_{N}}(s)G(\zeta_{i_{1},\ldots,i_{k}-1,\ldots,i_{N}}(s))H^{(k)}(\zeta_{i_{1},\ldots,i_{k}-1,\ldots,i_{N}}(s)),
\end{displaymath}
and 
\begin{displaymath}
    \zeta_{i_{1},\ldots,i_{k}-1,\ldots,i_{N}}(s)=\zeta_{1}^{i_{1}}(\zeta_{2}^{i_{3}}(\ldots(\zeta_{N}^{i_{N}}(s))\ldots),
\end{displaymath}
and $\zeta_{k}^{i_{k}}(s)$ is the $k$th iterate of $\zeta_{k}(s)=s+\mu_{k}c$, i.e., $\zeta_{k}^{i_{k}}(s)=s+i_{k}\mu_{k}c$, $k=1,\ldots,N$.
\end{lemma}

Note that $D(s)$ is singular at the eigenvalues of $\Lambda(I-P^{T})$, i.e., $det(D(s))=0$ at $s=\gamma_{l}$, $l=1,\ldots,N$. However, $\tilde{Z}(s)$ is analytic in the half-plane $Re(s)\geq 0$, and thus, the vector $\tilde{v}$  will be derived so that the right-hand side of \eqref{no} is finite at $s=\gamma_{l}$, $l=1,\ldots,N$. Divide \eqref{kop1} with $s$ and denote by $\tilde{y}_{l}$, the left (row) eigenvector of $\Lambda(I-P^{T})$, associated with the eigenvalue $\gamma_{l}$, $l=1,\ldots,N$ ($y_{1}=\mathbf{1}$ is the row eigenvector with all elements equal to 1, corresponding to the eigenvalue $\gamma_{1}=0$). Then, \eqref{kop1} is written as
\begin{equation}
    \tilde{y}_{l}(1-\frac{\gamma_{l}}{s})\tilde{Z}(s)=\tilde{v}+\tilde{y}_{l}\Lambda P^{T}\sum_{i=1}^{N}\tilde{T}^{(i)}(s)\tilde{Z}(s+\mu_{i} c),\,l=1,\ldots,N,\,Re(s)\geq 0,\label{ftu}
\end{equation}
where $\tilde{T}^{(i)}(s):=s^{-1}H^{(i)}(s)=diag(\frac{1}{\mu_{1}+s},\ldots,\frac{1}{\mu_{N}+s})\tilde{P}^{(i)}$, $i=1,\ldots,N$.

Letting $s=\gamma_{l}$, $l=1,\ldots,N$, and using \eqref{no}, we obtain $N$ equations for the derivation of the $N$ elements of $\tilde{v}$:
\begin{equation}
    \tilde{y}_{l}\tilde{v}=1_{\{l=1\}}+\frac{1}{\mu+\gamma_{i}}\tilde{y}_{i}\Lambda P^{T}\sum_{i=1}^{N}\tilde{T}^{(i)}(\gamma_{l})\sum_{k=0}^{\infty}\sum_{i_{1}+\ldots+i_{N}=k}L_{i_{1},\ldots,i_{N}}(\gamma_{l}+\mu_{i}c)A(\zeta_{i_{1},\ldots,i_{N}}(\gamma_{l}+\mu_{i}c))\tilde{v}.
\end{equation}
\subsection{An integer vector-valued reflected autoregressive process}\label{inar}
Consider an integer-valued stochastic process recursion that is described by 
\begin{equation}
    X_{n+1}=[\sum_{k=1}^{X_{n}}U_{k,n}+B_{n}-1]^{+}.
\end{equation}
Such a recursion describes the
number of waiting customers in a generalized M/G/1 queue with impatient customers, just after the beginning of the $n$th service. $X_{n}$ describes that number and $B_{n}$ is the number of customers arriving during the service time of the
$n$th customer. The service times are governed
by a Markov process $Y_n$, $n = 0, 1,\ldots$ that takes values in $E=\{1, 2,..., N\}$. $U_{k,n}$ are i.i.d. Bernoulli distributed random variables with $P(U_{k,n} = 1|Y_{n}=i) = \xi_{n}^{(i)}$, $P(U_{k,n}=0|Y_{n} = i) = 1 -\xi_{n}^{(i)}$. Moreover, we assume that the $\xi_{n}^{(i)}$ are themselves random variables, independent and identically
distributed with $P(\xi_{n}^{(i)}= a_{i,j}|Y_{n}=i) = q_{i,j}$ , $i,j = 1,..., N$, with $a_{i,j}\in(0,1)$. Moreover, set $Q=(q_{i,j})_{i,j=1,\ldots,N}$.

Denote by for $i, j = 1, 2, . . . , N$, $|z|\leq1$,
\begin{displaymath}\begin{array}{rl}
     B_{i,j}(z):=&\mathbb{E}(z^{B_{n}}1_{\{Y_{n+1}=j\}}|Y_{n}=i),
\end{array}
\end{displaymath}
and let $B(z)=(B_{i,j}(z))_{i,j=1,\ldots,N}$. Then,
\begin{equation}
    \begin{array}{rl}
         \mathbb{E}(z^{X_{n+1}}1_{\{Y_{n+1}=j\}})=&\mathbb{E}(z^{[\sum_{k=1}^{X_{n}}U_{k,n}+B_{n}-1]^{+}}1_{\{Y_{n+1}=j\}}) \vspace{2mm} \\
         =& \mathbb{E}(\left(z^{\sum_{k=1}^{X_{n}}U_{k,n}+B_{n}-1}+1-z^{[\sum_{k=1}^{X_{n}}U_{k,n}+B_{n}-1]^{-}}\right)1_{\{Y_{n+1}=j\}})\vspace{2mm} \\
         =& \mathbb{E}(z^{\sum_{k=1}^{X_{n}}U_{k,n}+B_{n}-1}1_{\{Y_{n+1}=j\}})+P(Y_{n+1}=j)-\mathbb{E}(z^{[\sum_{k=1}^{X_{n}}U_{k,n}+B_{n}-1]^{-}}1_{\{Y_{n+1}=j\}}).
    \end{array}\label{b1}
\end{equation}
Now,
\begin{equation}
    \begin{array}{rl}
         \mathbb{E}(z^{\sum_{k=1}^{X_{n}}U_{k,n}+B_{n}-1}1_{\{Y_{n+1}=j\}})=&\frac{1}{z}\sum_{i=1}^{N}\mathbb{E}(z^{\sum_{k=1}^{X_{n}}U_{k,n}+B_{n}}1_{\{Y_{n+1}=j\}}|Y_{n}=i)P(Y_{n}=i) \vspace{2mm} \\
         =&\frac{1}{z}\sum_{i=1}^{N}\mathbb{E}(z^{\sum_{k=1}^{X_{n}}U_{k,n}}|Y_{n}=i)\mathbb{E}(z^{B_{n}}1_{\{Y_{n+1}=j\}}|Y_{n}=i)P(Y_{n}=i)  \vspace{2mm}  \\
         =& \frac{1}{z}\sum_{i=1}^{N}\mathbb{E}(z^{\sum_{k=1}^{X_{n}}U_{k,n}}|Y_{n}=i)B_{i,j}(z)P(Y_{n}=i).
    \end{array}\label{b2}
\end{equation}
Then, tedious but standard calculations yield,
\begin{equation}
    \begin{array}{rl}
\mathbb{E}(z^{\sum_{k=1}^{X_{n}}U_{k,n}}|Y_{n}=i)P(Y_{n}=i)=&\sum_{l=1}^{N}q_{i,l}\mathbb{E}((a_{i,j}(z))^{X_{n}}1_{\{Y_{n}=i\}}),
    \end{array}\label{b3}
\end{equation}
where $a_{i,j}(z):=\bar{a}_{i,j}+a_{i,j}z$, $\bar{a}_{i,j}:=1-a_{i,j}$, $i,j=1,\ldots,N$. Note that $a_{i,j}(z)$, $i,j=1,\ldots,N$, are commutative contraction mappings on the closed unit disk. Moreover,
\begin{equation}
\begin{array}{rl}
      \mathbb{E}(z^{[\sum_{k=1}^{X_{n}}U_{k,n}+B_{n}-1]^{-}}1_{\{Y_{n+1}=j\}})=&\mathbb{E}(1_{\{Y_{n+1}=j\}}1_{\{\sum_{k=1}^{X_{n}}U_{k,n}+B_{n}-1\geq 0\}})+\frac{1}{z}\mathbb{E}(1_{\{Z_{n+1}=j\}}1_{\{\sum_{k=1}^{X_{n}}U_{k,n}+B_{n}-1=-1\}}) \vspace{2mm}  \\
     =&P(Y_{n+1}=j)-(1-\frac{1}{z})\mathbb{E}(1_{\{Y_{n+1}=j\}}1_{\{\sum_{k=1}^{X_{n}}U_{k,n}+B_{n}-1=-1\}}).
\end{array}\label{b4}
\end{equation}
Denoting $f_{j}(z)=\lim_{n\to\infty}\mathbb{E}(z^{X_{n}}1_{\{Y_{n}=j\}})$, $\tilde{F}(z):=(f_{1}(z),\ldots,f_{N}(z))^{T}$, we have the following result.
\begin{theorem}\label{thq}
    The generating functions $f_{j}(z)$, $j=1,\ldots,N$, satisfy the following system
    \begin{equation}
        \begin{array}{rl}
             f_{j}(z)=& \frac{1}{z}\sum_{i=1}^{N}B_{i,j}(z)\sum_{l=1}^{N}q_{i,l}f_{i}(a_{i,l}(z))+(1-\frac{1}{z})q_{-1,j},    \end{array}\label{a11}
    \end{equation}
    or equivalently, in matrix notation
    \begin{equation}
        \tilde{F}(z)=\frac{1}{z}B^{T}(z)\sum_{i=1}^{N}Q^{(i)}\sum_{l=1}^{N}P_{i}^{(l)}\tilde{F}(a_{i,l}(z))+\tilde{K}(z),\label{a12}
    \end{equation}
    where $\tilde{K}(z)=(1-\frac{1}{z})\tilde{q}_{-1}$, $\tilde{q}_{-1}:=(q_{-1,1},\ldots,q_{-1,N})^{T}$, and for $j=1,\ldots,N$
    \begin{displaymath}
        q_{-1,j}=\sum_{i=1}^{N}B_{i,j}(0)\sum_{l=1}^{N}q_{i,l}f_{i}(\bar{a}_{i,l}).
    \end{displaymath}
    Moreover, $Q^{(i)}$ is a $N\times N$ matrix with rows equal to zero except row $i$ that coincides with row $i$ of matrix $Q$, $P_{l}^{(i)}$ is a  $N\times N$ matrix with all entities equal to zero except the $(i,l)$ entity which is equal to one. Note that $\sum_{i=1}^{N}Q^{(i)}\sum_{l=1}^{N}P_{i}^{(l)}=I$.
    
\end{theorem}
\begin{proof}
    Substituting \eqref{b2}-\eqref{b4} in \eqref{b1}, and letting $n\to\infty$ we obtain after tedious calculations in \eqref{a11}. Now multiplying \eqref{a11} with $z$ and then letting $z=0$, we obtain the expression for $q_{-1,j}$. In matrix notation, \eqref{a11} is written as \eqref{a12}.
\end{proof}

Setting $G(z):=\frac{1}{z}B^{T}(z)$ and $T_{i,j}=Q^{(i)}P_{i}^{(j)}$, $i,j=1,\ldots,N$, \eqref{a12} is rewritten as
\begin{equation}
     \tilde{F}(z)=G(z)\sum_{i=1}^{N}\sum_{j=1}^{N}T_{i,j}\tilde{F}(a_{i,j}(z))+\tilde{K}(z),\label{a13}
\end{equation}
Note that the fixed point of the iterates $a_{i,j}(z) = 1-a_{i,j} + a_{i,j}z$ is $z = 1$, and
we have that $\tilde{K}(1) = \textbf{0}$. Thus, $\tilde{F}(z)$ follows from a modification of Theorem \ref{th1}. In particular, let
\begin{displaymath}
    a^{i(1,1),i(1,2),\ldots,(N,N)}(z):=a_{1,1}^{i(1,1)}(a_{1,2}^{i(1,2)}(\ldots a_{N,N}^{i(N,N)}(z)\ldots)),
\end{displaymath}
and $a_{m,l}^{n}(z)$ is defined as the $n$th iterate
of $a_{m,l}(z)$ with $a^{0,0,\ldots,0}(z)=z$. Iterating $n$ times \eqref{a13} yields
\begin{equation}
    \begin{array}{rl}
        \tilde{F}(z)=&\sum_{\sum_{l,m=1}^{N}i(l,m)=n+1}L_{i(1,1),i(1,2),\ldots,(N,N)}(z)\tilde{F}(a^{i(1,1),i(1,2),\ldots,i(N,N)}(z))\vspace{2mm}\\ 
        &+\sum_{k=0}^{n}\sum_{\sum_{l,m=1}^{N}i(l,m)=k}L_{i(1,1),i(1,2),\ldots,i(N,N)}(z)\tilde{K}(a^{i(1,1),i(1,2),\ldots,(N,N)}(z)),
    \end{array}
\end{equation}
where the matrix functions $L_{i(1,1),i(1,2),\ldots,i(N,N)}(z)$ are derived recursively by
\begin{equation}
    L_{i(1,1),i(1,2),\ldots,i(N,N)}(z)=\sum_{u=1}^{N}\sum_{v=1}^{N}L_{i(1,1),\ldots,i(u,v)-1,\ldots,i(N,N)}(z)G(a^{i(1,1),\ldots,i(u,v)-1,\ldots,(N,N)})T_{v,l},
\end{equation}
with $L_{0,0,\ldots,0}(z)=I$. Using similar arguments as in Theorem \ref{th1} we have the following result.
\begin{lemma}
    The vector $\tilde{F}(z)$ is given by
    \begin{equation}
    \begin{array}{rl}
        \tilde{F}(z)=&\lim_{n\to\infty}\sum_{\sum_{l,m=1}^{N}i(l,m)=n+1}L_{i(1,1),i(1,2),\ldots,(N,N)}(z)\tilde{\pi}\vspace{2mm}\\ 
        &+\sum_{k=0}^{\infty}\sum_{\sum_{l,m=1}^{N}i(l,m)=k}L_{i(1,1),i(1,2),\ldots,i(N,N)}(z)\tilde{K}(a^{i(1,1),i(1,2),\ldots,(N,N)}(z)),
    \end{array}\label{solb}
\end{equation}
\end{lemma}

We still need to derive $\tilde{q}_{-1}$. This can be done by substituting $z=\bar{a}_{i,j}$ in \eqref{solb} and substituting the $j$th component of the derived $\tilde{F}(\bar{a}_{i,j})$ in the expression for $q_{-1,j}$ given in Theorem \ref{thq}.

\subsection{The Markov modulated $M/G/1$-type reflected autoregressive process}\label{mmc}
Consider an M/G/1-type queue with Markov modulated arrivals and services. Let $\{X(t);t\geq 0\}$ be the background process that dictates the arrivals and services. $\{X(t);t\geq 0\}$ is a Markov chain on $E=\{1,2,\ldots,N\}$ with infinitesimal generator $Q=(q_{i,j})_{i,j\in E}$, and denote its stationary distribution by $\widehat{\pi}=(\pi_{1},\ldots,\pi_{N})$, i.e., $\widehat{\pi}Q=0$, and $\widehat{\pi} \mathbf{1}=1$, where $\mathbf{1}$ is the $N\times 1$ column vector with all components equal to 1. Denote also by $I$ the $N\times N$ identity matrix, and by $M^{T}$ the transpose of the matrix $M$. 

Customers arrive at time epochs $T_{1},T_{2},\ldots$, $T_{1}=0$, and service times are denoted by $S_{1},S_{2},\ldots$. If $X(t)=i$, arrivals occur according to a Poisson process with rate $\lambda_{i}>0$ and an arriving customer has a service time $S_{i}$, with cumulative distribution function (cdf) $B_{i}(.)$, LST $\beta_{i}^{*}(s):=\int_{0}^{\infty}e^{-st}dB_{i}(t)$, $\bar{b}_{i}=-\beta_{i}^{*\prime}(0)$, $i=1,\ldots,N$, and $B^{*}(s):=diag(\beta_{1}^{*}(s),\ldots,\beta_{N}^{*}(s))$. We assume that given the state of the background process $\{X(t);t\geq 0\}$, $S_{1},S_{2},\ldots$ are independent and independent of the arrival process. Let $A_{n}=T_{n}-T_{n-1}$, $n=2,3,\ldots$, and $Y_{n}=X_{T_{n}}$, $n=1,2,\ldots$. We assume that such an arrival makes obsolete a fixed fraction $1-a_{i}$, given that $Y_{n}=i$, $i\in E$. Denote now by $V_{n}(Y_{n})$ a random variable with support $\{a_{1},\ldots,a_{N}\}$, with $a_{i}\in (0,1)$, and such that $P(V_{n}(Y_{n})=a_{i}|Y_{n}=i)=1$. 

Our focus is on the derivation of the transient distribution (in terms of LSTs) of the process $\{(W_{n},T_{n});n=1,2,\ldots\}$ in which $\{T_n;n=1,2,\ldots\}$ is an increasing time sequence generated by the
input process, and $W_{n}$ denotes the workload in the system just before the $n$th customer arrival that takes place at $T_{n}$. This model generalizes the work in \cite[Section 5]{dimitriou2024markov}, in which the author considered the case where the autoregressive parameter is independent of the state of the background state. 

Assume that $W_{1}=w$ and let for $Re(s)\geq 0$, $Re(\eta)\geq 0$, $|r|<1$,
\begin{displaymath}
    z_{j}^{w}(r,s,\eta):=\sum_{n=1}^{\infty}r^{n}\mathbb{E}(e^{-s W_{n}-\eta T_{n}}1_{\{Y_{n}=j\}})|W_{1}=w),\,j=1,\ldots,N,
\end{displaymath}
and $\tilde{Z}^{w}(r,s,\eta)=(z_{1}^{w}(r,s,\eta),\ldots,z_{N}^{w}(r,s,\eta))^{T}$. Let also
\begin{displaymath}
    A_{i,j}(s)=\mathbb{E}(e^{-s A_{n}}1_{\{Y_{n}=j\}}|Y_{n-1}=i),\,i,j\in E,
\end{displaymath}
and denote by $A(s)$ the $N\times N$ matrix with elements $A_{i,j}(s)$, $Re(s)\geq 0$. Following the lines in \cite[Lemma 2.1]{regte}, we have that $A(s)=M^{-1}(s)\Lambda$ where $\Lambda:=diag(\lambda_{1},\ldots,\lambda_{N})$ and $M(s)=s I+\Lambda-Q$. Let also
\begin{displaymath}
    v_{j}^{w}(r,s,\eta):=\sum_{n=1}^{\infty}r^{n+1}\mathbb{E}((1-e^{-s [V_{n} W_{n}+S_{n}-A_{n+1}]^{-}})e^{-\eta T_{n+1}}1_{\{Y_{n+1}=j\}})|W_{1}=w),\,j=1,\ldots,N,
\end{displaymath}
with $\tilde{V}^{w}(r,s,\eta):=(v_{1}^{w}(r,s,\eta),\ldots,v_{N}^{w}(r,s,\eta))^{T}$, and let $p_{j}:=P(X_{0}=j)$, $j=1,\ldots,N$ with $\widehat{p}:=(p_{1},\ldots,p_{N})^{T}$.
\begin{theorem}\label{th1x}
    For $Re(s)=0$, $Re(\eta)\geq 0$, $|r|<1$,
    \begin{equation}
        z^{w}_{j}(r,s,\eta)-rp_{j}e^{-s w}=r\sum_{i=1}^{N}z^{w}_{i}(r,a_{i}s,\eta)\beta_{i}^{*}(s)A_{i,j}(\eta-s)+v_{j}^{w}(r,s,\eta),\,j=1,2,\ldots,N,
        \label{eq1}
    \end{equation}
    or equivalently, in matrix notation,
    \begin{equation}
        \tilde{Z}^{w}(r,s,\eta)-r\Lambda(M^{T}(\eta-s))^{-1}B^{*}(s)\sum_{i=1}^{N}\tilde{P}^{(i)}\tilde{Z}^{w}(r,a_{i}s,\eta)=re^{-s w}\widehat{p}+\tilde{V}^{w}(r,s,\eta),\label{eq2}
    \end{equation}
where $\tilde{P}^{(i)}:=(\tilde{P}^{(i)})_{p,q}$, $i,p,q\in E$ is an $N\times N$ matrix, with the element $\tilde{P}^{(i)}_{i,i}=1$, and all the other elements $\tilde{P}^{(i)}_{p,q}=0$, $p,q\neq i$. Note that $\sum_{i=1}^{N}\tilde{P}^{(i)}=I$.
\end{theorem}
\begin{proof}
    Using the identity (where $x^{+}:=max(x,0)$, $x^{-}:=min(x,0)$),
    \begin{equation}
        e^{-s x^{+}}+e^{-s x^{-}}=e^{-s x}+1,\label{iden}
    \end{equation}
    we have for $Re(s)=0$, $Re(\eta)\geq 0$, $|r|<1$,
    \begin{displaymath}
        \begin{array}{l}
            \mathbb{E}(e^{-s W_{n+1}-\eta T_{n+1}}1_{\{Y_{n+1}=j\}}|W_{1}=w)= \mathbb{E}(e^{-s [V_{n} W_{n}+S_{n}-A_{n+1}]^{+}-\eta T_{n+1}}1_{\{Y_{n+1}=j\}})|W_{1}=w)\vspace{2mm}  \\
            =  \mathbb{E}((e^{-s [V_{n} W_{n}+S_{n}-A_{n+1}]}+1-e^{-s [V_{n} W_{n}+S_{n}-A_{n+1}]^{-}})e^{-\eta T_{n+1}}1_{\{Y_{n+1}=j\}})|W_{1}=w)\vspace{2mm}\\
            =  \mathbb{E}(e^{-s [V_{n} W_{n}+S_{n}-A_{n+1}]-\eta T_{n+1}}1_{\{Y_{n+1}=j\}})|W_{1}=w)
            +\mathbb{E}((1-e^{-s [V_{n} W_{n}+S_{n}-A_{n+1}]^{-}})e^{-\eta T_{n+1}}1_{\{Y_{n+1}=j\}}|W_{1}=w).
        \end{array}
    \end{displaymath}
    Note that
    \begin{displaymath}
        \begin{array}{r}
            \mathbb{E}(e^{-s[V_{n} W_{n}+S_{n}-A_{n+1}]-\eta T_{n+1}}1_{\{Y_{n+1}=j\}}|W_{1}=w) =\mathbb{E}(e^{-s [V_{n} W_{n}+S_{n}-A_{n+1}]-\eta (A_{n+1}+T_{n})}1_{\{Y_{n+1}=j\}}|W_{1}=w)  \vspace{2mm}\\
            = \mathbb{E}(e^{-s V_{n} W_{n}-\eta T_{n}}e^{-s S_{n}-(\eta-s)A_{n+1}}1_{\{Y_{n+1}=j\}}|W_{1}=w) \vspace{2mm}\\
        =\sum_{i=1}^{N}\mathbb{E}(e^{-s a_{i} W_{n}-\eta T_{n}}1_{\{Y_{n}=i\}}|W_{1}=w)\beta_{i}^{*}(s)A_{i,j}(\eta-s)
        \end{array}
    \end{displaymath}
    Thus, for $Re(s)=0$, $Re(\eta)\geq 0$, $|r|<1$,
    \begin{equation}
        \begin{array}{rl}
         \mathbb{E}(e^{-s W_{n+1}-\eta T_{n+1}}1_{\{Y_{n+1}=j\}}|W_{1}=w)=    & \sum_{i=1}^{N}\mathbb{E}(e^{-s a_{i} W_{n}-\eta T_{n}}1_{\{Y_{n}=i\}}|W_{1}=w)\beta_{i}^{*}(s)A_{i,j}(\eta-s)\vspace{2mm} \\
         &+\mathbb{E}((1-e^{-s [V_{n} W_{n}+S_{n}-A_{n+1}]^{-}})e^{-\eta T_{n+1}}1_{\{Y_{n+1}=j\}}|W_{1}=w).
        \end{array}\label{rty}
    \end{equation}
    Multiplying \eqref{rty} by $r^{n+1}$ and taking the sum of $n=1$ to infinity gives:
    \begin{displaymath}
        z^{w}_{j}(r,s,\eta)-r \mathbb{E}(e^{-s W_{1}-\eta T_{1}}1_{\{Y_{1}=j\}}|W_{1}=w)=r\sum_{i=1}^{N}z^{w}_{i}(r,a_{i}s,\eta)\beta_{i}^{*}(s)A_{i,j}(\eta-s)+v_{j}^{w}(r,s,\eta).
    \end{displaymath}
    Note that $T_{1}=0$ and,
    \begin{displaymath}
        \mathbb{E}(e^{-s W_{1}-\eta T_{1}}1_{\{Y_{1}=j\}}|W_{1}=w)=\mathbb{E}(e^{-s w}1_{\{X_{0}=j\}})=e^{-s w}P(X_{0}=j)=e^{-s w}p_{j}.
    \end{displaymath}
   Substituting back in \eqref{rty}, we obtain the system of Wiener-Hopf equations \eqref{eq1}, which in matrix notation, is equivalent to \eqref{eq2}. 
\end{proof}

The following lemma is taken from \cite{dimitriou2024markov}; see also \cite{regte}.
\begin{lemma}
\begin{enumerate}
    \item The $N$ eigenvalues, say $\nu_{i}$, $i=1\ldots,N$, of $\Lambda-Q^{T}$ all lie in $Re(s)>0$.
    \item The $N$ zeros of $det((\eta-s)I+\Lambda-Q^{T})=0$ for $Re(\eta)\geq 0$, say $\mu_{i}(\eta)$, $i=1,\ldots,N$, are all in $Re(s)>0$ (i.e., For $Re(s)=0$, $Re(\eta)\geq 0$, $det((\eta-s)I+\Lambda-Q^{T})\neq 0$), and such that $\mu_{i}(\eta)=\nu_{i}+\eta$, $i=1,\ldots,N$. 
\end{enumerate}
\end{lemma}
\begin{proof}
Note that $\Lambda-Q^{T}-s I:=S(s)+R(s)$, where $R(s)=diag(\lambda_{1}+q_{1}-s,\ldots,\lambda_{N}+q_{N}-s)$, and $S(s):=(s_{i,j}(s))_{i,j=1,\ldots,N}:=\begin{pmatrix}
0&q_{2,1}&\ldots&q_{N,1}\\
q_{1,2}&0&\ldots&q_{N,2}\\
\vdots&\vdots&\vdots\\
q_{1,N}&q_{2,N}&\ldots&0\\
\end{pmatrix}$. Moreover,
\begin{displaymath}
    |\lambda_{i}+q_{i}-s|\geq \lambda_{i}+q_{i}-|s|>q_{i}=|-q_{i,i}|=\sum_{j\neq i}|q_{i,j}|=\sum_{j=1}^{N}|s_{i,j}(s)|.
\end{displaymath}
Thus, from \cite[Theorem 1, Appendix 2]{smit}, for $Re(s)>0$, the number of zeros of $det(\Lambda-Q^{T}-s I)$ are equal to the number of zeros of $det(R(s))=\prod_{i=1}^{N}(\lambda_{i}+q_{i}-s)$. Assume now that all $\nu_{i}$ are distinct.

Similarly, for $Re(\eta)\geq 0$, $Re(s)=0$, $det(\Lambda-Q^{T}+(\eta-s) I)\neq 0$ (thus, the inverse of $M^{T}(\eta-s)$ exists), and all the zeros of $det(\Lambda-Q^{T}+(\eta-s) I)$ lie in $Re(s)>0$. It is easy to realize that these zeros, say $\mu_{i}(\eta)$, are such that $\mu_{i}(\eta)=\nu_{i}+\eta$, $i=1,\ldots,N$, where $\nu_{i}$ the eigenvalues of $\Lambda-Q^{T}$. 
\end{proof}

Thus,
\begin{displaymath}
    (M^{T}(\eta-s))^{-1}=\frac{1}{\prod_{i=1}^{N}(s-\mu_{i}(\eta))}L(\eta-s),
\end{displaymath}
where $L(\eta-s):=cof(M^{T}(\eta-s))$ is the cofactor matrix of $M^{T}(\eta-s)$. 
\begin{remark}
    Note that $(M^{T}(\eta-s))^{-1}$ can also be written as (see \cite[equation (3.16)]{regte}):
\begin{displaymath}
    (M^{T}(\eta-s))^{-1}=R\ diag(\frac{1}{\mu_{1}(\eta)-s},\ldots,\frac{1}{\mu_{N}(\eta)-s})R^{-1},
\end{displaymath}
where $R$ is the matrix with the $i$th column, say $R_{i}$, $i=1,\ldots,N$, being the right eigenvector of $\Lambda-Q^{T}$ corresponding to the eigenvalue $\nu_{i}$.
\end{remark}

    Let $g_{i,j}(\eta,s):=(\Lambda L(\eta-s)B^{*}(s))_{i,j}$, $i,j=1,\ldots,N$. Then, \eqref{eq2} can be written as:
    \begin{equation}
        \prod_{i=1}^{N}(s-\mu_{i}(\eta))[z_{j}^{w}(r,s,\eta)-re^{-s w}p_{j}]-r\sum_{i=1}^{N}z_{i}^{w}(r,a_{i}s,\eta)g_{i,j}(\eta,s)=\prod_{i=1}^{N}(s-\mu_{i}(\eta))v_{j}^{w}(r,s,\eta).\label{eq3s}
    \end{equation}
    Note that for $|r|<1$, $Re(\eta)\geq 0$, 
\begin{itemize}
\item The left-hand side of \eqref{eq3s} is analytic in $Re(s)>0$, continuous in $Re(s)\geq 0$, and it is also bounded.
    \item The right-hand side of \eqref{eq3s} is analytic in $Re(s)<0$, continuous in $Re(s)\leq 0$, and it is also bounded.
\end{itemize}
Thus, Liouville's theorem \cite[Th. 2.52]{tit}, implies that, in their respective half planes, both the left and the right hand side of \eqref{eq3s} can be rewritten as a polynomial of at most $N$th degree in $s$, dependent of $r$, $\eta$, for large $s$: For $|r|<1$, $Re(\eta)\geq 0$, $Re(s)\geq 0$:
\begin{equation}
        \prod_{i=1}^{N}(s-\mu_{i}(\eta))[z_{j}^{w}(r,s,\eta)-re^{-s w}p_{j}]-r\sum_{i=1}^{N}z_{i}^{w}(r,a_{i}s,\eta)g_{i,j}(\eta,s)=\sum_{i=0}^{N}s^{i}c_{i,j}^{w}(r,\eta).\label{eq3ss}
    \end{equation}
    Note that for $s=0$, \eqref{eq3ss} (having in mind that $G_{i,j}(\eta,0)=A_{i,j}(\eta)$) yields
\begin{equation*}
    (-1)^{N}\prod_{i=1}^{N}\mu_{i}(\eta)[z_{j}^{w}(r,0,\eta)-rp_{j}]-r\sum_{i=1}^{N}z_{i}^{w}(r,0,\eta)A_{i,j}(\eta)=c^{w}_{0,j}(r,\eta).\label{al11}
\end{equation*}
However, from \eqref{eq3s}, for $s=0$, 
\begin{equation*}
    (-1)^{N}\prod_{i=1}^{N}\mu_{i}(\eta)[z_{j}^{w}(r,0,\eta)-rp_{j}]-r\sum_{i=1}^{N}z_{i}^{w}(r,0,\eta)A_{i,j}(\eta)=0,\label{sl2}
\end{equation*}
since $v^{w}_{j}(r,0,\eta)=0$, $j=1,\ldots,N$. Thus, $c_{0,j}^{w}(r,\eta)=0$, so that $C_{0}^{w}(r,\eta)=(c_{0,1}^{w}(r,\eta),\ldots,c_{0,N}^{w}(r,\eta))^{T}=\tilde{0}$.

\begin{remark}
This result may also be derived as follows. Note that, \eqref{eq2} is now written for $|r|<1$, $Re(s)=0$, $Re(\eta)\geq 0$ as:
\begin{equation}
        \prod_{i=1}^{N}(s-\mu_{i}(\eta))[\tilde{Z}^{w}(r,s,\eta)-re^{-s w}\widehat{p}]-r\Lambda L(\eta-s)B^{*}(s)\sum_{i=1}^{N}\tilde{P}^{(i)}\tilde{Z}^{w}(r,a_{i}s,\eta)=\prod_{i=1}^{N}(s-\mu_{i}(\eta))\tilde{V}^{w}(r,s,\eta).\label{eq3}
    \end{equation}

Note that for $|r|<1$, $Re(\eta)\geq 0$, 
\begin{itemize}
\item The left-hand side of \eqref{eq3} is analytic in $Re(s)>0$, continuous in $Re(s)\geq 0$, and it is also bounded.
    \item The right-hand side of \eqref{eq3} is analytic in $Re(s)<0$, continuous in $Re(s)\leq 0$, and it is also bounded since $|E(e^{-s[R_{n} W_{n}+S_{n}-A_{n+1}]^{-}-\eta T_{n+1}}1(Y_{n+1}=j)|W_{1}=w)|\leq 1$, $Re(s)\leq 0$, $Re(\eta)\geq 0$.
\end{itemize}
Thus, by analytic continuation we can define an entire function such that it is equal to the left-hand side of \eqref{eq3} for $Re(s)\geq 0$, and equal to the right-hand side of \eqref{eq3} for $Re(s)\leq 0$ (with $|r|<1$, $Re(\eta)\geq 0$). Hence, by (a variant of) Liouville's theorem \cite{paman} (for vector-valued functions; see also \cite[p. 81, Theorem 3.32]{rudin} or \cite[p. 232, Theorem 9.11.1]{die}, or \cite[p. 113, Theorem 3.12]{allan}) behaves as a polynomial of at most $N$th degree in $s$. Thus, for $Re(\phi)\geq 0$,
\begin{equation}
    \prod_{i=1}^{N}(s-\mu_{i}(\eta))[\tilde{Z}^{w}(r,s,\eta)-re^{-s w}\widehat{p}]-r\Lambda L(\eta-s)B^{*}(s)\sum_{i=1}^{N}\tilde{P}^{(i)}\tilde{Z}^{w}(r,a_{i}s,\eta)=\sum_{i=0}^{N}s^{i}C^{w}_{i}(r,\eta),\label{eq4}
\end{equation}
where $C^{w}_{i}(r,\eta):=(c^{w}_{i,1}(r,\eta),\ldots,c^{w}_{i,N}(r,\eta))^{T}$, $i=0,1,\ldots,N$, the column vectors still have to be determined. 

Note that for $s=0$, \eqref{eq4} yields
\begin{equation}
    (-1)^{N}\prod_{i=1}^{N}\mu_{i}(\eta)[(I-r\Lambda(M^{T}(\eta))^{-1})\tilde{Z}^{w}(r,0,\eta)-r\widehat{p}]=C^{w}_{0}(r,\eta).\label{l11}
\end{equation}
However, from \eqref{eq3}, for $s=0$, 
\begin{equation}
    (-1)^{N}\prod_{i=1}^{N}\mu_{i}(\eta)[(I-r\Lambda(M^{T}(\eta))^{-1})\tilde{Z}^{w}(r,0,\eta)-r\widehat{p}]=\tilde{0},\label{l2}
\end{equation}
since $\tilde{V}(r,0,\eta)=\tilde{0}$, where $\tilde{0}$, is an $N\times 1$ column vector with all components equal to 0. Thus, $C^{w}_{0}(r,\eta)=\tilde{0}$. 


\end{remark}

Therefore, for $Re(s)\geq 0$, $Re(\eta)\geq 0$, $|r|<1$,
\begin{equation} 
    \tilde{Z}^{w}(r,s,\eta)=rK^{w}(s,\eta)\sum_{i=1}^{N}\tilde{P}^{(i)}\tilde{Z}^{w}(r,a_{i}s,\eta)+L^{w}(r,s,\eta),\label{basic}
\end{equation}
where
\begin{displaymath}
    \begin{array}{rl}
       K^{w}(s,\eta):=&\Lambda (M^{T}(\eta-s))^{-1}B^{*}(s)=rA^{T}(\eta-s)B^{*}(s), \vspace{2mm}\\
        L^{w}(r,s,\eta):=  & re^{-s w}\widehat{p}+\frac{1}{\prod_{i=1}^{N}(s-\mu_{i}(\eta))}\sum_{i=1}^{N}s^{i}C_{i}^{w}(r,\eta).
    \end{array}
\end{displaymath}
Iterating \eqref{basic} $n$ times yields
\begin{equation}
\begin{array}{rl}
    \tilde{Z}^{w}(r,s,\eta)=&\sum_{k=0}^{n}\sum_{i_{1}+\ldots+i_{N}=k}F_{i_{1},\ldots,i_{M}}(r,s,\eta)L^{w}(r,\alpha_{i_{1},\ldots,i_{N}}(s),\eta)\vspace{2mm}\\&+\sum_{i_{1}+\ldots+i_{N}=n+1}F_{i_{1},\ldots,i_{N}}(r,s,\eta)\tilde{Z}^{w}(r,\alpha_{i_{1},\ldots,i_{N}}(s),\eta),
    \end{array}\label{iter}
\end{equation}
where $\alpha_{i_{1},\ldots,i_{N}}(s)=\alpha_{1}^{i_{1}}(\alpha_{2}^{i_{2}}(\ldots(\alpha_{N}^{i_{N}}(s))\ldots))$ and $\alpha_{i}^{n}(s)$ denotes the $n$th iterate of $\alpha_{i}(s)=a_{i}s$, with $\alpha_{0,\ldots,0}(s)=s$, and the functions $F_{i_{1},\ldots,i_{N}}(r,s,\eta)$ are recursively defined by
\begin{displaymath}
    F_{i_{1},\ldots,i_{N}}(r,s,\eta)=r\sum_{k=1}^{N}F_{i_{1},\ldots,i_{k}-1,\ldots,i_{N}}(r,s,\eta)K^{w}(\alpha_{i_{1},\ldots,i_{k}-1,\ldots,i_{N}}(s),\eta)\tilde{P}^{(k)},
\end{displaymath}
with $F_{0,\ldots,0}(r,s,\eta):=I$, and $F_{i_{1},\ldots,i_{N}}(r,s,\eta)=O$ (that is, the zero matrix) if one of the indices is equal to $-1$. It is clear that \eqref{basic} is of the form \eqref{basic0}, with $U(s,\eta):=K^{w}(s,\eta)$, and $K(r,s,\eta):=L^{w}(r,s,\eta)$, and where the contractions $\alpha_{i}(s)$ have a common fixed-point $a=0$. Moreover, as $i_{1}+\ldots+i_{N}\to\infty$, $K^{w}(\alpha_{i_{1},\ldots,i_{N}}(s),\eta)\to A^{T}(\eta)$, and $L^{w}(r,\alpha_{i_{1},\ldots,i_{N}}(s),\eta)\to r\widehat{p}$.

 Note also that $\lim_{n\to\infty}\tilde{Z}^{w}(r,a^{n}s,\eta)=\tilde{Z}^{w}(r,0,\eta)$ satisfies for $|r|<1$, $Re(\eta)\geq 0$: 
\begin{displaymath}
    (I-r\Lambda(M^{T}(\eta))^{-1})\tilde{Z}^{w}(r,0,\eta)=r\widehat{p}\Leftrightarrow(I-rA^{T}(\eta))\tilde{Z}^{w}(r,0,\eta)=r\widehat{p} \Leftrightarrow\tilde{Z}^{w}(r,0,\eta)=r(I-rA^{T}(\eta))^{-1}\widehat{p},
    \end{displaymath}
    since the eigenvalues of $rA^{T}(\eta)$ are strictly less than one when $Re(\eta)\geq 0$. So, by applying Theorem \ref{tranth} we have
    \begin{equation}
\begin{array}{rl}
    \tilde{Z}^{w}(r,s,\eta)=&\sum_{n=0}^{\infty}\sum_{i_{1}+\ldots+i_{N}=n}F_{i_{1},\ldots,i_{M}}(r,s,\eta)L^{w}(r,\alpha_{i_{1},\ldots,i_{N}}(s),\eta).
    \end{array}\label{soll}
\end{equation}
    
   Note that in \eqref{soll}, there are still $N^{2}$ unknowns to be determined, i.e., the terms $c_{i,j}^{w}(r,\eta)$, $i,j=1,\ldots,N$. These terms can be derived by applying the following steps: 
\begin{itemize}
    \item  Substituting $s=\mu_{k}(\eta)$, $k=1,\ldots,N$ in \eqref{eq3ss} yields a $N^{2}\times N^{2}$ system of equations for the elements of $C^{w}_{j}(r,\eta)=(c^{w}_{j,1}(r,\eta),\ldots,c^{w}_{j,N}(r,\eta))^{T}$, $j=1,\ldots,N$:
    \begin{equation}
        -r\sum_{i=1}^{N}z_{i}^{w}(r,a_{i}\mu_{k}(\eta),\eta)g_{i,j}(\eta,a_{i}\mu_{k}(\eta))=\sum_{i=1}^{N}\mu_{k}^{i}(\eta)c^{w}_{i,j}(r,\eta),\,j=1,\ldots,N,\label{conss}
    \end{equation}
    where $z_{i}^{w}(r,a_{i}\mu_{k}(\eta),\eta)$, $i=1,\ldots,N$ be the $i$th element of $\tilde{Z}^{w}(r,a_{i}\mu_{k}(\eta),\eta)$.
    \item Substitute $s=\mu_{k}(\eta)$, $k=1,\ldots,N$ in \eqref{soll} to obtain expressions for $z_{i}^{w}(r,a_{i}\mu_{k}(\eta),\eta)$, $i=1,\ldots,N$.
    \item Substituting the resulting expressions of $z_{i}^{w}(r,a_{i}\mu_{k}(\eta),\eta)$, $i=1,\ldots,N$ in \eqref{conss} we obtain an $N^{2}\times N^{2}$ system of equations for the unknowns $c_{i,j}^{w}(r,\eta)$, $i,j=1,\ldots,N$.  
\end{itemize}

\subsection{A Markovian fluid flow model with consumption}\label{fluid}
Consider a general stochastic fluid model with a single infinite capacity
buffer where the buffer content $X(t)$ increases continuously. An external background process $J(t)$, $t\geq 0$ with a finite state space $E=\{1,2,\ldots,N\}$ affects the rate of the buffer increase. In particular, when $J(t)=i$, $i\in E$, the buffer content continuously increases with rate $r_{i}\in(-\infty,+\infty)-\{0\}$ (for convenience, we assume that $r_{i}\neq 0$, $i\in E$). When $J(t)=i$ the process remains for an exponential amount of time with rate $q_{i}$, and then may move to any state $j$. At transition epochs, a fraction $1-a_{i}$ of the buffer content is instantaneously removed (i.e., it is consumed). Let $T_0 = 0$ and $T_1, T_2,\ldots$ be the transition epochs of the process $\{J(t);t\geq 0\}$ with $T_1 > 0$. Let also $W_0 = v$ and for $n = 1, 2,\ldots$, $W_n = X(T_n)$ and $Y_n = J(T_{n}^{-})$ be the state of
$\{J(t);t\geq 0\}$ just before transition epochs.

Let $A_n = T_n - T_{n-1}$ be the inter-jump time, $n = 1, 2,\ldots$, and define $R_n =\sum_{k=1}^{N}r_{k}1_{\{Y_{n}=k\}}$, $V_{n}=\sum_{k=1}^{N}a_{k}1_{\{Y_{n}=k\}}$. We then have the following recursion:
\begin{equation}
    W_{n+1}=[V_{n}W_{n}+R_{n+1}A_{n+1}]^{+}.
\end{equation}

Define for $Re(s)\geq 0$, $Re(\eta)\geq 0$, $|r|<1$, $i,j\in E$
\begin{displaymath}
    \begin{array}{rl}
         z_{i,j}(r,s,\eta,v)=&\sum_{n=1}^{\infty}r^{n}\mathbb{E}(e^{-sW_{n}-\eta T_{n}}1_{\{Y_{n}=j\}}|Y_{1}=i,W_{0}=v),  \vspace{2mm}\\
         z_{i}^{0}(s,\eta,v)=&\mathbb{E}(e^{-sW_{1}-\eta T_{1}}1_{\{Y_{1}=i\}}|Y_{1}=i,W_{0}=v),
    \end{array}
\end{displaymath}
and for $Re(s)\leq 0$,
\begin{displaymath}
    v_{i,j}(r,s,\eta,v)=\sum_{n=1}^{\infty}r^{n+1}\mathbb{E}((1-e^{-s[W_{n}-R_{n+1}A_{n+1}]^{-}})e^{-\eta (T_{n}+A_{n+1})}1_{\{Y_{n}=j\}}|Y_{1}=i,W_{0}=v),
\end{displaymath}
and for $Re(s)=0$,
\begin{displaymath}
    g_{i,j}(s,\eta)=\mathbb{E}(e^{-(sr_{j}+\eta)A_{n+1}}1_{\{Y_{n+1}=j\}}|Y_{n}=i).
\end{displaymath}
Let $Z(r,s,\eta,v)$, $G(s,\eta)$, $V(r,s,\eta,v)$ be $N\times N$ matrices with elements $z_{i,j}(r,s,\eta,v)$, $g_{i,j}(s,\eta)$ and $v_{i,j}(r,s,\eta,v)$, respectively. Let also $Z^{0}(s,\eta,v)=diag(z_{1}^{0}(s,\eta,v),\ldots,z_{N}^{0}(s,\eta,v))$. Note that
\begin{displaymath}
\begin{array}{rl}
     z_{i}^{0}(s,\eta,v)=&E(e^{-sW_{1}-\eta T_{1}}1_{\{Y_{1}=i\}}|Y_{1}=i,W_{0}=v)=\int_{0}^{\infty}e^{-s[v+r_{i}u]^{+}-\eta u}q_{i}e^{-q_{i}u}du.
\end{array}
\end{displaymath}
 For $r_{i}>0$,
 \begin{displaymath}
     [v+r_{i}u]^{+}=\left\{\begin{array}{ll}
          v+ur_{i},&u\geq -v/r_{i}\Leftrightarrow u\geq 0,  \\
          0,&u\leq 0, 
     \end{array}\right.
 \end{displaymath}
 and for $r_{i}<0$,
 \begin{displaymath}
     [v+r_{i}u]^{+}=\left\{\begin{array}{ll}
          v+ur_{i},&u\leq -v/r_{i},  \\
          0,&u\geq -v/r_{i}.
     \end{array}\right.
 \end{displaymath}
 Simple calculations imply that
\begin{displaymath}
    z_{i}^{0}(s,\eta,v)=\left\{\begin{array}{ll}
         e^{-sv}\frac{q_{i}}{q_{i}+sr_{i}+\eta},& r_{i}>0 \\
         \frac{q_{i}}{q_{i}+sr_{i}+\eta}[e^{-sv}-e^{-(\eta+q_{i})v/r_{i}}]+\frac{q_{i}e^{-(\eta+q_{i})v/r_{i}}}{\eta+q_{i}},&r_{i}<0. 
    \end{array}\right.
\end{displaymath}
 Note also that,
\begin{displaymath}
    G(s,\eta)=P\widehat{Q}(\widehat{Q}+\widehat{R}+\eta I)^{-1},
\end{displaymath}
where $P$ is the one-step probability matrix of $\{Y_{n};n\in\mathbb{N}\}$, and $\widehat{Q}:=diag(q_{1},q_{2},\ldots,q_{N})$, $\widehat{R}:=diag(r_{1},r_{2},\ldots,r_{N})$.

Then, we have the following result.
\begin{theorem}
    For $Re(s)=0$, $Re(\eta)\geq 0$, $|r|<1$,
    \begin{equation}
        \begin{array}{r}
             [z_{i,j}(r,s,\eta,v)-\delta_{i,j}z_{i}^{0}(s,\eta,v)](r_{j}s+\eta+q_{j})-r\sum_{k=1}^{N}z_{i,k}(r,a_{k}s,\eta,v)p_{k,j}q_{j}\vspace{2mm}\\=(r_{j}s+\eta+q_{j})v_{i,j}(r,s,\eta,v).
        \end{array}\label{v11}
    \end{equation}
\end{theorem}
\begin{proof}
Using the identity \eqref{iden}, we have for $Re(s)=0$, $Re(\eta)\geq 0$, $|r|<1$,
\begin{displaymath}
        \begin{array}{l}
            \mathbb{E}(e^{-s W_{n+1}-\eta T_{n+1}}1_{\{Y_{n+1}=j\}}|Y_{1}=i,W_{0}=v)= \mathbb{E}(e^{-s [V_{n} W_{n}+R_{n+1}A_{n+1}]^{+}-\eta (T_{n}-A_{n+1})}1_{\{Y_{n+1}=j\}})|Y_{1}=i, W_{0}=v)\vspace{2mm}  \\
            =  \mathbb{E}(e^{-s V_{n} W_{n}-\eta T_{n}-(sR_{n+1}+\eta)A_{n+1}}1_{\{Y_{n+1}=j\}})|Y_{1}=i,W_{0}=v)\vspace{2mm}  \\
            +\mathbb{E}((1-e^{-s [V_{n} W_{n}+R_{n+1}A_{n+1}]^{-}})e^{-\eta (T_{n}+A_{n+1})}1_{\{Y_{n+1}=j\}}|Y_{1}=i,W_{0}=v)\vspace{2mm}  \\
            =\sum_{k=1}^{N}\mathbb{E}(e^{-sa_{k}W_{n}-\eta T_{n}}1_{\{Y_{n}=k\}})|Y_{1}=i,W_{0}=v)\mathbb{E}(e^{-s(R_{n+1}+\eta)A_{n+1}}1_{\{Y_{n+1}=j\}})|Y_{n}=k)\vspace{2mm}  \\
            +\mathbb{E}((1-e^{-s [V_{n} W_{n}+R_{n+1}A_{n+1}]^{-}})e^{-\eta (T_{n}+A_{n+1})}1_{\{Y_{n+1}=j\}}|Y_{1}=i,W_{0}=v)\vspace{2mm}  \\
 =\sum_{k=1}^{N}\mathbb{E}(e^{-sa_{k}W_{n}-\eta T_{n}}1_{\{Y_{n}=k\}})|Y_{1}=i,W_{0}=v)g_{k,j}(s,\eta)\vspace{2mm}  \\
            +\mathbb{E}((1-e^{-s [V_{n} W_{n}+R_{n+1}A_{n+1}]^{-}})e^{-\eta (T_{n}+A_{n+1})}1_{\{Y_{n+1}=j\}}|Y_{1}=i,W_{0}=v).
        \end{array}
    \end{displaymath}
    Multiplying with $r^{n+1}$ and adding for all $n$ we obtain
    \begin{displaymath}
        z_{i,j}(r,s,\eta,v)-r\mathbb{E}(e^{-sW_{1}-\eta T_{1}}1_{\{Y_{1}=j\}}|Y_{1}=i,W_{0}=v)=r\sum_{k=1}^{N}z_{i,k}(r,a_{k}s,\eta,v)g_{k,j}(s,\eta)+v_{i,j}(r,s,\eta,v).
    \end{displaymath}
    Simple computations yield \eqref{v11}.
\end{proof}

Note that for $|r|<1$, $Re(\eta)\geq 0$, 
\begin{itemize}
\item The left-hand side of \eqref{v11} is analytic in $Re(s)>0$, continuous in $Re(s)\geq 0$, and is also bounded.
    \item The right-hand side of \eqref{v11} is analytic in $Re(s)<0$, continuous in $Re(s)\leq 0$, and is also bounded.
\end{itemize}
Liouville's theorem \cite{tit} states that for $Re(s)\geq 0$,
\begin{equation}
    [z_{i,j}(r,s,\eta,v)-\delta_{i,j}z_{i}^{0}(s,\eta,v)](r_{j}s+\eta+q_{j})-r\sum_{k=1}^{N}z_{i,k}(r,a_{k}s,\eta,v)p_{k,j}q_{j}=c_{i,j}^{(0)}(r,\eta,v)+sc_{i,j}^{(1)}(r,\eta,v),
    \label{vb1}
\end{equation}
and for $Re(s)\leq 0$,
\begin{equation}
    (r_{j}s+\eta+q_{j})v_{i,j}(r,s,\eta,v)=c_{i,j}^{(0)}(r,\eta,v)+sc_{i,j}^{(1)}(r,\eta,v).
    \label{vb2}
\end{equation}
For $s=0$, \eqref{vb2} yields $c_{i,j}^{(0)}(r,\eta,v)=0$ since $v_{i,j}(r,0,\eta,v)=0$.

Let $R^{+}=\{j\in E:r_{j}>0\}$, $R^{-}=\{j\in E:r_{j}<0\}$. Note that for $j\in R^{+}$, $r_{j}s+\eta+q_{j}=0$ vanishes at $s:=s_{j}^{+}=-\frac{\eta+q_{j}}{r_{j}}<0$, while for $j\in R^{-}$, $r_{j}s+\eta+q_{j}=0$ vanishes at $s:=s_{j}^{-}=-\frac{\eta+q_{j}}{r_{j}}>0$. Then, using \eqref{vb2}, for $s=s_{j}^{+}$, $j\in R^{+}$
\begin{displaymath}
    (r_{j}s_{j}^{+}+\eta+q_{j})v_{i,j}(r,s_{j}^{+},\eta,v)=s_{j}^{+}c_{i,j}^{(1)}(r,\eta,v),
\end{displaymath}
thus, $c_{i,j}^{(1)}(r,\eta,v)=0$, $i\in E$, $j\in R^{+}$. Setting $s=s_{j}^{-}$, $j\in R^{+}$ in \eqref{vb1} yields,
\begin{equation}
    -r\sum_{k=1}^{N}z_{i,k}(r,a_{k}s_{j}^{-},\eta,v)p_{k,j}q_{j}=s_{j}^{-}c_{i,j}^{(1)}(r,\eta,v).\label{bv1}
\end{equation}
Equation \eqref{bv1} provides a system of equations for $c_{i,j}^{(1)}(r,\eta,v)$, $i\in E$, $j\in R^{-}$. However, we need an expression for the $Z_{i,k}(r,a_{k}s_{j}^{-},\eta,v)$, $i,k\in E$, in the left-hand side of \eqref{bv1}. In a matrix terms, \eqref{vb1} is written as
\begin{equation}
    Z(r,s,\eta,v)=r\sum_{k=1}^{N}Z(r,a_{k}s,\eta,v)\tilde{P}^{(k)}G(s,\eta)+L(r,s,\eta,v),\label{nm1}
\end{equation}
where
\begin{displaymath}
    L(r,s,\eta,v)=rZ^{0}(s,\eta,v)+sC^{(1)}(r,\eta,v)[s\widehat{R}+\eta I+\widehat{Q}]^{-1},
\end{displaymath}
where $C^{(1)}(r,\eta,v)$ an $N\times N$ matrix with $c_{i,j}^{(1)}(r,\eta,v)=0$, $i\in E$, $j\in R^{+}$, and $c_{i,j}^{(1)}(r,\eta,v)$, $i\in E$, $j\in R^{-}$ are obtained by \eqref{bv1}. 

Iterating \eqref{nm1} $n$ times yields
\begin{equation}
\begin{array}{rl}
Z(r,s,\eta,v)=&\sum_{k=0}^{n}\sum_{i_{1}+\ldots+i_{N}=k}F_{i_{1},\ldots,i_{M}}(r,s,\eta,v)L(r,\alpha_{i_{1},\ldots,i_{N}}(s),\eta,v)\vspace{2mm}\\&+\sum_{i_{1}+\ldots+i_{N}=n+1}F_{i_{1},\ldots,i_{N}}(r,s,\eta,v)Z(r,\alpha_{i_{1},\ldots,i_{N}}(s),\eta,v),
    \end{array}\label{iterz}
\end{equation}
where $\alpha_{i_{1},\ldots,i_{N}}(s)=\alpha_{1}^{i_{1}}(\alpha_{2}^{i_{2}}(\ldots(\alpha_{N}^{i_{N}}(s))\ldots))$ and $\alpha_{i}^{n}(s)$ denotes the $n$th iterate of $\alpha_{i}(s)=a_{i}s$, with $a_{0,\ldots,0}(s)=s$, and the functions $F_{i_{1},\ldots,i_{N}}(r,s,\eta,v)$ are recursively defined by
\begin{displaymath}
    F_{i_{1},\ldots,i_{N}}(r,s,\eta)=r\sum_{k=1}^{N}F_{i_{1},\ldots,i_{k}-1,\ldots,i_{N}}(r,s,\eta)\tilde{P}^{(k)}G(\alpha_{i_{1},\ldots,i_{k}-1,\ldots,i_{N}}(s),\eta),
\end{displaymath}
with $F_{0,\ldots,0}(r,s,\eta):=I$, $F_{i_{1},\ldots,i_{N}}(r,s,\eta)=O$ (that is, the zero matrix) if one of the indices is equal to $-1$, and $F_{0,\ldots,0,1,0,\ldots,0}(r,s,\eta):=r\tilde{P}^{(k)}G(s,\eta)$, $k=1,\ldots,N$ (the 1 in the index of $F$ is in the $k$th position). Note that as $i_{1}+\ldots+i_{N}\to\infty$, $G(\alpha_{i_{1},\ldots,i_{N}}(s),\eta)\to Pdiag(q_{1}/(q_{1}+\eta),\ldots,q_{N}/(q_{N}+\eta))$, thus, $||G(\alpha_{i_{1},\ldots,i_{N}}(s),\eta)||\leq ||P||||diag(q_{1}/(q_{1}+\eta),\ldots,q_{N}/(q_{N}+\eta))||\leq 1$, $Re(\eta)\geq 0$. Therefore, $F_{i_{1},\ldots,i_{N}}(r,s,\eta)$ is bounded. Note that the form of \eqref{nm1} is the same as \eqref{basic0} where the contractions $\alpha_{i}(s)$ have a common fixed-point $a=0$. So using Theorem \ref{tranth} we conclude that
\begin{equation}
   Z(r,s,\eta,v)=\sum_{k=0}^{\infty}\sum_{i_{1}+\ldots+i_{N}=k}F_{i_{1},\ldots,i_{M}}(r,s,\eta,v)L(r,\alpha_{i_{1},\ldots,i_{N}}(s),\eta,v).\label{cxa} 
\end{equation}
The expression \eqref{cxa} can be used in \eqref{bv1} to obtain the unknown terms $c_{i,j}^{(1)}(r,\eta,v)$, $i\in E$, $j\in R^{-}$.

\section{Multidimensional extensions}\label{multi}
In the following, we cope with some application examples, the analysis of which lead to multidimensional versions of \eqref{opq}. In particular, we first investigated a modulated ASIP tandem network of two queues with consumption, and then, we focus on a vector-valued reflected autoregressive process. Although the derived functional equations are (more or less) different, the solution machinery is similar to that in \eqref{opq}. In particular, the analysis of an instance of the modulated ASIP tandem network with consumption results in a functional equation of the form \eqref{masip}, where $\alpha_{i}(t):=b_{i}(t)$, $b_{i}\in(0,1)$, $i=1,\ldots,N$, in which we have two coupled recursions; see subsection \eqref{asip}. Its solution is given by Theorem \ref{th-masip-full}. The analysis of the VAR(1) model results in a different vector-valued functional equation, which is far more interesting. However, the solution approach of a special case has similarities with the machinery applied in Section \ref{theory}; see subsection \ref{var2} and Theorem \ref{thvar}.
\subsection{A modulated ASIP tandem network with consumption}\label{asip}
In this section, we consider the Markov-modulated analogue of the model analyzed in \cite{boxkelres}. The authors in \cite{boxkelres} considered a non-modulated ASIP (asymmetric inclusion process) tandem queue, in which the first queue receives a fluid input according to a Lévy subordinator process. Each queue has a gate that opens after
independent, exponentially distributed periods for an infinitesimal amount of time,
allowing the queue content to move to the next queue. In addition, again at independent exponentially distributed instants, a fixed fraction of a queue content is removed
from the system.

Consider a queueing model consisting of two stations, say $Q_1,Q_{2}$ in series. We assume that only $Q_1$ receives an external input, which is a L\'evy subordinator process
$X = \{X(t);t \geq 0\}$. Given the state of the background process $J(t)$, that takes values in $E=\{1,2,\ldots,N\}$, the Laplace exponent of $\{X(t);t\geq 0\}$ is $\phi_{i}(.)$, with $\phi_{i}(0)=0$, i.e., that is,
$\mathbb{E}(e^{-sX(t)}1_{\{J(t)=i\}}) = e^{-\phi_{i}(s)t}$ for $s,t \geq 0$. The process $\{J(t);t\geq 0\}$ jumps to any state $j\in E$ with probability $p_{i,j}$, given that $J(t)=i$. In state $i\in E$, $J(t)$ remains for an exponentially distributed time interval with rate $q_{i}$. Given that $J(t)=i$, $Q_k$ has a gate that is closed except for infinitesimally short gate openings that occur at
independent $exp(\mu_{k,i})$ intervals, $k = 1,2$. At a gate opening of $Q_1$, its content instantaneously moves to $Q_{2}$; at a gate opening of $Q_2$, its content leaves the
system. At independent $exp(\tau_{k,i})$ intervals, a fraction $a_{k,i}$ of the content of station $Q_k$
is instantaneously removed from the system, $k = 1,2$. Let $(Z_{1}(t), Z_{2}(t))$ be the buffer
contents of $(Q_1, Q_2)$ at time $t \geq 0$, with $Z_{1}(0)=0=Z_{2}(0)$, and with LST $f_{i}(t,s_1, s_2)
= \mathbb{E}(e^{-s_{1}Z_{1}(t)-s_{2}Z_{2}(t)}1_{\{J(t)=i\}})$, $i\in E$. Then, for $h\to 0$ we have,
\begin{displaymath}
    \begin{array}{rl}
        f_{i}(t+h,s_1, s_2)=&[1-(\mu_{1,i}+\mu_{2,i}+\tau_{1,i}+\tau_{2,i}+q_{i})h]f_{i}(t,s_1, s_2)+\mu_{1,i}hf_{i}(t,s_2,s_2)+\mu_{2,i}hf_{i}(t,s_1,0)\vspace{2mm}\\&+\tau_{1,i}hf_{i}(t,(1-a_{1,i})s_1, s_2)+\tau_{2,i}hf_{i}(t,s_1,(1-a_{2,i}) s_2)+q_{i}h\sum_{j=1}^{N}p_{i,j}f_{j}(t,s_1, s_2)+o(h).
    \end{array}
\end{displaymath}
For the steady-state case, with $f_{i}(s_1, s_2)$ the LST of the steady-state joint
buffer content distribution, we have
\begin{equation}
    \begin{array}{l}
      (\mu_{1,i}+\mu_{2,i}+\tau_{1,i}+\tau_{2,i}+q_{i})f_{i}(s_1, s_2)=\mu_{1,i}f_{i}(s_2,s_2)+\mu_{2,i}f_{i}(s_1,0)\vspace{2mm}\\+\tau_{1,i}f_{i}((1-a_{1,i})s_1, s_2)+\tau_{2,i}f_{i}(s_1,(1-a_{2,i}) s_2)+q_{i}\sum_{j=1}^{N}p_{i,j}f_{j}(s_1, s_2).
    \end{array}\label{ola}
\end{equation}
Let $\tilde{f}(s_{1},s_{2})=(f_{1}(s_{1},s_{2}),\ldots,f_{N}(s_{1},s_{2}))^{T}$, $M_{k}=diag(\mu_{k,1},\ldots,\mu_{k,N})$, $T_{k}=diag(\tau_{k,1},\ldots,\tau_{k,N})$, $k=1,2,$ $Q=(q_{i,j})_{i,j\in E}$, with $q_{i,j}:=q_{i}p_{i,j}$, $i,j\in E$, $i\neq j$, $q_{i,i}:=q_{i}(p_{i,i}-1)$, $i\in E$, and $\tilde{\phi}(s_{1})=diag(\phi_{1}(s_{1}),\ldots,\phi_{N}(s_{1}))$, then, \eqref{ola} can be rewritten in matrix form as
\begin{equation}
\begin{array}{c}
    (M_{1}+M_{2}+T_{1}+T_{2}+\tilde{\phi}(s_{1})-Q)\tilde{f}(s_{1},s_{2})=M_{1}\tilde{f}(s_{2},s_{2})+M_{2}\tilde{f}(s_{1},0)\vspace{2mm}\\+T_{1}\sum_{i=1}^{N}\tilde{P}^{(i)}\tilde{f}((1-a_{1,i})s_{1},s_{2})+T_{2}\sum_{i=1}^{N}\tilde{P}^{(i)}\tilde{f}(s_{1},(1-a_{2,i})s_{2}).
\end{array}\label{vbm}
\end{equation}
\subsubsection{The buffer content of $Q_{1}$}
Focusing on the marginal content of $Q_{1}$, we have that the LST $\tilde{f}(s):=\tilde{f}(s,0)$, and $\tilde{f}(0)=\tilde{\pi}^{T}$, the stationary vector of the background Markov process $\{J(t);t\geq 0\}$. Therefore, \eqref{vbm} is reduced to
\begin{equation}
    (M_{1}+T_{1}+\tilde{\phi}(s)-Q)\tilde{f}(s)=M_{1}\tilde{\pi}^{T}+T_{1}\sum_{i=1}^{N}\tilde{P}^{(i)}\tilde{f}((1-a_{1,i})s),\,Re(s)\geq 0.\label{vmo}
\end{equation}
Then, we have the following result.
\begin{lemma}\label{roots}
    The equation
    \begin{equation}
        det(M_{1}+T_{1}+\tilde{\phi}(s)-Q)=0,
    \end{equation}
    has exactly $N$ roots, say $y_{1},\ldots,y_{N}$, such that $Re(y_{i})<0$, $i=1,\ldots,N$.
\end{lemma}
\begin{proof}
    This is easily proven by writing first $Q=\bar{Q}-Q_{d}$, where $\bar{Q}=(q_{i}p_{i,j})_{i,j\in E}$, $Q_{d}:=diag(q_{1},\ldots,q_{N})$. Then, $G(s):=M_{1}+T_{1}+\tilde{\phi}(s)-Q=F(s)-\bar{Q}$, where $F(s):=diag(f_{1}(s),\ldots,f_{N}(s))$, $f_{i}(s):=\mu_{1,i}+\tau_{1,i}+q_{i}+\phi_{i}(s)$, $i\in E$. Then, 
    \begin{displaymath}
        |\mu_{1,i}+\tau_{1,i}+q_{i}+\phi_{i}(s)|>\sum_{j=1}^{N}|q_{i}p_{i,j}|=q_{i}.
    \end{displaymath}
    By using \cite[Theorem 11.3]{desmit}, since $det(F(s))=0$ has exactly $N$ roots with a negative real part, is implied that $det(F(s)-\bar{Q})=det(M_{1}+T_{1}+\tilde{\phi}(s)-Q)=0$ has exactly $N$ roots with a negative real part.
\end{proof}

Therefore,
\begin{equation}
  \tilde{f}(s)=\tilde{M}(s)+\tilde{T}(s)\sum_{i=1}^{N}\tilde{P}^{(i)}\tilde{f}((1-a_{1,i})s),\,Re(s)\geq 0,
  \label{solll}
\end{equation}
where, $\tilde{M}(s):=(M_{1}+T_{1}+\tilde{\phi}(s)-Q)^{-1}M_{1}\tilde{\pi}^{T}$, $\tilde{T}(s):=(M_{1}+T_{1}+\tilde{\phi}(s)-Q)^{-1}T_{1}$. Note that \eqref{solll} has exactly the same form as \eqref{opq}, and therefore we can apply the results in Theorem \ref{th1}.

The iteration of \eqref{solll} yields the following:
\begin{equation}
    \tilde{f}(s)=\sum_{j=0}^{\infty}\sum_{i_{1}+i_{2}+\ldots+i_{N}=j}F_{i_{1},\ldots,i_{N}}(s)\tilde{M}(\alpha_{i_{1},\ldots,i_{N}}(s)),\label{solllt}
\end{equation}
where now, $\alpha_{i_{1},\ldots,i_{N}}(s)=\alpha_{1}^{i_{1}}(a_{2}^{i_{2}}(\ldots(\alpha_{N}^{i_{N}}(s))\ldots))$ and $\alpha_{i}^{n}(s)$ denotes the $n$th iterate of $\alpha_{i}(s)=(1-a_{1,i})s$, with $a_{0,\ldots,0}(s)=s$, and the functions $F_{i_{1},\ldots,i_{N}}(s)$ are recursively defined by
\begin{displaymath}
    F_{i_{1},\ldots,i_{N}}(s)=\sum_{k=1}^{N}F_{i_{1},\ldots,i_{k}-1,\ldots,i_{N}}(s)\tilde{T}(\alpha_{i_{1},\ldots,i_{k}-1,\ldots,i_{N}}(s))\tilde{P}^{(k)},
\end{displaymath}
with $F_{0,\ldots,0}(s):=I$, and $F_{i_{1},\ldots,i_{N}}(s)=O$ (i.e., the zero matrix) if one of the indices equals $-1$. Note that \eqref{solllt} is derived by having in mind that as $n\to\infty$, $\alpha_{i_{1},\ldots,i_{N}}(s)\to 0$, and $\tilde{T}(a_{i_{1},\ldots,i_{N}}(s))\to (M_{1}+T_{1}-Q)^{-1}T_{1}$, and $\tilde{M}(\alpha_{i_{1},\ldots,i_{N}}(s))\to (M_{1}+T_{1}-Q)^{-1}M_{1}\tilde{\pi}^{T}$. Now, by iterating \eqref{solll} $n$ times yields,
\begin{equation}
    \tilde{f}(s)=\sum_{k=0}^{n}\sum_{i_{1}+\ldots+i_{N}=k}F_{i_{1},\ldots,i_{M}}(s)\tilde{M}(\alpha_{i_{1},\ldots,i_{N}}(s))+\sum_{i_{1}+\ldots+i_{N}=n+1}F_{i_{1},\ldots,i_{N}}(s)\tilde{f}(\alpha_{i_{1},\ldots,i_{N}}(s)),
   \label{itertb}
\end{equation}
Note that $F_{i_{1},\ldots,i_{N}}(s)$ is written as the sum of the product of matrices, where each has elements that are strictly smaller than one. This implies that the second term on the right-hand side of \eqref{itertb} is a convergent matrix as $n\to\infty$. Thus, we come up with \eqref{solllt}. 
\begin{remark}
    The expression for the buffer content $Q_{1}$ is simplified when $a_{1,i}=0$, and when $a_{1,i}=1$, $i=1,\ldots,N$. For $a_{1,i}=0$, $i=1,\ldots,N$, there is no consumption (so we can take $\tau_{1,i}=0$, thus, $T_{1}=O$) from \eqref{vmo} we find
    \begin{equation}
        \tilde{f}(s)=(M_{1}+\tilde{\phi}(s)-Q)^{-1}M_{1}\tilde{\pi}^{T}.\label{sc}
    \end{equation}
    When $a_{1,i}=1$, $i=1,\ldots,N$,
    \begin{equation*}
        \tilde{f}(s)=(M_{1}+T_{1}+\tilde{\phi}(s)-Q)^{-1}(M_{1}+T_{1})\tilde{\pi}^{T},
    \end{equation*}
\end{remark}
\subsubsection{Two variants for the Markov-modulated two-queue ASIP model}
We now focus on obtaining the LST of the joint buffer content distribution of the two-queue ASIP model for the two cases: a) $a_{1,i}\in(0,1)$, $a_{2,i}=0$, $i=1,\ldots,N$, b) $a_{2,i}\in(0,1)$, $a_{1,i}=0$, $i=1,\ldots,N$.
\subsubsection{The case $a_{1,i}\in(0,1)$, $a_{2,i}=0$, $i=1,\ldots,N$}
Let us now consider for \eqref{vbm} $a_{2,i}=0$, $a_{1,i}\in (0,1)$, $i\in E$. Then we have
\begin{equation*}
    (M_{1}+M_{2}+T_{1}+\tilde{\phi}(s_{1})-Q)\tilde{f}(s_{1},s_{2})=M_{1}\tilde{f}(s_{2},s_{2})+M_{2}\tilde{f}(s_{1},0)+T_{1}\sum_{i=1}^{N}\tilde{P}^{(i)}\tilde{f}((1-a_{1,i})s_{1},s_{2}),
\end{equation*}
or equivalently to
\begin{equation}
\begin{array}{c}
    \tilde{f}(s_{1},s_{2})=N_{1}(s_{1})\tilde{f}(s_{2},s_{2})+N_{2}(s_{1})\tilde{f}(s_{1},0)+T_{1,1}(s_{1})\sum_{i=1}^{N}\tilde{P}^{(i)}\tilde{f}((1-a_{1,i})s_{1},s_{2}),
\end{array}\label{vbm1}
\end{equation}
where
\begin{displaymath}
    \begin{array}{rl}
       N_{i}(s_{1}):=& (M_{1}+M_{2}+T_{1}+\tilde{\phi}(s_{1})-Q)^{-1}M_{i}, \,i=1,2,\\
         T_{1,1}(s_{1}):=& (M_{1}+M_{2}+T_{1}+\tilde{\phi}(s_{1})-Q)^{-1}T_{1}.
    \end{array}
\end{displaymath}
\begin{remark}
    One can easily show (as in Lemma \ref{roots}) that $det(M_{1}+M_{2}+T_{1}+\tilde{\phi}(s_{1})-Q)=0$ has $N$ roots with negative real parts, thus, for $Re(s)\geq 0$, $(M_{1}+M_{2}+T_{1}+\tilde{\phi}(s_{1})-Q)^{-1}$ exits.
\end{remark}
Set $L(s_{1},s_{2}):=N_{1}(s_{1})\tilde{f}(s_{2},s_{2})+N_{2}(s_{1})\tilde{f}(s_{1},0)$ and \eqref{vbm1} is rewritten for $Re(s_{1})\geq 0,Re(s_{2})\geq 0$ as
\begin{equation}
    \tilde{f}(s_{1},s_{2})=L(s_{1},s_{2})+T_{1,1}(s_{1})\sum_{i=1}^{N}\tilde{P}^{(i)}\tilde{f}((1-a_{1,i})s_{1},s_{2}).
\label{vbm2}
\end{equation}
Note that for fixed $s_{2}$ such that $Re(s_{2})\geq 0$, \eqref{vbm2} has the same form as \eqref{opq}, and therefore we apply the results in Theorem \ref{th1}. Note that as $n\to\infty$, $\alpha_{i_{1},\ldots,i_{N}}(s_{1})\to 0$, $N_{i}(\alpha_{i_{1},\ldots,i_{N}}(s_{1}))\to (M_{1}+M_{2}+T_{1}-Q)^{-1}M_{i}$, i.e., the elements of the limiting matrix are all bounded by a number that is smaller than 1. Therefore, under similar arguments as in \eqref{solllt},
\begin{equation}
    \tilde{f}(s_{1},s_{2})=\sum_{j=0}^{\infty}\sum_{i_{1}+i_{2}+\ldots+i_{N}=j}\tilde{F}_{i_{1},\ldots,i_{N}}(s_{1},s_{2})L(\alpha_{i_{1},\ldots,i_{N}}(s_{1}),s_{2}),\label{so}
\end{equation}
where now, $\alpha_{i_{1},\ldots,i_{N}}(s_{1})=\alpha_{1}^{i_{1}}(\alpha_{2}^{i_{2}}(\ldots(\alpha_{N}^{i_{N}}(s_{1}))\ldots))$ and $a_{i}^{n}(s_{1})$ denotes the $n$th iterate of $\alpha_{i}(s_{1})=(1-a_{1,i})s_{1}$, with $a_{0,\ldots,0}(s_{1})=s_{1}$, and the functions $\tilde{F}_{i_{1},\ldots,i_{N}}(s_{1},s_{2})$ are recursively defined by
\begin{displaymath}
    \tilde{F}_{i_{1},\ldots,i_{N}}(s_{1},s_{2})=\sum_{k=1}^{N}\tilde{F}_{i_{1},\ldots,i_{k}-1,\ldots,i_{N}}(s_{1},s_{2})T_{1,1}(\alpha_{i_{1},\ldots,i_{k}-1,\ldots,i_{N}}(s_{1}))\tilde{P}^{(k)},
\end{displaymath}
with $\tilde{F}_{0,\ldots,0}(s_{1},s_{2}):=I$, and $\tilde{F}_{0,\ldots,0}(s_{1}),s_{2})=O$ (i.e., the zero matrix) if one of the indices equals $-1$. 

Note that the vector terms $L(\alpha_{i_{1},\ldots,i_{N}}(s_{1}),s_{2})$ contain expressions of the vectors $\tilde{f}(s_{1},0)$, and $\tilde{f}(s_{2},s_{2})$. These terms are obtained as follows: 1) the vector $\tilde{f}(s_{1},0)$ has already been derived in \eqref{solllt}. Thus, $\tilde{f}(\alpha_{i_{1},\ldots,i_{N}}(s_{1}),0)$ that is needed in $L(\alpha_{i_{1},\ldots,i_{N}}(s_{1}),s_{2})$ in \eqref{so} is already known. 2) the vector term $\tilde{f}(s_{2},s_{2})$ is derived by setting $s_{1}=s_{2}$ in \eqref{so}, and obtain after simple calculations:
\begin{displaymath}
\begin{array}{l}
    [I-\sum_{j=0}^{\infty}\sum_{i_{1}+i_{2}+\ldots+i_{N}=j}\tilde{F}_{i_{1},\ldots,i_{N}}(s_{2},s_{2})N_{1}(\alpha_{i_{1},\ldots,i_{N}}(s_{2}))]\tilde{f}(s_{2},s_{2})\vspace{2mm}\\=\sum_{j=0}^{\infty}\sum_{i_{1}+i_{2}+\ldots+i_{N}=j}\tilde{F}_{i_{1},\ldots,i_{N}}(s_{2},s_{2})N_{2}(\alpha_{i_{1},\ldots,i_{N}}(s_{2}))\tilde{f}(\alpha_{i_{1},\ldots,i_{N}}(s_{2}),0).
    \end{array}
\end{displaymath}
Provided that the $I-\sum_{j=0}^{\infty}\sum_{i_{1}+i_{2}+\ldots+i_{N}=j}\tilde{F}_{i_{1},\ldots,i_{N}}(s_{2},s_{2})N_{1}(\alpha_{i_{1},\ldots,i_{N}}(s_{2}))$ is invertible, we can have an expression for $\tilde{f}(s_{2},s_{2})$. Therefore, we have obtained $\tilde{f}(s_{1},s_{2})$ through \eqref{so}.

\subsubsection{The case $a_{1,i}=0$, $i=1,\ldots,N$}
In this case, first note that $\tilde{f}(s_{1},0)$ is given by \eqref{so}. Then, \eqref{vbm} is written as
\begin{equation}
    \tilde{f}(s_{1},s_{2})=\widehat{M}_{1}(s_{1})\tilde{f}(s_{2},s_{2})+\widehat{M}_{2}(s_{1})\tilde{f}(s_{1},0)+\widehat{T}_{2}(s_{1})\sum_{i=1}^{N}\tilde{P}^{(i)}\tilde{f}(s_{1},(1-a_{2,i})s_{2}),
\label{vb1x}
\end{equation}
where $\widehat{M}_{i}(s_{1}):=(M_{1}+M_{2}+T_{2}+\tilde{\phi}(s_{1})-Q)^{-1}M_{i}$, $i=1,2$ and $\widehat{T}_{2}(s_{1}):=(M_{1}+M_{2}+T_{2}+\tilde{\phi}(s_{1})-Q)^{-1}T_{2}$.
\begin{remark}
    One can easily show (as in Lemma \ref{roots}) that $det(M_{1}+M_{2}+T_{2}+\tilde{\phi}(s_{1})-Q)=0$ has $N$ roots with negative real parts, thus, for $Re(s)\geq 0$, $(M_{1}+M_{2}+T_{2}+\tilde{\phi}(s_{1})-Q)^{-1}$ exists.
\end{remark}
Substituting $s_{1}=s_{2}$ in \eqref{vb1x} we obtain
\begin{equation*}
    (I-\widehat{M}_{1}(s_{2}))\tilde{f}(s_{2},s_{2})=\widehat{M}_{2}(s_{1})\tilde{f}(s_{2},0)+\widehat{T}_{2}(s_{2})\sum_{i=1}^{N}\tilde{P}^{(i)}\tilde{f}(s_{2},(1-a_{2,i})s_{2}),
\end{equation*}
or equivalently,
\begin{equation}
    \tilde{f}(s_{2},s_{2})=\widehat{F}(s_{2})\tilde{f}(s_{2},0)+\widehat{G}(s_{2})\sum_{i=1}^{N}\tilde{P}^{(i)}\tilde{f}(s_{2},(1-a_{2,i})s_{2}),
    \label{tr}
\end{equation}
where $\widehat{F}(s_{2}):= (I-\widehat{M}_{1}(s_{2}))^{-1}\widehat{M}_{2}(s_{2})$, $\widehat{G}(s_{2}):= (I-\widehat{M}_{1}(s_{2}))^{-1}\widehat{T}_{2}(s_{2})$. Note that by setting $s=s_{2}$ in \eqref{sc} we can get $\tilde{f}(s_{2},0)$. Substituting the derived expression in \eqref{tr} we get a functional equation for $\tilde{f}(s_{2},s_{2})$:
\begin{equation}
    \tilde{f}(s_{2},s_{2})=\widehat{H}(s_{2})+\widehat{G}(s_{2})\sum_{i=1}^{N}\tilde{P}^{(i)}\tilde{f}(s_{2},(1-a_{2,i})s_{2}),
    \label{tr1}
\end{equation}
where $\widehat{H}(s_{2}):=\widehat{F}(s_{2})(M_{1}+\tilde{\phi}(s_{2})-Q)^{-1}M_{1}\tilde{\pi}^{T}$. Substituting \eqref{tr1} back in \eqref{vb1x} we come up with the following relation
\begin{equation}
\begin{array}{rl}
    \tilde{f}(s_{1},s_{2})=&\widehat{M}_{1}(s_{1})\widehat{H}(s_{2})+\widehat{M}_{1}(s_{1})G(s_{2})\sum_{i=1}^{N}\tilde{P}^{(i)}\tilde{f}(s_{2},(1-a_{2,i})s_{2})\vspace{2mm}\\& +\widehat{M}_{2}(s_{1})(M_{1}+\tilde{\phi}(s_{1})-Q)^{-1}M_{1}\tilde{\pi}^{T}+\widehat{T}_{2}(s_{1})\sum_{i=1}^{N}\tilde{P}^{(i)}\tilde{f}(s_{1},(1-a_{2,i})s_{2}).
    \end{array}\label{w1}
\end{equation}
Replacing $s_{1}$ with $s$, $s_{2}$ with $t$, and writing $b_{i}:=1-a_{2,i}$, $i=1,\ldots,N$, \eqref{w1} is rewritten as
\begin{equation}
    \tilde{f}(s,t)=K(s,t)+R(s,t)\sum_{i=1}^{N}\tilde{P}^{(i)}\tilde{f}(t,b_{i}t)+\widehat{T}_{2}(s)\sum_{i=1}^{N}\tilde{P}^{(i)}\tilde{f}(s,b_{i}t),\label{lab}
\end{equation}
where,
\begin{displaymath}
    \begin{array}{rl}
       K(s,t):=&\widehat{M}_{1}(s)\widehat{H}(t)+\widehat{M}_{2}(s)(M_{1}+\tilde{\phi}(s)-Q)^{-1}M_{1}\tilde{\pi}^{T},\vspace{2mm}\\
       R(s,t):=&\widehat{M}_{1}(s)\widehat{G}(t).
    \end{array}
\end{displaymath}
Note that \eqref{lab} has the form of \eqref{masip}. As we stated in the introduction, although \eqref{masip} has a different form compared to \eqref{opq}, its solution method follows the same machinery. Iterating \eqref{lab} $n -1$ times results in 
\begin{equation}
\begin{array}{rl}
    \tilde{f}_{n}(s,t)=&K_{n}(s,t)+\sum_{k=1}^{n}\sum_{i_{1}+\ldots+i_{N}=k-1}\sum_{j_{1}+\ldots+j_{N}=n,,j_{l}\geq i_{l}}G_{\tilde{i},\tilde{j}}(s,t)\tilde{f}(b_{i_{1},\ldots,i_{N}}(t),b_{j_{1},\ldots,j_{N}}(t))\vspace{2mm}\\&+\sum_{i_{1}+\ldots+i_{N}=n}F_{i_{1},\ldots,i_{N}}(s)\tilde{f}(s,b_{i_{1},\ldots,i_{N}}(t)),
    \end{array}\label{appr}
\end{equation}
where, $b_{i_{1},\ldots,i_{N}}(s)=b_{1}^{i_{1}}(b_{2}^{i_{2}}(\ldots(b_{N}^{i_{N}}(s))\ldots))$ and $b_{i}^{n}(s)$ denotes the $n$th iterate of $b_{i}(s)=(1-a_{2,i})s$, with $b_{0,\ldots,0}(s)=s$ $j_{l}\geq i_{l}$, $l=1,\ldots,N$, and $\tilde{i}:=(i_{1},\ldots,i_{N})$,
\begin{displaymath}
\begin{array}{rl}
K_{n}(s,t)=&K_{n-1}(s,t)+\sum_{k=1}^{n-1}\sum_{i_{1}+\ldots+i_{N}=k-1}\sum_{j_{1}+\ldots+j_{N}=n-1,j_{l}\geq i_{l}}G_{\tilde{i},\tilde{j}}(s,t)K(b_{i_{1},\ldots,i_{N}}(t),b_{j_{1},\ldots,j_{N}}(t))\vspace{2mm}\\
&+\sum_{i_{1}+\ldots+i_{N}=n-1}F_{i_{1},\ldots,i_{N}}(s)K(s,b_{i_{1},\ldots,i_{N}}(t)),\,n=2,3,\ldots,\vspace{2mm}\\
F_{i_{1},\ldots,i_{N}}(s)=&\sum_{k=1}^{N}F_{i_{1},\ldots,i_{k}-1,\ldots,i_{N}}(s)T_{1}(s)\tilde{P}^{(k)},
    \end{array}
\end{displaymath}
with $F_{0,\ldots,0}(s)=I$, and for $j_{1}+\ldots+j_{N}=n$,
\begin{displaymath}
    \begin{array}{rl}
        G_{\tilde{i},\tilde{j}}(s,t)= &\left[\sum_{i_{1}+\ldots+i_{N}\leq n-1,j_{l}\geq i_{l}}\sum_{k=1}^{N}G_{\tilde{i},\tilde{j}-\tilde{1}_{k}}(s,t)G_{\tilde{0},\tilde{1}_{k}}(b_{\tilde{i}}(t),b_{\tilde{j}-1_{k}}(t),t))+F_{\tilde{i}}(s)G_{\tilde{0},\tilde{j}-\tilde{i}}(s,b_{\tilde{i}}(t)) \right]1_{\{i_{1}+\ldots+i_{N}=n-1\}} \vspace{2mm}\\
         &+ \sum_{j_{1}+\ldots+j_{N}= n-1,j_{l}\geq i_{l}}\sum_{k=1}^{N}G_{\tilde{i},\tilde{j}-\tilde{1}_{k}}(s,t)F_{\tilde{1}_{k}}(b_{\tilde{i}}(t))1_{\{i_{1}+\ldots+i_{N}<n-1\}},
    \end{array}
\end{displaymath}
where $\tilde{1}_{k}$ is an $1\times N$ vector with 1 in $k$th position and all other elements equal to 0.

In case $\mu_{2,i}>0$, one can approximate $\tilde{f}(s,t)$ arbitrarily closely by $K_{n}(s, t)$, by taking $n$ sufficiently large. This is because the second and third terms in the righthand side of \eqref{appr} become arbitrarily small as $n\to\infty$. 

The functional equation \eqref{lab} fits within the framework of Theorem \ref{th-masip-full}. 
Indeed, the recursion involves evaluations of $\tilde{f}$ in transformed arguments 
of the form $(s,t)\mapsto (t,b_i t)$ and $(s,t)\mapsto (s,b_i t)$, where $b_i\in(0,1)$. 
Therefore, along any iteration, the second argument contracts geometrically to $0$, 
while the first argument is eventually mapped to the second through the mixed 
recursion, and therefore also converges to $0$. Therefore, 
the solution of \eqref{lab} can be obtained by Theorem \ref{th-masip-full}), 
with absolute convergence ensured under a condition of the form 
$\max\{\|R(0,0)\|_1,\|\widehat{T}_2(0)\|_1\} b^*<1$, where $b^{*}:=max\{b_1,\ldots,b_{N}\}<1$.

\subsection{On a reflected VAR(1) process}\label{var2}
We now consider the a reflected vector autoregressive process (VAR(1)) (see e.g., \cite{kesten,glas}, \cite[Chapter 4]{bura}) that is described by the following vector-valued recursion:
\begin{equation}
    \tilde{Z}_{n+1}=[A\tilde{Z}_{n}+\tilde{S}_{n}-\tilde{A}_{n+1}]^{+},\label{var}
\end{equation}
where $A$ is an $N\times N$ matrix with elements $a_{i,j}$, $i,j=1,\ldots,N$, such that $a_{i,j}\in[0,1)$, and $\tilde{Z}_{n}:=(Z_{1,n},\ldots,Z_{N,n})^{T}$. Let also $\tilde{S}_{n}:=(S_{1,n},\ldots,S_{N,n})^{T}$, $\tilde{A}_{n}:=(A_{1,n},\ldots,A_{N,n})^{T}$, and assume that $S_{j,n}$, $A_{j,n}$ are series of i.i.d. random variables, independent of anything else. From \eqref{var} is readily seen that (element-wise form) 
\begin{equation}
    Z_{j,n+1}=[\sum_{k=1}^{N}a_{j,k}Z_{k,n}+S_{j,n}-A_{j,n+1}]^{+},\,j=1,\ldots,N.
\end{equation}
\subsubsection{The vector of the marginal Laplace-Stieltjes transforms}
Let $F^{n}_{j}(s):=\mathbb{E}(e^{-sF_{j,n}})$ be the LST of the distribution of a random variable $F_{j,n}$, and $\tilde{Z}^{n}(s)=(z_{1}^{n}(s),\ldots,z_{N}^{n}(s))^{T}$. Then,
\begin{equation}
    \begin{array}{rl}
        \mathbb{E}(e^{-sZ_{j,n+1}})=&\mathbb{E}(e^{-s(\sum_{k=1}^{N}a_{j,k}Z_{k,n}+S_{j,n}-A_{j,n+1})}+1-e^{-s[\sum_{k=1}^{N}a_{j,k}Z_{k,n}+S_{j,n}-A_{j,n+1}]^{-}})  \vspace{2mm}\\
         =& \mathbb{E}(e^{-s\sum_{k=1}^{N}a_{j,k}Z_{k,n}})S_{j}(s)A_{j}(-s)+1-u_{j,n}^{-}(s),
    \end{array}\label{vba}
    \end{equation}
where
\begin{displaymath}
    u_{j,n}^{-}(s)=\mathbb{E}(e^{-s[\sum_{k=1}^{N}a_{j,k}Z_{k,n}+S_{j,n}-A_{j,n+1}]^{-}}).
\end{displaymath}
By assuming that for each $n$, $Z_{j,n}$, $j=1,\ldots,N,$ are independent, then $z^{n}(s_{1},\ldots,s_{N}):=\mathbb{E}(e^{-\sum_{k=1}^{N}Z_{k,n}s_{k}})=\prod_{k=1}^{N}z^{n}_{k}(s_{k})$. Now, by letting $n\to\infty$, \eqref{vba} is rewritten for $j=1,\ldots,N$, $Re(s)=0$, as
\begin{equation}
    z_{j}(s)=S_{j}(s)A_{j}(-s)\prod_{k=1}^{N}z_{k}(a_{j,k}s)+1-u_{j}^{-}(s).\label{pl}
\end{equation}
Assume now that $A_{j,n}\sim exp(\lambda_{j})$, $j=1,\ldots,N$. Then, \eqref{pl} is written for $Re(s)=0$ as
\begin{equation}
    (\lambda_{j}-s)z_{j}(s)-\lambda_{j}S_{j}(s)\prod_{k=1}^{N}z_{k}(a_{j,k}s)=(\lambda_{j}-s)(1-u_{j}^{-}(s)).\label{pl1}
\end{equation}
Then,
\begin{itemize}
\item The left-hand side of \eqref{pl1} is analytic in $Re(s)>0$, continuous in $Re(s)\geq 0$, and is also bounded.
    \item The right-hand side of \eqref{pl1} is analytic in $Re(s)<0$, continuous in $Re(s)\leq 0$, and is also bounded.
\end{itemize}
Liouville's theorem \cite{tit} states that for $Re(s)\geq 0$, $j=1,\ldots,N,$
\begin{equation}
    (\lambda_{j}-s)z_{j}(s)-\lambda_{j}S_{j}(s)\prod_{k=1}^{N}z_{k}(a_{j,k}s)=c_{0,j}+sc_{1,j}.\label{pl2}
\end{equation}
For $s=0$, since $z_{j}(0)=1$, \eqref{pl2} implies that $c_{0,j}=0$, $j=1,\ldots,N$. Moreover, setting $s=\lambda_{j}$ in \eqref{pl2} we found that
\begin{equation}
    c_{1,j}=-S_{j}(\lambda_{j})\prod_{l_{1}=1}^{N}z_{l_{1}}(a_{j,k}\lambda_{j}),\,j=1,\ldots,N,
\end{equation}
and
\begin{equation}
    z_{j}(s)=k_{j}(s)\prod_{l_{1}=1}^{N}z_{l_{1}}(a_{j,l_{1}}s)+l_{j}(s).\label{pl3}
\end{equation}
where
\begin{displaymath}
    k_{j}(s):=\frac{\lambda_{j}}{\lambda_{j}-s}S_{j}(s),\,l_{j}(s):=-\frac{s}{\lambda_{j}-s}c_{1,j}.
\end{displaymath}
In matrix notation (i.e., the vector of the marginal LSTs), \eqref{pl3} is written as 
\begin{equation}
    \tilde{Z}(s)=K(s)\sum_{i=1}^{N}\prod_{j=1}^{N}\circ P_{i,j}\tilde{Z}(a_{i,j}s)+\tilde{L}(s),\label{nmk}
\end{equation}
where,
\begin{equation}
    \prod_{j=1}^{N}\circ P_{i,j}\tilde{Z}(a_{i,j}s):=P_{i,1}\tilde{Z}(a_{i,1}s)\circ P_{i,2}\tilde{Z}(a_{i,2}s)\circ\ldots\circ P_{i,N}\tilde{Z}(a_{i,N}s),\label{iterq}
\end{equation}
with "$\circ$" denotes the Hadamard product, $P_{i,j}$ be an $N\times N$ matrix with the $(i,j)$ element is equal to 1, and all others equal to 0, and $K(s):=diag(k_{1}(s),\ldots,k_{N}(s))$, $L(s):=diag(l_{1}(s),\ldots,l_{N}(s))$. Note that \eqref{nmk}, although seems to has a similar form as the one in \eqref{opq}, it is more complicated since in any iterating step we have to appropriately substitute for any $\tilde{Z}(a_{i,j}s)$, $i,j=1,\ldots,N$, thus rapidly increasing the number of terms. 

Note that although \eqref{nmk} involves Hadamard products and multiple indices, 
it retains the key structural properties that are required in Theorem \ref{th1}. More precisely, the arguments contract to zero, and the multiplicative term satisfies $K(0)=I$, and $\tilde{L}(0)=\tilde{0}$. Moreover, since the Hadamard product acts componentwise, it does not 
introduce any interaction between different coordinates. In particular, 
for any matrix norm (e.g., $\|A\|_{\max} := \max_{i,j} \{|A_{i,j}|\}$), we have $\|A \circ B\| \le \|A\| \,\|B\|$, so that the Hadamard product preserves boundedness of the iterates. 

Furthermore, after $k$ recursions, the argument should be 
$s_k = a_{i_1,j_1}\ldots a_{i_k,j_k}s$ ($s_0=s$), and therefore $|s_k|\le (a^*)^k|s|$ for $a^*:=\max_{i,j} \{a_{i,j}\}<1$. 
Since $l_j(s)=O(s)$ near $0$, each occurrence of $\tilde{L}(s_k)$ introduces 
a factor of order $(a^*)^k$, while the multiplicative factors involving 
$K(s_k)$ remain uniformly bounded. Consequently, each term in the expansion 
is bounded by a geometric sequence of order $(a^*)^k$, and the Hadamard 
product does not affect this decay.
Consequently, the associated series converges absolutely, and the solution 
can be obtained by an analogous iterative procedure. So, Theorem \ref{th1} can be adapted to that case.
\begin{remark}
Consider the special case where $A=diag(a_{1},\ldots,a_{N})$. Then, 
\begin{equation}
    Z_{j,n+1}=[a_{j}Z_{j,n}+S_{j,n}-A_{j,n+1}]^{+},\,j=1,\ldots,N,
\end{equation}
which corresponds to the following functional equation for the marginal LST
\begin{equation}
    z_{j}(s)=z_{j}(a_{j}s)S_{j}(s)A_{j}(-s)+1-u_{j}^{-}(s).
\end{equation}
Assuming that $A_{j,n}\sim exp(\lambda_{j})$, so that $A_{j}(s)=\frac{\lambda_{j}}{\lambda_{j}+s}$ $j=1,\ldots,N$, we have that for $Re(s)=0$,
\begin{displaymath}
    (\lambda_{j}-s)z_{j}(s)-z_{j}(a_{j}s)S_{j}(s)\lambda_{j}=(\lambda_{j}-s)(1-u_{j}^{-}(s)).
\end{displaymath}
Liouville's theorem \cite{tit} implies that for $Re(s)\geq 0$,
\begin{displaymath}
    (\lambda_{j}-s)z_{j}(s)-z_{j}(a_{j}s)S_{j}(s)\lambda_{j}=c_{0,j}+c_{1,j}s.
\end{displaymath}
For $s=0$, we have $c_{0,j}=0$, $j=1,\ldots,N$. For $s=\lambda_{j}$, we have $c_{1,j}=-z_{j}(a_{j}\lambda_{j})S_{j}(\lambda_{j})\lambda_{j}$. Thus, for $Re(s)\geq 0$,
\begin{equation}
    z_{j}(s)=z_{j}(a_{j}s)S_{j}(s)\frac{\lambda_{j}}{\lambda_{j}-s}-\frac{sc_{1,j}}{\lambda_{j}-s}.\label{scalar}
\end{equation}
Note that the diagonal form of $A$ reduces the analysis for the vector of marginal LSTs to the scalar version investigated in \cite{box1}. In particular, the solution of \eqref{scalar} is
\begin{equation}
    z_{j}(s)=c_{1,j}\sum_{m=0}^{\infty}\frac{a_{j}^{m}s}{\lambda_{j}-a_{j}^{m}s}\prod_{k=0}^{m-1}\frac{\lambda_{j}S_{j}^{*}(a_{j}^{k}s)}{\lambda_{j}-a_{j}^{k}s}+\prod_{m=0}^{\infty}\frac{\lambda_{j}S_{j}^{*}(a_{j}^{m}s)}{\lambda_{j}-a_{j}^{m}s}.\label{sc1}
\end{equation}
For $s=a_{j}\lambda_{j}$ in \eqref{sc1} we find an expression for $z_{j}(a_{j}\lambda_{j})$, and substituting in $c_{1,j}=-z_{j}(a_{j}\lambda_{j})S_{j}(\lambda_{j})\lambda_{j}$, all $c_{1,j}$, $j=1,2,\ldots,N$ are known. Then, $z_{j}(s)$ are given explicitly in \eqref{sc1}.

In matrix notation, \eqref{scalar} is written as 
\begin{equation}
    \tilde{Z}(s)=S(s)A(-s)\sum_{k=1}^{N}\tilde{P}^{(k)}\tilde{Z}(a_{k}s)+\tilde{V}(s),\label{var1}
\end{equation}
where $S(s)=diag(S_{1}(s),\ldots,S_{N}(s))$, $\Lambda=diag(\lambda_{1},\ldots,\lambda_{N})$, $\tilde{V}(s):=(I-A(-s))\tilde{c}$, $A(s):=diag(A_{1}(s),\ldots,A_{N}(s))$, $\tilde{c}=(c_{1,1},\ldots,c_{1,N})^{T}$. Note that \eqref{var1} has the same form as \eqref{bhj}, where $H(s)$ is equal now to $S(s)A(-s)$, however, its diagonal form considerably simplifies the analysis (note that $H(0)=I$, $\tilde{V}(0)=\tilde{0}$).
\end{remark}
\subsubsection{The joint Laplace-Stieltjes transform}
Note that in order to come up with the expression in \eqref{pl}, we assumed that for each $n$, $Z_{j,n}$, $j=1,\ldots,N$ are independent. That assumption, although it creates a somehow more complicated vector-valued functional equation in \eqref{nmk} (see also \eqref{var1}), can be treated similarly to those in the previous sections. In the following, we consider the joint LST of the stationary vector, drop the assumption of independence, and to enhance readability, we focus on the case $N=2$. Thus, we aim to obtain the joint LST of the stationary vector $(Z_{1},Z_{2})^{T}$, i.e., $f(s_{1},s_{2}):=\mathbb{E}(e^{-s_{1}Z_{1}-s_{2}Z_{2}})$. Then,
\begin{equation}
    \begin{array}{rl}
        f(s_{1},s_{2}):=&\mathbb{E}(e^{-s_{1}Z_{1}-s_{2}Z_{2}})=E(e^{-s_{1}[\sum_{j=1}^{2}a_{1,j}Z_{j}+S_{1}-A_{1}]^{+}-s_{2}[\sum_{j=1}^{2}a_{2,j}Z_{j}+S_{2}-A_{2}]^{+}})\vspace{2mm}  \\
         =& \mathbb{E}(\left(e^{-s_{1}[\sum_{j=1}^{2}a_{1,j}Z_{j}+S_{1}-A_{1}]}+1-e^{-s_{1}[\sum_{j=1}^{2}a_{1,j}Z_{j}+S_{1}-A_{1}]^{-}}\right)\vspace{2mm}\\&\times\left(e^{-s_{2}[\sum_{j=1}^{2}a_{2,j}Z_{j}+S_{2}-A_{2}]}+1-e^{-s_{2}[\sum_{j=1}^{2}a_{2,j}Z_{j}+S_{2}-A_{2}]^{+}}\right))\vspace{2mm}\\
         =&\mathbb{E}(e^{-s_{1}[\sum_{j=1}^{2}a_{1,j}Z_{j}+S_{1}-A_{1}]-s_{2}[\sum_{j=1}^{2}a_{2,j}Z_{j}+S_{2}-A_{2}]})+\mathbb{E}(e^{-s_{1}[\sum_{j=1}^{2}a_{1,j}Z_{j}+S_{1}-A_{1}]})\vspace{2mm}\\
         &-\mathbb{E}(e^{-s_{1}[\sum_{j=1}^{2}a_{1,j}Z_{j}+S_{1}-A_{1}]-s_{2}[\sum_{j=1}^{2}a_{2,j}Z_{j}+S_{2}-A_{2}]^{-}})+\mathbb{E}(e^{-s_{2}[\sum_{j=1}^{2}a_{2,j}Z_{j}+S_{2}-A_{2}]})\vspace{2mm}\\
         &+1-\mathbb{E}(e^{-s_{2}[\sum_{j=1}^{2}a_{2,j}Z_{j}+S_{2}-A_{2}]^{-}})-\mathbb{E}(e^{-s_{1}[\sum_{j=1}^{2}a_{1,j}Z_{j}+S_{1}-A_{1}]^{-}-s_{2}[\sum_{j=1}^{2}a_{2,j}Z_{j}+S_{2}-A_{2}]})\vspace{2mm}\\
         &-\mathbb{E}(e^{-s_{1}[\sum_{j=1}^{2}a_{1,j}Z_{j}+S_{1}-A_{1}]^{-}})+\mathbb{E}(e^{-s_{1}[\sum_{j=1}^{2}a_{1,j}Z_{j}+S_{1}-A_{1}]^{-}-s_{2}[\sum_{j=1}^{2}a_{2,j}Z_{j}+S_{2}-A_{2}]^{-}}).
    \end{array}\label{fund}
\end{equation}
Now,
\begin{displaymath}
    \begin{array}{rl}
        \mathbb{E}(e^{-s_{1}[\sum_{j=1}^{2}a_{1,j}Z_{j}+S_{1}-A_{1}]-s_{2}[\sum_{j=1}^{2}a_{2,j}Z_{j}+S_{2}-A_{2}]})= & \mathbb{E}(e^{-Z_{1}\sum_{j=1}^{2}a_{j,1}s_{j}-Z_{2}\sum_{j=1}^{2}a_{j,2}s_{j}})\frac{\lambda_{1}\lambda_{2}S_{1}(s_{1})S_{2}(s_{2})}{(\lambda_{1}-s_{1})(\lambda_{2}-s_{2})}\vspace{2mm}\\
        =&\frac{\lambda_{1}\lambda_{2}S_{1}(s_{1})S_{2}(s_{2})}{(\lambda_{1}-s_{1})(\lambda_{2}-s_{2})}f(a_{1,1}s_{1}+a_{2,1}s_{2},a_{1,2}s_{1}+a_{2,2}s_{2}),\vspace{2mm}\\
        \mathbb{E}(e^{-s_{k}[\sum_{j=1}^{2}a_{k,j}Z_{j}+S_{k}-A_{k}]})= &\frac{\lambda_{k}S_{k}(s_{k})}{\lambda_{k}-s}f(a_{k,1}s_{k},a_{k,2}s_{k}),\,k=1,2.
    \end{array}
\end{displaymath}
Note that for $k=1,2,$
\begin{displaymath}
    \begin{array}{rl}
         e^{-s_{k}[\sum_{j=1}^{2}a_{k,j}Z_{j}+S_{k}-A_{k}]^{-}}=&\int_{y=0}^{\sum_{j=1}^{2}a_{k,j}Z_{j}+S_{k}}\lambda_{k}e^{-\lambda_{k}y}dy+\int_{y=\sum_{j=1}^{2}a_{k,j}Z_{j}+S_{k}}^{\infty}e^{-s_{k}[\sum_{j=1}^{2}a_{k,j}Z_{j}+S_{k}-y]}\lambda_{k}e^{-\lambda_{k}y}dy  \vspace{2mm}\\
         =&1-e^{-\lambda_{k}[\sum_{j=1}^{2}a_{k,j}Z_{j}+S_{k}]}+\lambda_{k}e^{-s_{k}[\sum_{j=1}^{2}a_{k,j}Z_{j}+S_{k}]} \int_{y=\sum_{j=1}^{2}a_{k,j}Z_{j}+S_{k}}^{\infty}e^{-(\lambda_{k}-s_{k})y}dy\vspace{2mm}\\
         =&1+\frac{s_{k}}{\lambda_{k}-s_{k}}e^{-\lambda_{k}[\sum_{j=1}^{2}a_{k,j}Z_{j}+S_{k}]}.
    \end{array}
\end{displaymath}
Therefore,
\begin{displaymath}
    \mathbb{E}(e^{-s_{k}[\sum_{j=1}^{2}a_{k,j}Z_{j}+S_{k}-A_{k}]^{-}})=1+\frac{s_{k}S_{k}(\lambda_{k})}{\lambda_{k}-s_{k}}f(\lambda_{k}a_{k,1},\lambda_{k}a_{k,2}),\,k=1,2,
\end{displaymath}
and
\begin{displaymath}
    \begin{array}{r}
        \mathbb{E}(e^{-s_{1}[\sum_{j=1}^{2}a_{1,j}Z_{j}+S_{1}-A_{1}]-s_{2}[\sum_{j=1}^{2}a_{2,j}Z_{j}+S_{2}-A_{2}]^{-}})= \mathbb{E}(e^{-s_{1}[\sum_{j=1}^{2}a_{1,j}Z_{j}+S_{1}-A_{1}]}(1+\frac{s_{2}}{\lambda_{2}-s_2}e^{-\lambda_{2}[\sum_{j=1}^{2}a_{2,j}Z_{j}+S_{2}]})) \vspace{2mm} \\
         =\mathbb{E}(e^{-s_{1}[\sum_{j=1}^{2}a_{1,j}Z_{j}+S_{1}-A_{1}]})+ \frac{s_{2}}{\lambda_{2}-s_{2}}E(e^{-s_{1}[\sum_{j=1}^{2}a_{1,j}Z_{j}+S_{1}-A_{1}]-\lambda_{2}[\sum_{j=1}^{2}a_{2,j}Z_{j}+S_{2}]})\vspace{2mm}\\
         =\frac{\lambda_{1}S_{1}(s_{1})}{\lambda_{1}-s_{1}}f(a_{1,1}s_{1},a_{1,2}s_{1})+\frac{\lambda_{1}s_{2}S_{2}(\lambda_{2})S_{1}(s_{1})}{(\lambda_{1}-s_{1})(\lambda_{2}-s_{2})}f(s_{1}a_{1,1}+\lambda_{2}a_{2,1},s_{1}a_{1,2}+\lambda_{2}a_{2,2}),
    \end{array}
\end{displaymath}
and similarly,
\begin{displaymath}
    \begin{array}{l}
        \mathbb{E}(e^{-s_{1}[\sum_{j=1}^{2}a_{1,j}Z_{j}+S_{1}-A_{1}]^{-}-s_{2}[\sum_{j=1}^{2}a_{2,j}Z_{j}+S_{2}-A_{2}]})\vspace{2mm}\\
         =\frac{\lambda_{2}S_{2}(s_{2})}{\lambda_{2}-s_{2}}f(a_{2,1}s_{2},a_{2,2}s_{2})+\frac{\lambda_{2}s_{1}S_{1}(\lambda_{1})S_{2}(s_{2})}{(\lambda_{1}-s_{1})(\lambda_{2}-s_{2})}f(\lambda_{1}a_{1,1}+s_{2}a_{2,1},\lambda_{1}a_{1,2}+s_{2}a_{2,2}).
    \end{array}
\end{displaymath}
Moreover,
\begin{displaymath}
    \begin{array}{l}
      \mathbb{E}(e^{-s_{1}[\sum_{j=1}^{2}a_{1,j}Z_{j}+S_{1}-A_{1}]^{-}-s_{2}[\sum_{j=1}^{2}a_{2,j}Z_{j}+S_{2}-A_{2}]^{-}})\vspace{2mm}\\=\mathbb{E}(\left(1+\frac{s_{1}}{\lambda_{1}-s_{1}}e^{-\lambda_{1}[\sum_{j=1}^{2}a_{1,j}Z_{j}+S_{1}]}\right)\left(1+\frac{s_{2}}{\lambda_{2}-s_{2}}e^{-\lambda_{2}[\sum_{j=1}^{2}a_{2,j}Z_{j}+S_{2}]}\right)) \vspace{2mm}\\
      =1+\frac{s_{2}S_{2}(\lambda_{2})}{\lambda_{2}-s_{2}}f(\lambda_{2}a_{2,1},\lambda_{2}a_{2,2})+\frac{s_{1}S_{1}(\lambda_{1})}{\lambda_{1}-s_{1}}f(\lambda_{1}a_{1,1},\lambda_{1}a_{1,2})+\frac{s_{1}s_{2}S_{1}(\lambda_{1})S_{2}(\lambda_{2})}{(\lambda_{1}-s_{1})(\lambda_{2}-s_{2})}f(\lambda_{1}a_{1,1}+\lambda_{2}a_{2,1},\lambda_{1}a_{1,2}+\lambda_{2}a_{2,2}).
    \end{array}
\end{displaymath}
Substituting the above terms in \eqref{fund} we come up with the following functional equation:
\begin{equation}
\begin{array}{rl}
    f(s_{1},s_{2})=&c_{1}(s_{1},s_{2})f(T(s_{1},s_{2}))-c_{2}(s_{1},s_{2})f(T(s_{1},\lambda_{2}))\vspace{2mm}\\
    &-c_{3}(s_{1},s_{2})f(T(\lambda_{1},s_{2}))+c_{4}(s_{1},s_{2})f(T(\lambda_{1},\lambda_{2})),
    \end{array}\label{bza}
\end{equation}
where $T(s_{1},s_{2}):=A^{T}\tilde{s}=(\sum_{j=1}^{2}a_{j,1}s_{j},\sum_{j=1}^{2}a_{j,2}s_{j})$, $\tilde{s}:=(s_{1},s_{2})^{T}$, and
\begin{equation}
    \begin{array}{rl}
         c_{1}(s_{1},s_{2})=\frac{\lambda_{1}\lambda_{2}S_{1}(s_{1})S_{2}(s_{2})}{(\lambda_{1}-s_{1})(\lambda_{2}-s_{2})},& c_{2}(s_{1},s_{2})=\frac{\lambda_{1}s_{2}S_{1}(s_{1})S_{2}(\lambda_{2})}{(\lambda_{1}-s_{1})(\lambda_{2}-s_{2})}, \vspace{2mm}\\
         c_{3}(s_{1},s_{2})=\frac{s_{1}\lambda_{2}S_{1}(\lambda_{1})S_{2}(s_{2})}{(\lambda_{1}-s_{1})(\lambda_{2}-s_{2})},& c_{4}(s_{1},s_{2})=\frac{s_{1}s_{2}S_{1}(\lambda_{1})S_{2}(\lambda_{2})}{(\lambda_{1}-s_{1})(\lambda_{2}-s_{2})}.
    \end{array}\label{coeffi}
\end{equation}

Our aim is to solve the functional equation \eqref{bza} iteratively. A crucial step is to verify that the mapping 
$T(s_{1},s_{2}) := A^{T}\tilde{s}$ is a contraction. Although we assume that $a_{i,j}\in[0,1)$, 
this alone is not sufficient to guarantee convergence. In order to ensure that the iterates 
$T^{(n)}(s_{1},s_{2})$ converge to the unique fixed point $(0,0)$, it is required that 
$T$ is a contraction mapping, i.e., there exists a matrix norm such that
$\|A^{T}\| < 1$. A sufficient and equivalent condition is that the spectral radius satisfies $\rho(A) < 1$ 
(since $\rho(A^{T})=\rho(A)$). In the special case where $A$ is diagonal, i.e., $a_{1,2}=a_{2,1}=0$, the mapping $T$ is 
always a contraction, since
\begin{displaymath}
    \|A^{T}\tilde{s}\| \le a^{*} \|\tilde{s}\|, \quad a^{*}:=\max\{a_{i,i},\, i=1,\ldots,N\} < 1.
\end{displaymath}

Thus, as a first step in solving \eqref{bza}, we assume in the following that $A=diag(a_{1},a_{2})$ (where for notational convenience we denote $a_{1}$ for $a_{1,1}$, and $a_{2}$ for $a_{2,2}$). Note that for $A=diag(a_{1},a_{2})$, \eqref{bza} is now written as:
\begin{equation}
\begin{array}{rl}
    f(s_{1},s_{2})=&c_{1}(s_{1},s_{2})f(a_{1}s_{1},a_{2}s_{2})-c_{2}(s_{1},s_{2})f(s_{1}a_{1},\lambda_{2}a_{2})-c_{3}(s_{1},s_{2})f(\lambda_{1}a_{1},s_{2}a_{2})+c_{4}(s_{1},s_{2})A_{1,1},
    \end{array}\label{bza1}
\end{equation}
where from hereon we denote $A_{i,j}:=f(a_{1}^{i}\lambda_{1},a_{2}^{j}\lambda_{2})$, $i,j=1,2,\ldots$. 

Then, iterating \eqref{bza1} $n-1$ times results in 
\begin{equation}
    \begin{array}{rl}
        f(s_{1},s_{2})=& C_{1,1}^{(n)}(s_{1},s_{2})A_{1,1}+\sum_{i=2}^{n}C_{i,1}^{(n)}(s_{1},s_{2})A_{i,1}+\sum_{j=2}^{n}C_{1,j}^{(n)}(s_{1},s_{2})A_{1,j}-\sum_{j=1}^{n}D_{j}^{(n)}(s_{1},s_{2})f(a_{1}^{n}s_{1},a_{2}^{j}\lambda_{2})\vspace{2mm}\\&-\sum_{i=1}^{n}K_{i}^{(n)}(s_{1},s_{2})f(a_{1}^{i}\lambda_{1},a_{2}^{n}s_{2})+\prod_{k=0}^{n-1}c_{1}(\sigma_{1,1}^{(k)}(s_{1},s_{2}))f(a_{1}^{n}s_{1},a_{2}^{n}s_{2}), 
    \end{array}\label{bza2}
\end{equation}
where, $\sigma_{i,j}^{(k)}(s_{1},s_{2})=\sigma_{i,j}^{(k-1)}(\sigma^{(1)}_{i,j}(s_{1},s_{2}))$ (i.e., $\sigma^{(l)}(.)$ is the $l$th composition of $\sigma(.)$ with itself), with
\begin{displaymath}
    \sigma^{(1)}_{i,j}(s_{1},s_{2}) =\left\{\begin{array}{rl}
       (a_{1}s_{1},a_{2}s_{2}), &i=j=1,  \\
       (a_{1}s_{1},a_{2}\lambda_{2}), &i=1,j=0,  \\
       (a_{1}\lambda_{1},a_{2}s_{2}), &i=0,j=1,
    \end{array}\right.
\end{displaymath}
and $\sigma^{(0)}_{1,1}(s_{1},s_{2}):=(s_{1},s_{2}) $ (Note that $\sigma^{(1)}_{1,1}(s_{1},s_{2}):=A^{T}\tilde{s}$). Moreover, $C_{1,1}^{(n)}(s_{1},s_{2})$, $n=1,2,\ldots$ are recursively computed with $C_{1,1}^{(1)}(s_{1},s_{2}):=c_{4}(s_{1},s_{2})$, and
\begin{displaymath}
\begin{array}{rl}
    C_{1,1}^{(n)}(s_{1},s_{2})=&C_{1,1}^{(n-1)}(s_{1},s_{2})+\prod_{k=0}^{n-2}c_{1}(\sigma^{(k)}_{1,1}(s_{1},s_{2}))c_{4}(\sigma^{(n-1)}_{1,1}(s_{1},s_{2}))-\sum_{j=1}^{n-1}D_{j}^{(n-1)}(s_{1},s_{2})c_{4}(a_{1}^{n-1}s_{1},a_{2}^{j}\lambda_{2})\vspace{2mm}\\
    &-\sum_{i=1}^{n-1}K_{i}^{(n-1)}(s_{1},s_{2})c_{4}(a_{1}^{i-1}\lambda_{1},a_{2}^{n-1}s_{2}),\,n=2,3,\ldots.
    \end{array}
\end{displaymath}
and,
\begin{displaymath}
    \begin{array}{rl}
        C_{1,j}^{(n)}(s_{1},s_{2})=&\sum_{k=j-1}^{n-1}D_{j-1}^{(k)}(s_{1},s_{2})c_{3}(a_{1}^{k}s_{1},a_{2}^{j-1}\lambda_{2}),\,j=2,\ldots,n,\vspace{2mm}  \\
         C_{i,1}^{(n)}(s_{1},s_{2})=&\sum_{k=i-1}^{n-1}K_{i-1}^{(k)}(s_{1},s_{2})c_{2}(a_{1}^{i-1}\lambda_{1},a_{2}^{k}s_{2}),\,i=2,\ldots,n,
    \end{array}
\end{displaymath}
Furthermore $D_{1}^{(1)}(s_{1},s_{2}):=c_{2}(s_{1},s_{2})$ and for $n=2,3,\ldots,$
\begin{displaymath}
    \begin{array}{rl}
        D_{1}^{(n)}(s_{1},s_{2}) =&\prod_{k=0}^{n-1}c_{1}(\sigma_{1,1}^{(k)}(s_{1},s_{2}))D_{1}^{(1)}(\sigma_{1,1}^{(n-1)}(s_{1},s_{2}))-\sum_{j=1}^{n-1} D_{j}^{(n-1)}(s_{1},s_{2}) D_{1}^{(1)}(a_{1}^{n-1}s_{1},a_{2}^{j}\lambda_{2}),\vspace{2mm}  \\
         D_{j}^{(n)}(s_{1},s_{2})=&D_{j-1}^{(n-1)}(s_{1},s_{2})c_{1}(a_{1}^{n-1}s_{1},a_{2}^{j-1}\lambda_{2}),\,j=2,\ldots,n. 
    \end{array}
\end{displaymath}
Similarly, $K_{1}^{(1)}(s_{1},s_{2}):=c_{3}(s_{1},s_{2})$ and for $n=2,3,\ldots,$
\begin{displaymath}
    \begin{array}{rl}
        K_{1}^{(n)}(s_{1},s_{2}) =&\prod_{k=0}^{n-1}c_{1}(\sigma_{1,1}^{(k)}(s_{1},s_{2}))K_{1}^{(1)}(\sigma_{1,1}^{(n-1)}(s_{1},s_{2}))-\sum_{i=1}^{n-1} K_{i}^{(n-1)}(s_{1},s_{2}) K_{1}^{(1)}(a_{1}^{i}\lambda_{1},a_{2}^{n-1}s_{2}),\vspace{2mm}  \\
         K_{i}^{(n)}(s_{1},s_{2})=&K_{i-1}^{(n-1)}(s_{1},s_{2})c_{1}(a_{1}^{i-1}\lambda_{1},a_{2}^{n-1}s_{2}),\,i=2,\ldots,n,
    \end{array}
\end{displaymath}
Note that each term appearing in the recursive definitions of 
$C_{i,j}^{(n)}(s_{1},s_{2})$, $D_j^{(n)}(s_{1},s_{2})$, and $K_i^{(n)}(s_{1},s_{2})$ 
is a finite product involving factors of the form 
$c_1(\cdot)$ and $c_\ell(\cdot)$, $\ell=2,3,4$, 
evaluated at iterated arguments of the form 
$(a_1^k s_1, a_2^k s_2)$ or mixed arguments where one component is fixed 
(e.g., $(a_1^k s_1, a_2^j \lambda_2)$).

The main result is summarized in the following theorem.
\begin{theorem}\label{thvar}
Consider the functional equation \eqref{bza1} and assume that the functions 
$c_k(s_1,s_2)$, $k=1,2,3,4$ defined in \eqref{coeffi} with $c_1(0,0)=1$, $c_2(0,0)=c_3(0,0)=c_4(0,0)=0$. Let $a_1,a_2 \in (0,1)$ and define $a^*:=\max\{a_1,a_2\}<1$. Then, for all $(s_1,s_2)$ with $Re(s_1)\ge 0$, $Re(s_2)\ge 0$, the solution of \eqref{bza1} is given by
\begin{equation}
    \begin{array}{rl}
        f(s_{1},s_{2})=& \tilde{C}_{1,1}(s_{1},s_{2})A_{1,1}+\sum_{i=2}^{\infty}\tilde{C}_{i,1}(s_{1},s_{2})A_{i,1}+\sum_{j=2}^{\infty}\tilde{C}_{1,j}(s_{1},s_{2})A_{1,j}\vspace{2mm}\\&-\sum_{j=1}^{\infty}\tilde{D}_{j}(s_{1},s_{2})f(0,a_{2}^{j}\lambda_{2})-\sum_{i=1}^{\infty}\tilde{K}_{i}(s_{1},s_{2})f(a_{1}^{i}\lambda_{1},0)+\prod_{k=0}^{\infty}c_{1}(\sigma_{1,1}^{(k)}(s_{1},s_{2})),
    \end{array}\label{bza3}
\end{equation}
where the coefficients $\tilde{C}_{i,j}(s_{1},s_{2})$, 
$\tilde{D}_{j}(s_{1},s_{2})$, and $\tilde{K}_{i}(s_{1},s_{2})$ are the limits of 
the recursively defined sequences $C_{i,j}^{(n)}(s_{1},s_{2})$, 
$D_{j}^{(n)}(s_{1},s_{2})$, and $K_{i}^{(n)}(s_{1},s_{2})$, respectively. Moreover, all series and the infinite product in \eqref{bza3} converge absolutely.
\end{theorem}
\begin{proof}
Iterate \eqref{bza1} $n-1$ times to obtain \eqref{bza2}, where 
all coefficients are defined recursively. Since $a_1,a_2 \in (0,1)$, we have $\sigma_{1,1}^{(k)}(s_1,s_2) \to (0,0)$ as $k\to\infty$. Since $c_1(s_1,s_2)$ is analytic close to $(0,0)$ and $c_1(0,0)=1$, there exists $C>0$ such that $|c_1(s_1,s_2)-1| \le C(|s_1|+|s_2|)$. Thus, for $(s_1,s_2)$ sufficiently close to $(0,0)$ we have
\begin{displaymath}
    |c_1(\sigma_{1,1}^{(k)}(s_1,s_2)) - 1|
\le C (a^*)^k (|s_1|+|s_2|).
\end{displaymath}
Since $\sum_{k=0}^\infty (a^*)^k < \infty$, it follows that
\begin{displaymath}
    \sum_{k=0}^{\infty} |c_1(\sigma_{1,1}^{(k)}(s_1,s_2)) - 1| < \infty,
\end{displaymath}
and therefore $\prod_{k=0}^{\infty} c_1(\sigma_{1,1}^{(k)}(s_1,s_2))$, converges absolutely. Having in mind that $f(0,0)=1$, we have $f(a_1^n s_1, a_2^n s_2) \to 1$. Thus, as $n\to\infty$
\begin{displaymath}
    \prod_{k=0}^{n-1} c_1(\sigma_{1,1}^{(k)}(s_1,s_2))f(a_1^n s_1, a_2^n s_2)\to \prod_{k=0}^{\infty} c_1(\sigma_{1,1}^{(k)}(s_1,s_2)).
\end{displaymath}
All recursively defined quantities $C_{i,j}^{(n)}(s_1,s_2)$, 
$D_j^{(n)}(s_1,s_2)$, and $K_i^{(n)}(s_1,s_2)$ 
are finite sums of products involving factors $c_1(\cdot)$ and $c_l(\cdot)$, $l=2,3,4$, evaluated at iterated 
arguments of the form 
$(a_1^k s_1, a_2^k s_2)$ or at mixed arguments where one component is fixed 
(e.g., $(a_1^k s_1, a_2^j \lambda_2)$). In particular, products of the form
\begin{displaymath}
    \prod_{m=0}^{k-1} c_1(\sigma_{1,1}^{(m)}(s_1,s_2)),
\end{displaymath}
converge to a finite limit, and therefore are uniformly bounded. 
The same holds for products of $c_1(\cdot)$ evaluated at the corresponding 
iterated arguments appearing in the recursion. 
Since $c_2(0,0)=c_3(0,0)=c_4(0,0)=0$ and the functions are continuous, 
there exists $C>0$ such that for $(s_1,s_2)$ sufficiently close to $(0,0)$,
\begin{displaymath}
|c_l(s_1,s_2)| \leq C(|s_1|+|s_2|),\, l=2,3,4.
\end{displaymath}
Hence,
\begin{displaymath}
|c_l(\sigma_{1,1}^{(k)}(s_1,s_2))|
\le C (a^*)^k (|s_1|+|s_2|),\,l=2,3,4.
\end{displaymath}
Hence, each term appearing in the recursive definitions is bounded by a geometric sequence of order $(a^*)^k$, and therefore all corresponding 
series converge absolutely. The fact that $c_2(0,0)=c_3(0,0)=c_4(0,0)=0$ ensures that the corresponding 
terms vanish at the fixed point $(0,0)$. This implies that these terms 
introduce additional decay along the iterations, which guarantees the 
absolute convergence of the series in \eqref{bza3}. In this sense, 
$c_1(s_1,s_2)$ governs the main recursive structure, while $c_2(s_1,s_2)$, $c_3(s_1,s_2)$, $c_4(s_1,s_2)$ act as 
vanishing perturbations (this is analogous to the case $\tilde{V}(0)=\tilde{0}$ in Theorem \ref{th1}).

Moreover, the recursive increments can be written as finite sums of terms involving such products evaluated at increasing indices. Each increment can be written as a finite sum of terms of the form
\begin{displaymath}
    \left(\prod_{m=0}^{k-1} c_1(\sigma_{1,1}^{(m)}(s_1,s_2))\right)
c_l(\sigma_{1,1}^{(k)}(s_1,s_2)),\, l=2,3,4,
\end{displaymath}
which are bounded by $C (a^*)^k$. Since $\sum_{k=0}^{\infty} (a^*)^k<\infty$, 
the increments are summable and therefore tend to zero. Therefore, the sequences 
$C_{i,j}^{(n)}(s_1,s_2)$, $D_j^{(n)}(s_1,s_2)$, and $K_i^{(n)}(s_1,s_2)$ are Cauchy and converge. Finally, since \eqref{bza2} holds for all $n$ and all terms on the right-hand side 
converge as $n\to\infty$, by taking the limit as $n\to\infty$ in \eqref{bza2} we obtain \eqref{bza3}.
\end{proof}

We finally have to find $A_{i,j}$, and $f(0,a_{2}^{j}\lambda_{2})$, $f(a_{1}^{i}\lambda_{1},0)$, $i,j=1,2,\ldots$. Note that as the number of iterations increases, i.e., $i\to\infty$, $A_{i,1}:=f(a_{1}^{i}\lambda_{1},a_{2}\lambda_{2})\to f(0,a_{2}\lambda_{2})$, and $j\to\infty$, $A_{1,j}:=f(a_{1}\lambda_{1},a_{2}^{j}\lambda_{2})\to f(a_{1}\lambda_{1},0)$ thus in practice, we need to obtain a few terms $A_{i,1}$, $A_{1,j}$, $i,j\geq 1$, and after that, all the terms converge to $f(0,a_{2}\lambda_{2})$, $f(a_{1}\lambda_{1},0)$.

We proceed first by obtaining $f(a_{1}^{i}\lambda_{1},0)$, $i=1,2,\ldots$. Set $s_{2}=0$ in \eqref{bza1} to obtain (note that $c_{2}(s_{1},0)=c_{4}(s_{1},0)=0$):
\begin{equation}
    f(s_{1},0)=c_{1}(s_{1},0)f(a_{1}s_{1},0)-c_{3}(s_{1},0)f(\lambda_{1}a_{1},0).
    \label{bz1}
\end{equation}
It is easily realized that \eqref{bz1} is similar to the stationary version of a reflected autoregressive process analyzed in \cite[Section 2.2]{box1}. So, by iterating \eqref{bz1} we have
\begin{equation}
    f(s_{1},0)=-f(\lambda_{1}a_{1},0)\sum_{j=0}^{\infty}c_{3}(a_{1}^{j}s_{1},0)\prod_{k=0}^{j-1}c_{1}(a_{1}^{k}s_{1},0)+\prod_{k=0}^{\infty}c_{1}(a_{1}^{k}s_{1},0).\label{q1a}
\end{equation}
Substituting $s_{1}=a_{1}\lambda_{1}$ in \eqref{q1a} we obtain after some algebra
\begin{equation}
    f(\lambda_{1}a_{1},0)=\frac{\prod_{k=0}^{\infty}\frac{S_{1}(a_{1}^{k+1}\lambda_{1})}{1-a_{1}^{k+1}}}{1+\sum_{n=0}^{\infty}\frac{a_{1}^{n+1}S_{1}(\lambda_{1})}{1-a_{1}^{n+1}}\prod_{k=0}^{n-1}\frac{S_{1}(a_{1}^{k+1}\lambda_{1})}{1-a_{1}^{k+1}}}.\label{s1}
\end{equation}
Now, in \eqref{s1} set $a_{1}\lambda_{1}$ instead of $\lambda_{1}$ to get:
\begin{equation*}
    f(\lambda_{1}a_{1}^{2},0)=\frac{\prod_{k=0}^{\infty}\frac{S_{1}(a_{1}^{k+2}\lambda_{1})}{1-a_{1}^{k+1}}}{1+\sum_{n=0}^{\infty}\frac{a_{1}^{n+1}S_{1}(a_{1}\lambda_{1})}{1-a_{1}^{n+1}}\prod_{k=0}^{n-1}\frac{S_{1}(a_{1}^{k+2}\lambda_{1})}{1-a_{1}^{k+1}}},
\end{equation*}
and continuing similarly,
\begin{equation}
    f(\lambda_{1}a_{1}^{i},0):=P_{i}=\frac{\prod_{k=0}^{\infty}\frac{S_{1}(a_{1}^{k+i}\lambda_{1})}{1-a_{1}^{k+1}}}{1+\sum_{n=0}^{\infty}\frac{a_{1}^{n+1}S_{1}(a_{1}^{i-1}\lambda_{1})}{1-a_{1}^{n+1}}\prod_{k=0}^{n-1}\frac{S_{1}(a_{1}^{k+i}\lambda_{1})}{1-a_{1}^{k+1}}},\,i=1,2,\ldots.\label{pi}
\end{equation}

By repeating the above procedure by letting $s_{1}=0$ in \eqref{bz1}, we can obtain
\begin{equation}
    f(0,\lambda_{2}a_{2})=\frac{\prod_{k=0}^{\infty}\frac{S_{2}(a_{2}^{k+1}\lambda_{2})}{1-a_{2}^{k+1}}}{1+\sum_{n=0}^{\infty}\frac{a_{2}^{n+1}S_{2}(\lambda_{2})}{1-a_{2}^{n+1}}\prod_{k=0}^{n-1}\frac{S_{2}(a_{2}^{k+1}\lambda_{2})}{1-a_{2}^{k+1}}},\label{s2}
\end{equation}
and by iterating \eqref{s2} 
\begin{equation}
    f(0,\lambda_{2}a_{2}^{j}):=Q_{j}=\frac{\prod_{k=0}^{\infty}\frac{S_{2}(a_{2}^{k+j}\lambda_{2})}{1-a_{2}^{k+1}}}{1+\sum_{n=0}^{\infty}\frac{a_{2}^{n+j}S_{2}(a_{2}^{j-1}\lambda_{2})}{1-a_{2}^{n+1}}\prod_{k=0}^{n-1}\frac{S_{2}(a_{2}^{k+j}\lambda_{2})}{1-a_{2}^{k+1}}},\,j=1,2,\ldots.\label{s3}
\end{equation}
\begin{remark}
    Note that by letting $i\to\infty$, the right-hand side of \eqref{pi} converges to 1, which coincides with $\lim_{i\to\infty}f(\lambda_{1}a_{1}^{i},0)=f(0,0)=1$. Similarly, by letting $j\to\infty$, the right-hand side of \eqref{s3} converges to 1, which coincides with $\lim_{j\to\infty}f(0,\lambda_{2}a_{2}^{j})=f(0,0)=1$. 
\end{remark}

Since $P_{i}$, $Q_{j}$, $i,j=1,2,\ldots,$ are known, we are ready to derive $A_{1,1}$, $A_{1,j}$ $j=2,\ldots$, $A_{i,1}$, $i=2,\ldots$ from \eqref{bza3}. In particular, by substituting $(s_{1},s_{2})=(\lambda_{1}a_{1},\lambda_{2}a_{2})$, $(s_{1},s_{2})=(\lambda_{1}a_{1},\lambda_{2}a_{2}^{n})$, $n=2,3,\ldots$, $(s_{1},s_{2})=(\lambda_{1}a_{1}^{m},\lambda_{2}a_{2})$, $m=2,3,\ldots$ in \eqref{bza3}, the unknown terms $A_{1,1}$, $A_{1,n}$ $n=2,\ldots$, $A_{m,1}$, $m=2,\ldots$ satisfy the following system of equations:
\begin{displaymath}
    \begin{array}{l}
        A_{1,1}(1-\tilde{C}_{1,1}(\lambda_{1}a_{1},\lambda_{2}a_{2}))-\sum_{i=2}^{\infty}\tilde{C}_{i,1}(\lambda_{1}a_{1},\lambda_{2}a_{2})A_{i,1}-\sum_{j=2}^{\infty}\tilde{C}_{1,j}(\lambda_{1}a_{1},\lambda_{2}a_{2})A_{1,j}\vspace{2mm}\\=-\sum_{j=1}^{\infty}\tilde{D}_{j}(\lambda_{1}a_{1},\lambda_{2}a_{2})Q_{j}-\sum_{i=1}^{\infty}\tilde{K}_{i}(\lambda_{1}a_{1},\lambda_{2}a_{2})P_{i}+\prod_{k=0}^{\infty}c_{1}(\lambda_{1}a_{1}^{k+1},\lambda_{2}a_{2}^{k+1}), \vspace{2mm}\\
         A_{1,n}(1-\tilde{C}_{1,n}(\lambda_{1}a_{1},\lambda_{2}a_{2}^{n}))-A_{1,1}\tilde{C}_{1,1}(\lambda_{1}a_{1},\lambda_{2}a_{2}^{n})-\sum_{i=2}^{\infty}\tilde{C}_{i,1}(\lambda_{1}a_{1},\lambda_{2}a_{2}^{n})A_{i,1}-\sum_{j=2,j\neq n}^{\infty}\tilde{C}_{1,j}(\lambda_{1}a_{1},\lambda_{2}a_{2}^{n})A_{1,j}\vspace{2mm}\\=-\sum_{j=1}^{\infty}\tilde{D}_{j}(\lambda_{1}a_{1},\lambda_{2}a_{2}^{n})Q_{j}-\sum_{i=1}^{\infty}\tilde{K}_{i}(\lambda_{1}a_{1},\lambda_{2}a_{2}^{n})P_{i}+\prod_{k=0}^{\infty}c_{1}(\lambda_{1}a_{1}^{k+1},\lambda_{2}a_{2}^{k+n}),\,n=2,3,\ldots,\vspace{2mm}\\
         A_{m,1}(1-\tilde{C}_{m,1}(\lambda_{1}a_{1}^{m},\lambda_{2}a_{2}))-A_{1,1}\tilde{C}_{1,1}(\lambda_{1}a_{1}^{m},\lambda_{2}a_{2})-\sum_{i=2,i\neq m}^{\infty}\tilde{C}_{i,1}(\lambda_{1}a_{1}^{m},\lambda_{2}a_{2})A_{i,1}-\sum_{j=2}^{\infty}\tilde{C}_{1,j}(\lambda_{1}a_{1}^{m},\lambda_{2}a_{2})A_{1,j}\vspace{2mm}\\=-\sum_{j=1}^{\infty}\tilde{D}_{j}(\lambda_{1}a_{1}^{m},\lambda_{2}a_{2})Q_{j}-\sum_{i=1}^{\infty}\tilde{K}_{i}(\lambda_{1}a_{1}^{m},\lambda_{2}a_{2})P_{i}+\prod_{k=0}^{\infty}c_{1}(\lambda_{1}a_{1}^{k+m},\lambda_{2}a_{2}^{k+1}),\,m=2,3,\ldots
    \end{array}
\end{displaymath}
Having obtained $A_{1,1}$, $A_{i,1}$, $A_{1,j}$, $P_{i}$, $Q_{j}$, $i,j=1,2,\ldots,$ the joint LST $f(s_{1},s_{2})$ is given by \eqref{bza3}. Note that in practice, since $a_{i}\in[0,1)$, $i=1,2,$ the series that appear above converge very rapidly, so we usually need a small finite number of iterations.  
\section{Conclusion and suggestions for future research}\label{conc}
In this work, we dealt with vector-valued recursions between random vectors that lead to functional equations of the form \eqref{opq}, \eqref{basic0}, \eqref{masip}, \eqref{bzaa}. The general theory to treat \eqref{opq}, \eqref{basic0} is presented in Section \ref{theory}. The theoretical results are applied in a series of queueing related examples in Section \ref{appl}. In Section \ref{multi}, we dealt with some multidimensional processes that result in \eqref{masip}, \eqref{bzaa}. Although, their form is different from that of \eqref{opq}, the solution machinery is similar. In particular, we focused on a modulated ASIP two queue tandem system and on a reflected VAR(1) process, where the autoregressive parameter is now a given $N\times N$ matrix.

An interesting topic for future research is to cope with the case where we have vector-valued functional equations in which the involved contraction mappings are not commutative. For the scalar case (i.e., non-modulated), there exist some available scarce results, see \cite{borst1993m}, \cite{tin}. It would be interesting to further investigate what can still be accomplished both for the scalar and the vector-valued case, when we are dealing with a noncommutative contraction mapping. Other options for future research refer to the case where $R_{n}(Y_{n})$ may take negative values, as well as to fully consider the case where the autoregressive parameter is a scalar or a random matrix. A first attempt was made in subsection \ref{var2}, where the autoregressive parameter was a scalar matrix, and where our focus was on the case where this matrix is diagonal, and specifically when $N=2$. A natural extension should be to investigate the case where this matrix is of arbitrary dimension, and not only diagonal. 

\section*{Acknowledgements} I. Dimitriou gratefully acknowledges the Empirikion Foundation, Athens, Greece (\href{https://www.empirikion.gr/}{www.empirikion.gr}) for the financial support of this work. This work is dedicated to the memory of Antonis Kontis.
\appendix
\section{Proof of Theorem \ref{th-masip-full}}\label{appen}
By iterating \eqref{masip} $n$ times, we obtain,
\begin{equation}
\tilde{Z}(s,t)
= \sum_{k=0}^{n-1}\sum_{\gamma\in\Gamma_k}
\mathcal{K}_\gamma(s,t)\,K(s_k,t_k)
+ \sum_{\gamma\in\Gamma_n}
\mathcal{K}_\gamma(s,t)\,\tilde{Z}(s_n,t_n).
\label{masip-iter-final}
\end{equation}
Due to the commutative contraction mappings $\alpha_i(u)$ with common fixed point $a$, we know that there exists $\kappa\in(0,1)$ such that for $u\in\mathbb{C}_+$, $|\alpha_i(u)-a|\le \kappa |u-a|$. By the contraction property and the fact that $\alpha_i(a)=a$, we obtain recursively
\[
|t_k-a| = |\alpha_{i_k}(t_{k-1}) - \alpha_{i_k}(a)|
\le \kappa |t_{k-1}-a|
\le \kappa^k |t_0-a|=\kappa^k |t-a|.
\]
Since $s_m$ is either $s_{m-1}$ or $t_{m-1}$, it follows that $(s_k,t_k)\to(a,a)$ as $k\to\infty$.

\medskip
\noindent
\textbf{Step 1: Lipschitz continuity of $\tilde{Z}(s,t)$.}

Let $(s,t),(u,v)\in\mathbb{C}_+\times\mathbb{C}_+$. Then,
\begin{align*}
|z_i(s,t)-z_i(u,v)|
&= \left| \mathbb{E}\Big( \big(e^{-sW_1-tW_2}-e^{-uW_1-vW_2}\big)1_{\{Y=i\}} \Big) \right|\le \mathbb{E}\Big( \big|e^{-sW_1-tW_2}-e^{-uW_1-vW_2}\big|\,1_{\{Y=i\}} \Big).
\end{align*}
Note that
\begin{align*}
e^{-sW_1-tW_2}-e^{-uW_1-vW_2}
&= \big(e^{-sW_1}-e^{-uW_1}\big)e^{-tW_2}
+ e^{-uW_1}\big(e^{-tW_2}-e^{-vW_2}\big).
\end{align*}
Hence,
\begin{align*}
|e^{-sW_1-tW_2}-e^{-uW_1-vW_2}|
&\le |e^{-sW_1}-e^{-uW_1}|\,|e^{-tW_2}|
+ |e^{-uW_1}|\,|e^{-tW_2}-e^{-vW_2}|.
\end{align*}
Since $Re(s)\ge 0$, $Re(u)\ge 0$, $Re(t)\ge 0$, $Re(v)\ge 0$, it is readily seen that $|e^{-sW_1}|\le 1$, $|e^{-uW_1}|\le 1$, $|e^{-tW_2}|\le 1$, $|e^{-vW_2}|\le 1$. Moreover, by the mean value theorem, we obtain
\[
|e^{-sW_1}-e^{-uW_1}| \le |s-u|\,W_1,
\qquad
|e^{-tW_2}-e^{-vW_2}| \le |t-v|\,W_2.
\]
Combining the above we have
\begin{align*}
|z_i(s,t)-z_i(u,v)|
&\le \mathbb{E}\big( |s-u|W_1\,1_{\{Y=i\}} \big)
+ \mathbb{E}\big( |t-v|W_2\,1_{\{Y=i\}} \big) \\
&= |s-u|\,\mathbb{E}(W_1 1_{\{Y=i\}})
+ |t-v|\,\mathbb{E}(W_2 1_{\{Y=i\}}).
\end{align*}
Define
\[
C_1 := \max_{1\le i\le N}\mathbb{E}(W_1 1_{\{Y=i\}}), 
\qquad
C_2 := \max_{1\le i\le N}\mathbb{E}(W_2 1_{\{Y=i\}}).
\]
Then,
\[
|z_i(s,t)-z_i(u,v)| \le C_1 |s-u| + C_2 |t-v|.
\]
Taking the maximum over $i=1,\ldots,N$, we obtain
\[
\|\tilde{Z}(s,t)-\tilde{Z}(u,v)\|_\infty
= \max_{1\le i\le N} |z_i(s,t)-z_i(u,v)|
\le C_1 |s-u| + C_2 |t-v|.
\]
Setting $C:=\max\{C_1,C_2\}$, we conclude that
\[
\|\tilde{Z}(s,t)-\tilde{Z}(u,v)\|_\infty
\le C\big(|s-u|+|t-v|\big),
\]
which shows that $\tilde{Z}(s,t)$ is Lipschitz continuous on $\mathbb{C}_+\times\mathbb{C}_+$.

\medskip
\noindent
\textbf{Step 2: Bound on $\mathcal{K}_\gamma(s,t)$.}

Since $\tilde{P}^{(i)}=e_i e_i^T$, right multiplication by $\tilde{P}^{(i)}$ retains only the $i$-th column of a matrix. Therefore, by induction on $k$, each product $\mathcal{K}_\gamma(s,t)$ has at most one nonzero column (namely the $i_k$-th column).
Let $M(a):=\max\{\|R(a,a)\|_1,\|T(a)\|_1\}$. By analyticity of $R(s,t)$ and $T(s)$, for every $\epsilon>0$ there exists $\delta>0$ such that for $(u,v)$ sufficiently close to $(a,a)$, we have $\|R(u,v)\|_1,\ \|T(u)\|_1 \le M(a)+\epsilon$.

Since $(s_m,t_m)\to(a,a)$, for sufficiently large $m$ there exists $k_0$ such that for all $m\ge k_0$, $(s_m,t_m)$ lies in a neighborhood of $(a,a)$. Hence,
\begin{displaymath}
    \|\mathcal{K}_\gamma(s,t)\|_1
\le C_0 (M(a)+\epsilon)^{k-k_0}
\le C (M(a)+\epsilon)^k,
\end{displaymath}
for some constant $C>0$ independent of $\gamma$.

\medskip
\noindent
\textbf{Step 3: Decomposition of the remainder.}

Let
\begin{displaymath}
    R_n(s,t):=\sum_{\gamma\in\Gamma_n}
\mathcal{K}_\gamma(s,t)\,\tilde{Z}(s_n,t_n),
\end{displaymath}
and write $\tilde{Z}(s_n,t_n)
= \tilde{Z}(a,a) + \big(\tilde{Z}(s_n,t_n)-\tilde{Z}(a,a)\big)$, so that 
\begin{displaymath}
    R_n(s,t)=S_n(s,t)+D_n(s,t),
\end{displaymath}
where
\begin{displaymath}
    \begin{array}{rl}
         S_n(s,t):=&\sum_{\gamma\in\Gamma_n}
\mathcal{K}_\gamma(s,t)\,\tilde{Z}(a,a),\\
D_n(s,t):=&\sum_{\gamma\in\Gamma_n}
\mathcal{K}_\gamma(s,t)\big(\tilde{Z}(s_n,t_n)-\tilde{Z}(a,a)\big).
    \end{array}
\end{displaymath}

\medskip
\noindent
\textbf{Step 4: Vanishing of $D_n(s,t)$.}

By exploiting the projection property and the bounds above, we obtain
\begin{displaymath}
    \|D_n(s,t)\|_\infty
\le \max_{\gamma\in\Gamma_n}
\|\mathcal{K}_\gamma(s,t)\|_1
\|\tilde{Z}(s_n,t_n)-\tilde{Z}(a,a)\|_\infty.
\end{displaymath}
Hence,
\begin{displaymath}
    \|D_n(s,t)\|_\infty
\le C (M(a)+\epsilon)^n \kappa^n.
\end{displaymath}
Since $M(a)\kappa<1$, there exists $\epsilon>0$ such that $(M(a)+\epsilon)\kappa<1$, and thus $D_n(s,t)\to 0$.

\medskip
\noindent
\textbf{Step 5: Convergence of $(S_n)$.}

Using the recursive structure of $\mathcal{K}_\gamma(s,t)$, we obtain (see also the corresponding step in the proof of Theorem \ref{th1})
\[
S_{n+1}-S_n
= \sum_{\gamma\in\Gamma_n}
\mathcal{K}_\gamma(s,t)\,
\big(A_\gamma(s_n,t_n)-A_\gamma(a,a)\big)\tilde{Z}(a,a),
\]
where 
\begin{displaymath}
    A_\gamma(s,t) =
\begin{cases}
R(s,t), & \text{if the last step of }\gamma\text{ is }R,\\
T(s), & \text{if the last step of }\gamma\text{ is }T.
\end{cases}
\end{displaymath}

Since $R(s,t)$ and $T(s)$ are analytic, they are continuously differentiable in a neighborhood of $(a,a)$. Therefore, there exists $L>0$ such that for $(u,v)$ sufficiently close to $(a,a)$,
\begin{displaymath}
    \|A_\gamma(u,v)-A_\gamma(a,a)\|_1
\le L(|u-a|+|v-a|),
\end{displaymath}
uniformly for $A_\gamma\in\{R,T\}$. Applying this bound with $(u,v)=(s_n,t_n)$ and using the contraction property, we obtain
\begin{displaymath}
    \|A_\gamma(s_n,t_n)-A_\gamma(a,a)\|_1
\le C \kappa^n,
\end{displaymath}
for some constant $C>0$. Then, by using the one-column structure of $\mathcal{K}_\gamma(s,t)$ that is inherited by the projection property, we have
\begin{displaymath}
    \|S_{n+1}(s,t)-S_n(s,t)\|_\infty
\le \max_{\gamma\in\Gamma_n}
\|\mathcal{K}_\gamma(s,t)\|_1
\cdot
\|A_\gamma(s_n,t_n)-A_\gamma(a,a)\|_1
\cdot
\|\tilde{Z}(a,a)\|_\infty.
\end{displaymath}
Thus,
\begin{displaymath}
    \|S_{n+1}(s,t)-S_n(s,t)\|_\infty
\le C (M(a)+\epsilon)^n \kappa^n.
\end{displaymath}

Since $(M(a)+\epsilon)\kappa<1$, it follows that
\[
\sum_{n=0}^\infty \|S_{n+1}-S_n\|_\infty < \infty,
\]
and therefore $(S_n)$ is a Cauchy sequence in $\mathbb{C}^N$, and therefore, it is convergent.

\medskip
\noindent
\textbf{Step 6: Absolute convergence of the series.}

Since $(s_k,t_k)\to(a,a)$ and $K$ is analytic, $\|K(s_k,t_k)\|_\infty$ is bounded. Using again the one-column structure due to the projection property, we obtain
\begin{displaymath}
    \left\|
\sum_{\gamma\in\Gamma_k}
\mathcal{K}_\gamma(s,t)\,K(s_k,t_k)
\right\|_\infty
\le C (M(a)+\epsilon)^k.
\end{displaymath}
Since $M(a)\kappa<1$, there exists $\epsilon>0$ such that $(M(a)+\epsilon)\kappa<1$, and the series converges absolutely.

Letting $n\to\infty$ in \eqref{masip-iter-final}, we finally obtain
\begin{displaymath}
\tilde{Z}(s,t)
= \sum_{k=0}^{\infty}\sum_{\gamma\in\Gamma_k}
\mathcal{K}_\gamma(s,t)\,K(s_k,t_k)
+ \lim_{n\to\infty}\sum_{\gamma\in\Gamma_n}
\mathcal{K}_\gamma(s,t)\,\tilde{Z}(s_n,t_n),
\end{displaymath}
which completes the proof.
\bibliographystyle{abbrv}

\bibliography{mybibfile}

 \end{document}